\documentclass[11pt,reqno]{amsart}
\usepackage{amsmath,amssymb,amsthm,mathtools}
\usepackage[margin=1in]{geometry}
\usepackage{microtype}
\usepackage[hidelinks,hypertexnames=false]{hyperref}
\usepackage{zref-clever}
\zcsetup{cap=true,nosort}
\usepackage{aliascnt}
\usepackage{enumitem}
\usepackage{mathrsfs}
\usepackage{tikz}
\usepackage{cite}
\usetikzlibrary{arrows.meta}

\allowdisplaybreaks
\numberwithin{equation}{section}

\theoremstyle{plain}
\newtheorem{theorem}{Theorem}[section]
\newaliascnt{proposition}{theorem}
\newtheorem{proposition}[proposition]{Proposition}
\aliascntresetthe{proposition}
\newaliascnt{lemma}{theorem}
\newtheorem{lemma}[lemma]{Lemma}
\aliascntresetthe{lemma}
\newaliascnt{corollary}{theorem}
\newtheorem{corollary}[corollary]{Corollary}
\aliascntresetthe{corollary}

\theoremstyle{definition}
\newaliascnt{definition}{theorem}
\newtheorem{definition}[definition]{Definition}
\aliascntresetthe{definition}

\theoremstyle{remark}
\newaliascnt{remark}{theorem}
\newtheorem{remark}[remark]{Remark}
\aliascntresetthe{remark}

\newcommand{\R}{\mathbb R}
\newcommand{\E}{\mathbb E}
\newcommand{\Pp}{\mathcal P}
\newcommand{\M}{\mathcal M}
\newcommand{\supp}{\operatorname{supp}}
\newcommand{\sech}{\operatorname{sech}}
\newcommand{\Var}{\operatorname{Var}}
\newcommand{\dd}{\,\mathrm d}
\newcommand{\cN}{\mathcal N}

\title[RSB at arbitrary depth]{Replica symmetry breaking at arbitrary depth}
\author{P.\ M.\ Aronow}
\author{Patrick Lopatto}
\date{}

\begin{document}

\begin{abstract}
We study finite-step replica symmetry breaking (RSB) in mixed $p$-spin Ising spin glasses at positive temperature and zero external field. For every integer $k\geq1$, we construct  polynomial mixtures whose Parisi measures are exactly $k$-RSB, with each such phase persisting on an open set of coefficients. More generally, for a broad class of  polynomial covariances with an $r$-RSB Parisi measure, adding the interaction term $\lambda q^p$ for sufficiently large even $p$ produces a unique $r$-RSB-to-$(r+1)$-RSB transition as $\lambda$ varies near a critical value. Since such covariances exist for every $r\geq0$, this yields transitions of arbitrary depth. The main idea of the proof is that, for large $p$, the perturbation $\lambda q^p$ is negligible for overlaps bounded away from $1$, while its effect is concentrated at overlaps very close to $1$.
\end{abstract}

\maketitle

\tableofcontents

\section{Introduction}

A central problem in mean-field spin glass theory is to determine how the equilibrium structure changes with temperature and other parameters of the model.  The Parisi formula expresses the limiting free energy as the minimum of a functional over probability measures, and its unique minimizer, called the Parisi measure, serves as the order parameter for the model.  In the physical interpretation, this measure describes the distribution of overlaps between spin configurations sampled independently from the Gibbs measure.  These overlaps quantify the similarity between configurations and reveal their hierarchical organization.  Replica symmetry corresponds to the absence of a nontrivial hierarchy, while replica symmetry breaking (RSB) describes the emergence of several levels of similarity, interpreted physically as configurations arranged into clusters and nested subclusters.

Within the variational framework, this hierarchy is reflected in the support of the Parisi measure.  A single support point corresponds to a replica-symmetric phase; two support points describe one-step replica symmetry breaking; larger finite supports describe finitely many nested levels; and interval support corresponds to a continuous hierarchy, called full replica symmetry breaking (FRSB).  We say that a Parisi measure is $k$-step replica symmetry breaking, or $k$-RSB, if it has exactly $k+1$ support points.

The hierarchical picture of replica symmetry breaking originates in
Parisi's solution of mean-field spin glasses
\cite{Parisi1979,Parisi1980Sequence,Parisi1983OrderParameter,MezardEtAl1984Nature}.
Within this framework, changes in temperature or interaction
coefficients can produce transitions between different overlap hierarchies. For instance,  
Gardner's analysis of the Ising pure $p$-spin model predicted that a
$1$-RSB phase becomes unstable at low temperatures and gives way to a
phase with a continuous hierarchy of overlaps \cite{Gardner1985}.
Spherical models provide further examples of transitions between
different RSB structures.  Following Crisanti and Sommers's analysis
of the pure spherical $p$-spin model \cite{CrisantiSommers1992},
Crisanti and Leuzzi predicted $1$-RSB, FRSB, and intermediate
$1$-FRSB phases in the spherical $2+p$ model
\cite{CrisantiLeuzzi2004,CrisantiLeuzzi2006}.
In the $1$-FRSB phase, the overlap function has both a continuous
part and a single discontinuous jump.
Subsequent work on spherical mixtures predicted $1$-RSB and
$2$-RSB phases, together with more complicated RSB structures
\cite{CrisantiLeuzzi2007,CrisantiLeuzziPaoluzzi2011}.
Motivated by these examples, we ask which finite RSB hierarchies
can occur and how varying the interaction
coefficients can produce transitions between them.

The rigorous theory of mean-field spin glasses is built on the Parisi variational formula.  Following Guerra's interpolation bound for the Parisi functional \cite{Guerra2003}, Talagrand proved the formula for the Sherrington--Kirkpatrick model and mixed even $p$-spin models \cite{Talagrand2006}, and Panchenko extended it to general mixtures, including odd-spin interactions \cite{Panchenko2014}. 
Talagrand's subsequent study of Parisi measures
\cite{Talagrand2006ParisiMeasures} established differentiability and
variational properties of the order parameter that underlie later work
on its support.
Auffinger and W.-K.\ Chen proved uniqueness of the Parisi minimizer
\cite{AuffingerChen2015unique} and established basic structural properties of
Parisi measures, including that the support contains zero when there is no 
external field and smoothness of the density on intervals contained in the
support \cite{AuffingerChen2015properties}.   Jagannath and Tobasco recast the Parisi functional through
stochastic dynamic programming and developed first-variation criteria,
also obtained independently by W.-K.\ Chen, that
are particularly effective for determining the support of the minimizer
\cite{JagannathTobasco2016,JagannathTobasco2017,Chen2017Variational}.  These variational
conditions are the starting point for the analysis below.

Several rigorous results describe particular forms of replica symmetry
breaking in Ising models.
For the zero-field Sherrington--Kirkpatrick model, Zhou proved
that, for inverse temperatures just above the critical value,
the Parisi measure has support equal to an interval starting at
zero, a smooth density below its right endpoint, and an atom
at that endpoint \cite{Zhou2026FRSB}.
A recent
preprint extends this description throughout the low-temperature phase
\cite{Lopatto2026SKFRSB}, and the analysis in that preprint has since been generalized to accommodate an external field \cite{Lopatto2026ExternalField}.  Further, Auffinger, W.-K.\ Chen, and Zeng proved that the
zero-temperature Parisi minimizer of every mixed $p$-spin Ising model has
infinitely many points in its support
\cite{AuffingerChenZeng2020}.  As a consequence, they showed that for every $k\geq1$, the
positive-temperature Parisi measure of the same model has at least $k+1$
support points at all sufficiently low temperatures.
For the Sherrington--Kirkpatrick model at zero temperature, a recent
preprint by H.-B.\ Chen proves that the Stieltjes measure associated
with the Parisi minimizer has full support on $[0,1)$ and a smooth
density \cite{Chen2026ZeroTempFRSB}.

For the Ising pure $p$-spin model, Talagrand established overlap-gap
properties associated with one-step symmetry breaking just below the critical
temperature for sufficiently large $p$ \cite{Talagrand2000LowTemp}.  More recently, Zhou proved the existence of a unique first
critical inverse temperature and of a nonempty $1$-RSB interval
\cite{Zhou2024Gardner}.  A sequel proves a unique second transition to an
FRSB phase whose support has the form
$\{0\}\cup[q_\beta,q'_\beta]$ \cite{Zhou2026GardnerII}.

In the setting of spherical models, Talagrand's multiscale construction \cite{Talagrand2000Multiple} produced systems with an arbitrarily large finite number of well-separated overlap regions, with new overlap scales appearing successively as the temperature decreases.  This 
result localizes the overlap to a union of disjoint intervals and gives
positive mass to the relevant intervals, but it does not show that the
overlap distribution is supported on exactly one point in each interval. Talagrand subsequently proved the spherical Parisi formula for mixed
even $p$-spin models \cite{Talagrand2006Spherical}, and W.-K.\ Chen
then proved it for general mixtures, including odd-spin interactions
\cite{Chen2013Spherical}.  

More recent work has developed a detailed variational theory of finite-step
RSB in the spherical setting.  W.-K.\ Chen and Sen proved exact $1$-RSB at
zero temperature for the pure spherical $p$-spin model with $p\geq 3$ 
\cite{ChenSen2017Spherical}, while Jagannath and Tobasco developed the
corresponding zero-temperature variational theory and related spherical RSB
to the coincidence set of an obstacle problem
\cite{JagannathTobasco2017LowTemp,JagannathTobasco2018Spherical}.
Auffinger and Zeng constructed exact $2$-RSB examples at zero temperature
\cite{AuffingerZeng2019}, and Zhou developed general criteria for finite-step
spherical Parisi measures and gave an explicit exact $3$-RSB example at
zero temperature \cite{Zhou2025SphericalParisi}. Auffinger and Zhou classified the zero-temperature spherical $p+s$
model throughout its parameter space, including exact transitions between
$1$-RSB and $2$-RSB phases as well as the appearance of $1$-FRSB and
$2$-FRSB phases \cite{AuffingerZhou2025}. They also showed that, in parameter regimes with a continuous zero-temperature
component, the number of RSB steps at positive temperature becomes arbitrarily large as
the temperature tends to zero. 

We are interested here in realizing a prescribed finite number of RSB steps.  To our knowledge, previous  rigorous constructions of
finitely supported Parisi order parameters in mean-field Gaussian spin glasses with
covariance determined solely by the overlap have not gone beyond $3$-RSB.
At positive temperature, prior rigorous exact finite-step results in this
class of models have not gone beyond one step; the spherical
$2$-RSB and $3$-RSB examples cited above are at zero temperature. 
Nor has prior work established transitions from $r$-RSB to
$(r+1)$-RSB for arbitrarily large finite numbers of RSB steps.  Our contribution is to
establish both phenomena.  

For every integer $k\geq1$, we construct mixed even
$p$-spin Ising models whose Parisi measures are exactly $k$-RSB, with the
resulting $k$-RSB phase persisting on an open set of coefficients.  More generally, starting from any even  polynomial covariance $\Xi$ with $\Xi''(0)<1$ and an $r$-RSB Parisi measure that satisfies a certain nondegeneracy condition (defined below), 
we show that adding $\lambda q^p$, for sufficiently large even $p$,
produces a unique transition
from $r$-RSB to $(r+1)$-RSB as $\lambda$ varies near a critical value.  Since the construction can be iterated,
such transitions occur for every finite number of RSB steps. 

The main idea is that, for large $p$, the perturbation
$\lambda q^p$ is negligible for overlaps bounded away from $1$, while its
effect is concentrated at overlaps very close to $1$.  The existing support points move only slightly, while above the
transition, a single additional support point is created near $1$.

\subsection{Main results}

We consider mixed even $p$-spin Ising models at zero external field.  Fix a finite set $\mathscr P\subset\{2,4,6,\ldots\}$ and positive coefficients
$(a_p)_{p\in\mathscr P}$ with $\sum_{p\in\mathscr P}a_p=1$.  At inverse
temperature $\beta$, let $H_N$ be the centered Gaussian Hamiltonian on
$\Sigma_N=\{-1,1\}^N$ with normalized covariance mixture 
\begin{equation*}
 \xi_0(q)=\sum_{p\in\mathscr P}a_pq^p,
 \qquad q\in[-1,1].
\end{equation*}
We have
\begin{equation*}
 \E H_N(\sigma^1)H_N(\sigma^2)
 =N\beta^2\xi_0(R_{12}),
 \qquad
 R_{12}=\frac1N\sum_{i=1}^N\sigma_i^1\sigma_i^2.
\end{equation*}
We write
\[
 \Xi=\beta^2\xi_0
\]
for the effective covariance.  Let $\M$ denote the set of probability measures on
$[0,1]$, and let $\Pp_\Psi$ denote the Parisi functional with effective
covariance $\Psi$; its explicit form is recalled in \zcref{sec:geometry}.
The Parisi measure is the unique minimizer of $\Pp_\Xi$.

We now state our two main theorems, which are proved in \zcref{s:mainresults}. 
\begin{theorem}\label{thm:main}
For every integer $k\geq1$, there exist even integers
\[
 4=p_0<p_1<\cdots<p_k
\]
and a nonempty open set $\mathcal O_k\subset(0,\infty)^{k+1}$ such that, for
every $\mathbf c=(c_0,\ldots,c_k)\in\mathcal O_k$, the mixed $p$-spin
Ising model at zero external field with covariance function
\[
 \Xi_{\mathbf c}(q)=\sum_{j=0}^k c_jq^{p_j}
\]
has a $k$-RSB Parisi measure.
\end{theorem}
Equivalently, writing $c_j=\beta^2a_{p_j}$ with
$\sum_j a_{p_j}=1$, the theorem gives a $k$-RSB phase that is open jointly
in inverse temperature and in the normalized mixture coefficients, with openness in the mixture coefficients understood in the relative
topology of the simplex $\{\mathbf a:a_{p_j}>0,\ \sum_j a_{p_j}=1\}$.

For the second theorem, we call $\Xi(q)=\sum_{p\in\mathscr P}c_pq^p$
\emph{admissible} if $\mathscr P\subset\{2,4,6,\ldots\}$ is finite and
nonempty,  $c_p>0$ for all $p \in \mathscr P$, and $\Xi''(0)<1$.
A finite-step Parisi measure is \emph{nondegenerate} if it satisfies
the conditions on the zeros and curvature of its first-variation
function in \zcref{def:nondegenerate}.

\begin{theorem}
\label{thm:gardner-main}
Let $r\geq0$, and let $\Xi$ be an admissible covariance whose Parisi
measure is $r$-RSB and nondegenerate.  Define
\[
 \mathfrak p(b)=\inf_{\mu\in\M}\Pp_{b^2\Xi}(\mu),
 \qquad
 s_\Xi=\mathfrak p(1)-\mathfrak p'(1).
\]
Then $s_\Xi>0$, and there exists $\eta \in(0,2s_\Xi)$ such that, for every
sufficiently large even integer $p$, there exists a unique critical
coefficient
\[
 \lambda_{\mathrm c,p}\in(2s_\Xi-\eta,2s_\Xi+\eta)
\]
for the family
\[
 \Xi_{p,\lambda}(q)=\Xi(q)+\lambda q^p,
 \qquad
 \lambda\in[2s_\Xi-\eta,2s_\Xi+\eta],
\]
with the following phase diagram:
\[
 \begin{cases}
 r\text{-RSB}, & \lambda\leq\lambda_{\mathrm c,p},\\
 (r+1)\text{-RSB}, & \lambda>\lambda_{\mathrm c,p}.
 \end{cases}
\]
Moreover, $\lambda_{\mathrm c,p}\rightarrow2s_\Xi$ as $p\to\infty$ through even integers.
\end{theorem}

\zcref{prop:iteration} provides covariances satisfying these hypotheses for
every $r\geq0$.  The limiting critical coefficient is twice the entropy per
spin of the unperturbed model (\zcref{rem:entropy}).  \zcref{thm:gardner}
describes the transition in more detail, and \zcref{cor:gardner-asymptotics}
gives asymptotics for the mass and location of the new atom.  In particular, the mass of the new atom grows linearly to first order in
$\lambda-\lambda_{\mathrm c,p}$ just above the transition.

\subsection{Ideas of the proof}\label{s:ideas}

The proof is organized around the first-variation characterization of the
Parisi measure.  Given a probability measure $\mu$ on $[0,1]$, let $f_{\Xi,\mu}$ denote
its first-variation function, introduced in \zcref{sec:geometry}.  In the
present even, zero-field setting, the characterization of Jagannath and
Tobasco and of W.-K.\ Chen reduces the problem to showing that
$f_{\Xi,\mu}(q)\leq0$ for all $q\in[0,1]$ and that $\mu$ is supported on
the zero set of $f_{\Xi,\mu}$
\cite{JagannathTobasco2017,Chen2017Variational}.  Thus the zeros of this function determine the
possible support points of the Parisi measure.  The transition occurs when one additional zero appears near $q=1$.  The three stages of this
process are illustrated in \zcref{fig:gardner-schematic}.

\begin{figure}[tbp]
\centering
\begin{tikzpicture}[x=1.25cm,y=0.9cm,>=Latex]
 \begin{scope}
  \draw[->] (0,0) -- (2.75,0) node[right] {$q$};
  \draw (2.65,0.07) -- (2.65,-0.07) node[pos=0,above] {\scriptsize $1$};
  \draw[->] (0,-1.15) -- (0,0.35) node[above] {$f(q)$};
  \node at (1.35,0.55) {\small $\lambda<\lambda_{\mathrm c,p}$};
  \draw plot[smooth] coordinates
   {(0,0) (0.35,-0.55) (0.75,0) (1.20,-0.65)
    (1.75,-0.42) (2.15,-0.22) (2.65,-0.78)};
  \fill (0,0) circle (1.3pt);
  \fill (0.75,0) circle (1.3pt);
 \end{scope}
 \begin{scope}[xshift=4.1cm]
  \draw[->] (0,0) -- (2.75,0) node[right] {$q$};
  \draw (2.65,0.07) -- (2.65,-0.07) node[pos=0,above] {\scriptsize $1$};
  \draw[->] (0,-1.15) -- (0,0.35) node[above] {$f(q)$};
  \node at (1.35,0.55) {\small $\lambda=\lambda_{\mathrm c,p}$};
  \draw plot[smooth] coordinates
   {(0,0) (0.35,-0.55) (0.75,0) (1.20,-0.65)
    (1.75,-0.42) (2.15,0) (2.65,-0.78)};
  \fill (0,0) circle (1.3pt);
  \fill (0.75,0) circle (1.3pt);
  \draw (2.15,0) circle (1.7pt);
 \end{scope}
 \begin{scope}[xshift=8.2cm]
  \draw[->] (0,0) -- (2.75,0) node[right] {$q$};
  \draw (2.65,0.07) -- (2.65,-0.07) node[pos=0,above] {\scriptsize $1$};
  \draw[->] (0,-1.15) -- (0,0.35) node[above] {$f(q)$};
  \node at (1.35,0.55) {\small $\lambda>\lambda_{\mathrm c,p}$};
  \draw plot[smooth] coordinates
   {(0,0) (0.35,-0.55) (0.75,0) (1.20,-0.65)
    (1.72,-0.38) (2.15,0) (2.65,-0.78)};
  \fill (0,0) circle (1.3pt);
  \fill (0.75,0) circle (1.3pt);
  \fill (2.15,0) circle (1.3pt);
 \end{scope}
\end{tikzpicture}
\caption{Schematic first-variation picture for the transition.  The
first-variation function must remain nonpositive.  Filled circles represent
zeros carrying positive Parisi mass.  Below the transition there are only
the pre-existing support points.  At the critical coefficient, a new zero,
represented by the open circle, appears near $q=1$ but carries no mass.
Above the transition this zero acquires positive mass and becomes an
additional support point.}
\label{fig:gardner-schematic}
\end{figure}
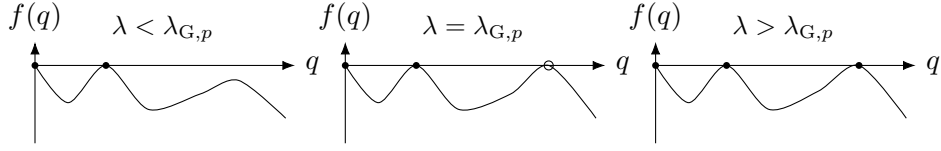

The proof makes this picture rigorous in three steps.

\emph{Step 1: Identify the limiting transition.}
Let $\mu_*$ denote the $r$-RSB Parisi measure of $\Xi$.  The basic reason that a high-degree perturbation can add one RSB step
without destroying the existing ones is that $q^p$ is extremely small when
$q$ is bounded away from $1$.  Thus, as $p\to\infty$, the perturbation
$\lambda q^p$ becomes invisible on the overlap scales containing
$\supp\mu_*$ and acts only very close to $q=1$.

This separation of overlap scales reduces the limiting problem to a single mass parameter.
We consider the candidate minimizer $\mu=m\nu+(1-m)\delta_Q$, where
$m\in(0,1]$, the point $Q$ is close to $1$, and $\nu$ is an atomic measure
close to $\mu_*$.  We call the atoms of $\nu$ the lower atoms.  Thus $m$
is the total Parisi mass of the lower atoms, and $1-m$ is the mass
available to the new point $Q$.  After optimizing
over $\nu$, the limiting variational objective becomes
\[
 J_\lambda(m)
 =\frac{\mathfrak p(m)}{m}+\frac{\lambda m}{2},
\]
where $\mathfrak p(m)$ is the minimum value of the Parisi functional for
the rescaled covariance $m^2\Xi$.  This rescaling comes from the effect of the new high-overlap region; as we will see, normalizing the
remaining mass from $m$ to $1$ amounts to replacing the inverse temperature
$\beta$ by $m\beta$.  For $m$ close to $1$, nondegeneracy ensures that the
rescaled model retains the same number of RSB steps.

The quantity
\[
 s_\Xi=\mathfrak p(1)-\mathfrak p'(1)
\]
is the limiting entropy per spin of the Gibbs measure for the unperturbed
model.  The derivative of $J_\lambda$ at the boundary point $m=1$ is
$\frac12(\lambda-2s_\Xi)$, which changes sign at $\lambda=2s_\Xi$.
For $\lambda$ near this value, the constrained minimizer over $m$ near
$1$ remains at $m=1$ below the transition, so the point $Q$ carries no
mass.  Above it, strict convexity produces a unique minimizer with $m<1$,
and the missing mass $1-m$ is placed at $Q$, which becomes a new atom.
This one-dimensional minimization identifies the candidate transition. The limiting objective and the positivity of $s_\Xi$
are established in \zcref{sec:limiting-problem}.

\emph{Step 2: Pass from the limiting regime to large finite $p$.}
The preceding reduction describes the singular limit $p=\infty$, so one
must show that this picture survives for large but finite interaction
degree.  For atomic measures, the Parisi PDE can be evaluated
through successive Gaussian Cole--Hopf recursions.  Uniform estimates for
these recursions allow us to reduce the finite-$p$ problem in stages, by first
locating the point $Q$, then optimizing the lower atoms, and finally choosing
the mass $m$.

We first solve the location equation $\Gamma(Q)=Q$, which is the
stationarity condition for a new atom of positive mass, where $\Gamma(q)$ is the average squared magnetization
at overlap $q$.  Thus
$1-Q$ must equal the magnetization deficit $1-\Gamma(Q)$.  An analysis of
the interval between a fixed cutoff and $Q$, using a comparison with an
explicit Gaussian model, gives
\[
 1-Q\asymp p^{-1/2}e^{-m^2\lambda p/2}.
\]
The lower atoms remain bounded away from $1$, while $Q$ lies
exponentially close to $1$.

To show that $Q$ is $C^2$ in $\lambda$, $m$, and the lower atom coordinates,
we express the deficit $1-\Gamma(Q)$ through derivatives of an
exponentially small remainder in a decomposition of the full Parisi functional. Estimates for
this remainder then allow us to solve for $Q$ as a $C^2$ function of the
remaining variables.  After substituting this location and optimizing the
lower atoms, we obtain a scalar function $J_p(\lambda,m)$ that is close to
$J_\lambda$ together with its first three derivatives.  Consequently, for
every sufficiently large even $p$, the derivative $\partial_mJ_p(\lambda,1)$
changes sign at a unique critical coefficient $\lambda_{\mathrm c,p}$ near
$2s_\Xi$, with an $r$-RSB candidate minimizer below the transition and an
$(r+1)$-RSB candidate minimizer above it.  The finite-$p$ reduction is developed in
\zcref{sec:high-degree-perturbation}. The supporting recursion analysis,
uniform perturbation estimates, and large-variance estimates are proved in
\zcref{app:recursion,app:uniform,app:large-variance}.

\emph{Step 3: Prove that the candidate minimizers are globally optimal.}
The preceding steps optimize the candidate minimizers within a
finite-dimensional family of atomic measures. This does not
by itself show that they minimize the Parisi functional over
all probability measures.  Writing $f$ for
the first-variation function of a candidate minimizer, we must show that $f(q)\leq0$ for
every overlap $q\in[0,1]$.

For this purpose, we divide $[0,1]$ into three regions: $[0,\bar q]$,
$[\bar q,\hat q_p]$, and $[\hat q_p,1]$, where $\bar q<1$ is a fixed cutoff
above the lower atoms and $1-\hat q_p$ is of order $(\log p)/p$.  On
$[0,\bar q]$, the high-degree perturbation is negligible, and the
first-variation function is close to $m^{-2}$ times that of the rescaled
model with covariance $m^2\Xi$.  Nondegeneracy and stationarity ensure
that its zeros there are exactly the lower atoms and that it is strictly
negative elsewhere on $[0,\bar q]$.  On $[\bar q,\hat q_p]$, we bound how much the first-variation function can increase.  This increase
is too small to overcome its negative value at the cutoff, so no new zero
can appear there.

The remaining region $[\hat q_p,1]$ lies very close to $q=1$ and contains
$Q$.  The Gaussian comparison from Step~2 shows that $1-\Gamma(q)$ is very
small there. A curvature bound then gives
$\Gamma'<1$, so $\Gamma(q)-q$ is strictly decreasing and vanishes only at
$Q$.  Since the derivative of the first-variation function has the same
sign as $\Gamma(q)-q$, the function increases up to $Q$ and decreases
after it.  Its maximum in this region is therefore $f(Q)$.

Increasing $m$ transfers mass from $Q$ to the lower atoms, where
$f=0$, and a first-variation analysis gives
$f(Q)=\partial_m J_p(\lambda,m)$.
At the mass minimizing $J_p(\lambda,\cdot)$ in Step~2, $f(Q)$ is
negative below the transition and zero at and above it.
Thus the first-variation function remains nonpositive and has no
zeros near $q=1$ other than the possible zero at $Q$.

This establishes the picture shown in
\zcref{fig:gardner-schematic}.  Below the transition, the zero set consists
of the lower atoms.  At
$\lambda=\lambda_{\mathrm c,p}$, an additional zero appears near $q=1$ but
carries no Parisi mass.  Above the transition, that zero becomes a support
point with positive mass, changing the Parisi measure from $r$-RSB to
$(r+1)$-RSB.  The global first-variation analysis and the resulting phase boundary
are established in \zcref{s:global,s:transition-conclusion}.

Finally, the construction can be repeated.  Beginning with an explicit
replica-symmetric model, we choose a coefficient $\lambda$ above the first transition to add
one support point, then add an interaction of still higher degree to create
the next one.  The separation of overlap scales keeps the previously
constructed support points stable while each new atom is added, and the
resulting Parisi measure is again nondegenerate.  Iterating this procedure
produces models with any prescribed finite number of RSB steps, and local
stability enlarges each constructed example to a nonempty open set of
coefficients.  This iteration is carried out in
\zcref{sec:consequences}.

As $p\to\infty$, the limiting free energy of the pure $p$-spin model
approaches that of Derrida's random energy model (REM)
\cite{Derrida1980REM,Derrida1981REM,GrossMezard1984}, whose freezing
threshold is $\beta^2=2\log2$.  Our limiting objective $J_\lambda$
generalizes the REM variational formula $\log2/m+\lambda m/2$
\cite[Section~2]{AuffingerChen2016Legendre}, with the entropy per spin of the
unperturbed model replacing $\log2$ in the critical coefficient.
The iteration parallels the successive freezing of the levels of Derrida's
generalized random energy model
\cite{Derrida1985GREM,BovierKurkova2004GREM}.

\subsection*{Acknowledgments}
P.L. was partially supported by NSF grant DMS-2450004. This paper was written by the authors with the assistance of large language models; this assistance included suggesting arguments, contributing to drafting and revision, and writing code for computational checks. 

\section{First variation and stability}\label{sec:geometry}

In this section, we introduce the first-variation function and give conditions under which a finite-step Parisi measure is stable under small
changes of the covariance.

Recall that an admissible covariance is a finite even polynomial
$\Xi(q)=\sum_{p\in\mathscr P}c_pq^p$ with positive coefficients and
$\Xi''(0)<1$.  Let $\M$ be the set of probability measures on $[0,1]$, and write
$\alpha_\mu(q)=\mu([0,q])$.  The Parisi solution satisfies
\begin{equation}\label{eq:parisi-pde}
 \begin{aligned}
  \partial_q u_\mu(q,x)
  &=-\frac{\Xi''(q)}2
  \left(u_{\mu,xx}(q,x)
  +\alpha_\mu(q)u_{\mu,x}(q,x)^2\right),\\
  u_\mu(1,x)&=\log\cosh x.
 \end{aligned}
\end{equation}
For general $\mu\in\M$, we interpret \eqref{eq:parisi-pde} in the weak
sense; it has a unique weak solution $u_\mu$, whose spatial derivatives of
every order are bounded and continuous
\cite[Theorem~4]{JagannathTobasco2016}.

The corresponding functional is
\begin{equation}\label{eq:parisi-functional}
 \Pp_\Xi(\mu)=\log2+u_\mu(0,0)
 -\frac12\int_0^1q\Xi''(q)\alpha_\mu(q)\dd q.
\end{equation}
Its minimum is the limiting free energy
\cite{Talagrand2006,Panchenko2014}, and its minimizer is unique
\cite{AuffingerChen2015unique}.

We call a measure atomic if it is finitely supported.  For such a measure,
the PDE is evaluated by successive Gaussian Cole--Hopf steps.  With $Z$
standard Gaussian and $\mathsf H_Vg(x)=\E g(x+\sqrt VZ)$, set
\begin{equation}\label{eq:T-def}
 T_{a,V}g=a^{-1}\log\mathsf H_V\bigl(e^{ag}\bigr)\quad(a>0),
 \qquad
 T_{0,V}g=\mathsf H_Vg.
\end{equation}
Since $\E\cosh(x+\sqrt VZ)=e^{V/2}\cosh x$, we have
\begin{equation}\label{eq:T-logcosh}
 T_{1,V}\log\cosh=\log\cosh+\frac V2.
\end{equation}

If $\alpha_\mu=a$ on $[q_-,q_+)$, then
\begin{equation}\label{eq:cole-hopf-step}
 u_\mu(q,\cdot)=T_{a,\Xi'(q_+)-\Xi'(q)}u_\mu(q_+,\cdot),
 \qquad q_-\leq q\leq q_+.
\end{equation}

This is Parisi's finite-step recursion
\cite{Parisi1980Sequence}; see also \cite{AuffingerChen2015properties} and
\cite[Lemma~2.4]{Lopatto2026SKFRSB}.  The function defined by
\eqref{eq:cole-hopf-step} is continuous on $[0,1]\times\R$, solves
\eqref{eq:parisi-pde} classically between consecutive atoms, and has a
bounded spatial derivative (\zcref{lem:curvature}).  Integrating by parts
on each interval between consecutive atoms, we see that it is the weak
solution.  General recursion and continuity results are proved in
\zcref{app:recursion}.

The optimal diffusion of an atomic measure \cite{AuffingerChen2015unique,JagannathTobasco2016} is
\[
 \dd X_q=\Xi''(q)\alpha_\mu(q)u_{\mu,x}(q,X_q)\dd q
 +\sqrt{\Xi''(q)}\dd B_q,\qquad X_0=0.
\]
This equation has a unique strong solution, since $u_{\mu,x}$ is bounded
and Lipschitz in $x$ (\zcref{lem:curvature}).  Define the mean squared magnetization and the first-variation function by
\[
 \Gamma_{\Xi,\mu}(q)=\E u_{\mu,x}(q,X_q)^2,
 \qquad
 f_{\Xi,\mu}(q)=\frac12\int_0^q\Xi''(s)
       (\Gamma_{\Xi,\mu}(s)-s)\dd s.
\]
We call $u_{\mu,x}(q,X_q)$ the magnetization process, so that
$\Gamma_{\Xi,\mu}$ is its mean square.  We also write
$\Gamma_{2,\Xi,\mu}(q)=\E u_{\mu,xx}(q,X_q)^2$ and often suppress the
covariance or measure from the notation.

\subsection{The first-variation criterion}\label{s:first-variation}

The characterization of Jagannath and Tobasco and of W.-K.\ Chen
\cite{JagannathTobasco2017,Chen2017Variational} takes the following form.
We use Zhou's sign convention, with $f(0)=0$
\cite[Theorem~3]{Zhou2024Gardner}.  The atom-derivative formulas are due to
Talagrand \cite{Talagrand2006ParisiMeasures}.

\begin{proposition}\label{prop:first-variation-criterion}
Let $\Xi$ be admissible, let $\mu,\nu\in\M$ be atomic, and write
$f=f_{\Xi,\mu}$.  Then
\begin{equation}\label{eq:directional-derivative}
 \left.\frac{\dd}{\dd t}\Pp_\Xi((1-t)\mu+t\nu)\right|_{t=0^+}
 =-\int_0^1 f(q)\dd(\nu-\mu)(q).
\end{equation}
In particular, for $Q\in[0,1]$, moving mass $1-m$ from $\mu$ to $Q$
gives, from the left at $m=1$,
\begin{equation}\label{eq:new-point-derivative}
 \left.\frac{\dd}{\dd m}\Pp_\Xi\bigl(m\mu+(1-m)\delta_Q\bigr)
 \right|_{m=1^-}
 =f(Q)-\int_0^1f\dd\mu.
\end{equation}
If $\mu=\sum_{i=0}^rw_i\delta_{q_i}$ with $w_i>0$ and $q_0<\cdots<q_r$, then
moving the location $q_i\in(0,1)$ with all masses and other locations
fixed, or moving mass from $q_j$ to $q_i$, gives, respectively,
\begin{equation}\label{eq:location-derivative}
 \partial_{q_i}\Pp_\Xi(\mu)=-w_i f'(q_i),
 \qquad
 \left.\frac{\dd}{\dd s}\Pp_\Xi\bigl(\mu+s(\delta_{q_i}-\delta_{q_j})\bigr)
 \right|_{s=0}=f(q_j)-f(q_i),
\end{equation}
where the $s$-derivative is two-sided.  Finally, if
\begin{equation}\label{eq:first-variation-criterion}
 f(q)\leq0\quad(0\leq q\leq1),
 \qquad
 \mu\bigl(\{q:f(q)=0\}\bigr)=1,
\end{equation}
then $\mu$ is the unique Parisi measure.
\end{proposition}

We prove \zcref{prop:first-variation-criterion} in
\zcref{s:variation-proofs}.  Its final assertion is how we identify Parisi
measures throughout the paper.  In particular, finding stationary atoms
is only part of the problem; one must also exclude positive values of $f$
at every other overlap.

\subsection{Curvature bounds}\label{s:curvature}

The derivative and martingale identities below are standard
\cite{AuffingerChen2015properties,AuffingerChen2015unique,JagannathTobasco2016}.
The bound $u_{xx}\leq1-u_x^2$ will also bound $\Gamma'$ near the new atom
(\zcref{prop:high-overlap-geometry}).

\begin{lemma}\label{lem:curvature}
Let $\Xi$ be admissible and $\mu$ atomic, let $X$ be the optimal diffusion
of $\mu$, and set $U_q=u_{\mu,x}(q,X_q)$.  For all
$(q,x)\in[0,1]\times\R$,
\begin{equation}\label{eq:curvature}
 |u_{\mu,x}|\leq1,
 \qquad
 0<u_{\mu,xx}\leq1-u_{\mu,x}^2.
\end{equation}
The process $U$ is a martingale with $U_0=0$, $U_1=\tanh X_1$, and
$\dd U_q=\sqrt{\Xi''(q)}u_{\mu,xx}(q,X_q)\dd B_q$.  Moreover,
$\Gamma=\Gamma_{\Xi,\mu}$ belongs to $C^1([0,1])$ and satisfies $\Gamma(0)=0$,
\begin{equation}\label{eq:Gamma-derivative}
 \Gamma'(q)
 =\Xi''(q)\E u_{\mu,xx}(q,X_q)^2=\Xi''(q)\Gamma_2(q),
\end{equation}
and
\begin{equation}\label{eq:Gamma-curvature}
 0\leq\Gamma'(q)\leq\Xi''(q)\bigl(1-\Gamma(q)\bigr),
 \qquad
 \Gamma(q)\leq\Xi'(q),
 \qquad 0\leq q\leq1.
\end{equation}
\end{lemma}

We prove \zcref{lem:curvature} in \zcref{s:curvature-proof}, by
induction over the Cole--Hopf steps, using It\^o's formula.

In particular, $f$ is $C^2$, with
\begin{equation}\label{eq:f-derivatives}
 f'=\frac12\Xi''\cdot(\Gamma-q),
 \qquad
 f''=\frac12\Xi'''\cdot(\Gamma-q)+\frac12\Xi''\cdot(\Xi''\Gamma_2-1).
\end{equation}
Because $\Xi''(0)<1$, the function $f$ has no zeros just to the right of
the origin. On a sufficiently short interval, $\Gamma(q)\leq\Xi'(q)<q$, so
$f$ is strictly negative to the right of $0$. Compare
\cite[Theorem~1(ii)]{AuffingerChen2015properties}.

\subsection{Nondegeneracy and the Hessian in atom coordinates}\label{s:nondegenerate}

Our construction requires that, after a small change of the covariance,
such as the rescaling $\Xi\mapsto m^2\Xi$ in \zcref{sec:limiting-problem},
the Parisi measure still has the same number of atoms and depends smoothly
on the covariance.  The
nondegeneracy condition below guarantees this (\zcref{prop:stability}).
For atomic measures with an atom at the origin, write
\begin{equation}\label{eq:atomic-measure}
 \nu_z=\sum_{i=0}^rw_i\delta_{q_i},
 \qquad
 0=q_0<q_1<\cdots<q_r<1,
 \qquad
 w_i>0,
 \qquad
 \sum_{i=0}^rw_i=1.
\end{equation}

We use the coordinates
\[
 z=(q_1,\ldots,q_r,w_0,\ldots,w_{r-1}),
 \qquad w_r=1-\sum_{i=0}^{r-1}w_i,
\]
and denote their open domain by $\mathcal Z_r\subset\R^{2r}$.  For $r=0$,
the coordinate vector is empty and $\nu_z=\delta_0$.  An atomic Parisi
measure at zero external field has an atom at $0$, because its support
contains the origin \cite[Theorem~1(i)]{AuffingerChen2015properties}.  Moreover, every
atomic Parisi measure $\mu$ satisfies $f_{\Xi,\mu}\leq0$ on $[0,1]$, with
equality on $\supp\mu$, and has no atom at $1$.  Indeed, write
$f=f_{\Xi,\mu}$.  Since $\mu$ is a minimizer,
\eqref{eq:directional-derivative} with $\nu=\delta_q$ gives
$f(q)\leq\int_0^1f\dd\mu$ for every $q\in[0,1]$.  Since $f$ is continuous,
we have $f=\int_0^1f\dd\mu$ on $\supp\mu$, and since $0\in\supp\mu$ and
$f(0)=0$, this constant is $0$.  If $1\in\supp\mu$, then $f$ attains its
maximum $0$ at $q=1$, so $f'(1)\geq0$.  However, since
$u_\mu(1,\cdot)=\log\cosh$, \eqref{eq:f-derivatives} gives
$f'(1)=\frac12\Xi''(1)(\E\tanh^2X_1-1)<0$.  Hence every atomic Parisi
measure has the form \eqref{eq:atomic-measure}.

\begin{definition}\label{def:nondegenerate}
Suppose that $\nu_z$ in \eqref{eq:atomic-measure} is the Parisi measure of
$\Xi$.  We call the $r$-RSB Parisi measure $\nu_z$ \emph{nondegenerate} if:
\begin{enumerate}[label=\textup{(\roman*)}]
 \item $f_{\Xi,\nu_z}\leq0$ on $[0,1]$, with equality exactly at
 $q_0,\ldots,q_r$;
 \item $f_{\Xi,\nu_z}''(q_i)<0$ for $1\leq i\leq r$.
\end{enumerate}
For $r=0$, condition~(ii) is vacuous.
\end{definition}

Strict negative curvature ensures that the zeros in $(0,1)$ remain
strict local maxima under small perturbations, and strict
negativity away from the zero set excludes the creation of additional zeros.  At the origin,
the required one-sided sign follows from $\Xi''(0)<1$, as shown in
\zcref{s:curvature}.  Since $q_r<1$, condition~(i) implies that
$f_{\Xi,\nu_z}(1)<0$;
this strict inequality yields the constant $\gamma>0$ of (S5) in
\zcref{sec:high-degree-perturbation}.  A zero $q_i\in(0,1)$ of $f$ is an interior maximum of $f\leq0$, so $f'(q_i)=0$ and hence
$\Gamma(q_i)=q_i$.  By
\eqref{eq:f-derivatives}, we then have
\[
 f''(q_i)=\frac12\Xi''(q_i)\bigl(\Xi''(q_i)\Gamma_2(q_i)-1\bigr),
\]
and condition~(ii) reads
\begin{equation}\label{eq:strict-replicon}
 \Xi''(q_i)\E u_{\nu_z,xx}(q_i,X_{q_i})^2<1,
 \qquad 1\leq i\leq r.
\end{equation}
Since the nonstrict inequality holds for every Parisi measure, condition~(ii)
is a strict replicon condition, generalizing the de~Almeida--Thouless
stability inequality \cite{deAlmeidaThouless1978,Toninelli2002,JagannathTobasco2017,Chen2017Variational}.  By
\zcref{lem:atomic-hessian}, it also makes the Hessian with respect to $z$ positive
definite, which allows the stationary point to be continued under
perturbations of the covariance by the implicit function theorem
(\zcref{prop:stability}).

The formulas of \zcref{prop:first-variation-criterion} also describe
stationarity with respect to $z$ for the measures
$m\nu_z+(1-m)\delta_Q$, in which the atoms of $\nu_z$ have total mass $m\leq1$.

\begin{corollary}\label{cor:atomic-gradient}
Let $\Xi$ be admissible, $z\in\mathcal Z_r$, $0<m\leq1$, and $q_r<Q<1$,
and let $f$ be the first-variation function of
$\mu=m\nu_z+(1-m)\delta_Q$.  Then, with $m$ and $Q$ fixed,
\[
 \partial_{q_i}\Pp_\Xi(\mu)=-mw_if'(q_i)\quad(1\leq i\leq r),
 \qquad
 \partial_{w_i}\Pp_\Xi(\mu)=m\bigl(f(q_r)-f(q_i)\bigr)\quad(0\leq i<r).
\]
In particular, the $z$-gradient of $\Pp_\Xi(\mu)$ vanishes if and only if
$f=0$ on $\supp\nu_z$ and $f'=0$ on $\supp\nu_z\setminus\{0\}$.
\end{corollary}

\begin{proof}
For $r=0$, the coordinate vector $z$ is empty, and the assertion follows
from $f(0)=0$.  Assume henceforth that $r\geq1$.
Since the mass of $\mu$ at $q_i$ is $mw_i$, the first formula is the first
identity in \eqref{eq:location-derivative}.  Increasing $w_i$ ($i<r$) with
the other coordinates fixed moves mass from $q_r$ to $q_i$ at rate $m$, and
the second formula follows from the second identity in
\eqref{eq:location-derivative}.  Since $q_0=0$ and $f(0)=0$, the vanishing of
$\partial_{w_0}$ means $f(q_r)=f(0)=0$, and then the vanishing of the other
$w$-derivatives means $f(q_i)=0$ for all $i$. The $q$-derivatives vanish
exactly when $f'(q_i)=0$ for $i\geq1$.
\end{proof}

For a fixed finite exponent set $\mathscr P$, let
$\mathcal C_{\mathscr P}$ be the open set of positive coefficient vectors
for which $\Xi_{\mathbf c}(q)=\sum_{p\in\mathscr P}c_pq^p$ is admissible.
Finite Gaussian recursions give the following smoothness result, which
we prove in \zcref{s:smooth-dependence}.

\begin{proposition}
\label{prop:finite-step-regularity}
For every finite $\mathscr P\subset\{2,4,6,\ldots\}$ and every $r\geq0$,
the map $(\mathbf c,z)\mapsto\Pp_{\Xi_{\mathbf c}}(\nu_z)$ is $C^\infty$ on
$\mathcal C_{\mathscr P}\times\mathcal Z_r$.
\end{proposition}

The next lemma adapts the strict-convexity argument of
\cite[Lemma~21]{JagannathTobasco2016}.  In directions that move the atoms,
positivity of the Hessian comes from the strict curvature $f''(q_i)<0$ at
the positive zeros.  In directions that change only the masses, it comes from
the strict convexity of $\log\cosh$, through a sharpened first-variation
inequality whose remainder is the error in the first-order Taylor
expansion of $\log\cosh$ at $X_1$ (\zcref{lem:sharp-first-variation}).  We prove
the lemma in \zcref{s:variation-proofs}.

\begin{lemma}\label{lem:atomic-hessian}
Let $\Xi$ be admissible, and let $z\in\mathcal Z_r$ be a critical point of
$z'\mapsto\Pp_\Xi(\nu_{z'})$.  If $f_{\Xi,\nu_z}''(q_i)<0$ for
$1\leq i\leq r$, then the Hessian $D_z^2\Pp_\Xi(\nu_z)$, which exists by
\zcref{prop:finite-step-regularity}, is positive definite.
\end{lemma}

\subsection{Stability and a starting model}

The next lemma shows that the conditions of \zcref{def:nondegenerate}
continue to hold after small changes of the covariance and of the atoms, as long
as the atoms remain stationary.  It also records the continuity of the
first-variation function, which we use again in
\zcref{sec:high-degree-perturbation}.  We prove the lemma in \zcref{s:truncated},
using the continuity of the Parisi recursion in its data.

\begin{lemma}\label{lem:persistence}
Let $\Xi_n$, $n\in\mathbb N\cup\{\infty\}$, be admissible covariances with
$\Xi_n\to\Xi_\infty$ in $C^3([0,1])$, let $z_n\in\mathcal Z_r$ with
$z_n\to z_\infty\in\mathcal Z_r$, and write $f_n=f_{\Xi_n,\nu_{z_n}}$.
Then $f_n\to f_\infty$ in $C^2([0,1])$.  Suppose further that
$\nu_{z_\infty}$ is the nondegenerate $r$-RSB Parisi measure of
$\Xi_\infty$, and that $z_n$ is a critical point of
$z\mapsto\Pp_{\Xi_n}(\nu_z)$ for every $n\in\mathbb N$.  Then, for all
sufficiently large $n$, $\nu_{z_n}$ is the nondegenerate $r$-RSB Parisi measure of
$\Xi_n$.
\end{lemma}

\begin{proposition}\label{prop:stability}
Let $\mathbf c_0\in\mathcal C_{\mathscr P}$, and suppose that the Parisi
measure $\nu_{z_*}$ of $\Xi_{\mathbf c_0}$ is $r$-RSB and nondegenerate.
Then there exist a neighborhood $\mathcal V\subset\mathcal C_{\mathscr P}$ of
$\mathbf c_0$ and a smooth map $z\colon\mathcal V\to\mathcal Z_r$ with
$z(\mathbf c_0)=z_*$ such that, for every $\mathbf c\in\mathcal V$,
$\nu_{z(\mathbf c)}$ is the nondegenerate $r$-RSB Parisi measure of
$\Xi_{\mathbf c}$.
\end{proposition}

\begin{proof}
For $r\geq1$, apply the implicit function theorem to
$\nabla_z\Pp_{\Xi_{\mathbf c}}(\nu_z)=0$ at $(\mathbf c_0,z_*)$.
The functional is smooth and its $z$-Hessian is positive definite by
\zcref{prop:finite-step-regularity,lem:atomic-hessian}, so there exists
a smooth map $z(\mathbf c)$ of stationary points.  After shrinking the
coefficient neighborhood, the atom locations remain strictly ordered
in $[0,1)$ and all masses remain positive.
Write $\nu_{\mathbf c}=\nu_{z(\mathbf c)}$ for this family.
For $r=0$, set $\nu_{\mathbf c}=\delta_0$.

It remains to show that, after shrinking $\mathcal V$, the measure
$\nu_{\mathbf c}$ is the nondegenerate $r$-RSB Parisi measure of
$\Xi_{\mathbf c}$ for every $\mathbf c\in\mathcal V$.  Otherwise, there
exist $\mathbf c_n\to\mathbf c_0$ at which this fails.  Since
$\Xi_{\mathbf c_n}\to\Xi_{\mathbf c_0}$ in $C^3([0,1])$,
$z(\mathbf c_n)\to z_*$, and each $z(\mathbf c_n)$ is a critical point of
$z\mapsto\Pp_{\Xi_{\mathbf c_n}}(\nu_z)$, this contradicts
\zcref{lem:persistence}.
\end{proof}

We close this section with an application to a concrete model, which will later form the basis of our induction.

\begin{lemma}\label{lem:replica-symmetric-seed}
The covariance $\Xi(q)=q^4/16$ has the nondegenerate
replica-symmetric Parisi measure $\delta_0$.
\end{lemma}

\begin{proof}
Write
$\Gamma=\Gamma_{\Xi,\delta_0}$ and $f=f_{\Xi,\delta_0}$.  By
\eqref{eq:Gamma-curvature}, we have
$\Gamma(q)\leq{\Xi}'(q)=\frac14q^3<q$ for $0<q\leq1$.  Then
$f'=\frac12{\Xi}''\cdot(\Gamma-q)<0$ on $(0,1]$, and
$f<0$ on $(0,1]$ since $f(0)=0$.  By \zcref{prop:first-variation-criterion},
$\delta_0$ is the Parisi measure of $\Xi$.  It is nondegenerate because
condition~(i) of \zcref{def:nondegenerate} was just shown, and
condition~(ii) is vacuous for $r=0$.
\end{proof}

\section{The limiting transition}\label{sec:limiting-problem}

We now identify the scalar transition of Step~1 in \zcref{s:ideas}.
We first explain heuristically how it arises.  Consider the perturbed
covariance $\Psi=\Xi+\lambda q^p$ and a candidate minimizer
$m\nu+(1-m)\delta_Q$, with $Q$ very close to $1$.  We call the atoms of
$\nu$ the \emph{lower atoms}.  The cumulative mass between $\supp\nu$ and
$Q$ is $m$.  For a fixed cutoff $\bar q>\max\supp\nu$, the Cole--Hopf step
across $[\bar q,Q]$ has variance $W=\Psi'(Q)-\Psi'(\bar q)$, which is approximately $\lambda p$ because
$Q^{p-1}$ is close to $1$.  The large-variance Cole--Hopf
expansion of \zcref{lem:large-V} is
\[
 T_{m,W}\log\cosh
 =\frac{mW}2+\Bigl(\frac1m-1\Bigr)\log2
 +\psi_m+O(e^{-cW}),\qquad
 \psi_m(x)=\frac1m\log\cosh(mx).
\]
Up to a constant and an exponentially small error, the Parisi solution at
$\bar q$ is therefore $\psi_m$, so the recursion for the lower atoms starts
from $\psi_m$ instead of $\log\cosh$.  By rescaling the spatial variable in
the Parisi PDE (\zcref{lem:temperature-mass-scaling}), we identify the
resulting variational problem for the lower atoms with the Parisi problem
for the covariance $m^2\Xi$.  The
term $mW/2$, of order $p$, cancels against the deterministic part of the Parisi
functional, and we expect that, to leading order, the Parisi functional of
the candidate minimizer is
\[
 F_\lambda(z,m):=\frac1m\Pp_{m^2\Xi}(\nu_z)+\frac{\lambda m}2.
\]
We prove this decomposition in \zcref{lem:lower-recursion}, up to an
explicit integral term and an exponentially small error.

Fix an admissible covariance $\Xi$ with a nondegenerate $r$-RSB Parisi
measure $\mu_* =\nu_{z_*}$.  By \zcref{prop:stability}, for $b$ near $1$
the Parisi measure of $b^2\Xi$ is $\nu_{z(b)}$, with $z(b)$ smooth and
$z(1)=z_*$.  Set
\[
 \mathfrak p(b)=\inf_{\mu\in\M}\Pp_{b^2\Xi}(\mu).
\]
Then $\mathfrak p(b)=\Pp_{b^2\Xi}(\nu_{z(b)})$ is smooth near $1$ by
\zcref{prop:finite-step-regularity}.
 Optimizing
over $z$ therefore gives
\[
 J_\lambda(m)=F_\lambda(z(m),m)
 =\frac{\mathfrak p(m)}m+\frac{\lambda m}{2}.
\]
The remaining question is whether the minimizing mass stays at $m=1$ or
moves into $m<1$.

\subsection{Temperature--mass scaling}\label{s:tempmass}

The rescaling described above follows directly from the Gaussian
Cole--Hopf recursion.  Let $\Xi$ be admissible, let $0<m\leq1$, let
$\nu\in\M$ be atomic with $\max\supp\nu<1$, and fix a cutoff
$q_{\mathrm c}\in(\max\supp\nu,1]$.  Let $0=q_0<\cdots<q_r$ be $0$
together with the atoms of $\nu$ in $(0,1)$, and set
$q_{r+1}=q_{\mathrm c}$.  The \emph{lower problem} is the recursion below
the cutoff from the heuristic above: it has cumulative mass $m\alpha_\nu$
on $[0,q_{\mathrm c})$ and terminal function
$\psi_m(x)=m^{-1}\log\cosh(mx)$ at $q_{\mathrm c}$.  Its solution
$\widetilde u$ is defined on $[0,q_{\mathrm c}]\times\R$ by
$\widetilde u(q_{\mathrm c},\cdot)=\psi_m$ and, as in
\eqref{eq:cole-hopf-step},
\begin{equation}\label{eq:lower-problem}
 \widetilde u(q,\cdot)
 =T_{m\alpha_\nu(q_i),\,\Xi'(q_{i+1})-\Xi'(q)}\widetilde u(q_{i+1},\cdot),
 \qquad q_i\leq q\leq q_{i+1},\quad 0\leq i\leq r.
\end{equation}

\begin{lemma}\label{lem:temperature-mass-scaling}
In this setting, for $0\leq q\leq q_{\mathrm c}$ and $x\in\R$,
\[
 \widetilde u(q,x)=\frac1mu_{m^2\Xi,\nu}(q,mx)
 -\frac m2\bigl(\Xi'(1)-\Xi'(q_{\mathrm c})\bigr).
\]
\end{lemma}

\begin{proof}
For $a,t\geq0$ and a one-Lipschitz function $g$ on $\R$, we have
\begin{equation}\label{eq:T-scaling}
 T_{ma,t}\bigl[m^{-1}g(m\,\cdot)\bigr](x)
 =\frac1m\bigl(T_{a,m^2t}g\bigr)(mx).
\end{equation}
Indeed, if $a>0$, then by \eqref{eq:T-def} both sides equal
$(ma)^{-1}\log\E e^{ag(mx+m\sqrt tZ)}$, and if $a=0$, then both sides
equal $m^{-1}\E g(mx+m\sqrt tZ)$.

Let $u=u_{m^2\Xi,\nu}$, and let $v(q,x)$ denote the right side of the
claimed formula for $\widetilde u$.  Since $\alpha_\nu=1$ on
$[q_{\mathrm c},1]$ and the covariance is $m^2\Xi$, we find from
\eqref{eq:T-logcosh} that
$u(q_{\mathrm c},y)=\log\cosh y+\frac{m^2}2(\Xi'(1)-\Xi'(q_{\mathrm c}))$.
Then $v(q_{\mathrm c},\cdot)=\psi_m=\widetilde u(q_{\mathrm c},\cdot)$.
Let $0\leq i\leq r$ and $q_i\leq q\leq q_{i+1}$.  Since $\alpha_\nu$ is
equal to $\alpha_\nu(q_i)$ on $[q_i,q_{i+1})$,
\eqref{eq:cole-hopf-step} implies that
$u(q,\cdot)=T_{\alpha_\nu(q_i),\,m^2(\Xi'(q_{i+1})-\Xi'(q))}u(q_{i+1},\cdot)$.
We apply \eqref{eq:T-scaling} with $g=u(q_{i+1},\cdot)$, which is
one-Lipschitz by \zcref{lem:curvature}, and use the identity
$T_{a,V}(g+c)=T_{a,V}g+c$ for constants $c$, which is immediate from
\eqref{eq:T-def}.  This shows that
$v(q,\cdot)=T_{m\alpha_\nu(q_i),\,\Xi'(q_{i+1})-\Xi'(q)}v(q_{i+1},\cdot)$,
which is \eqref{eq:lower-problem} for $v$.  By backward induction over
$i=r,\ldots,0$, we conclude that $v=\widetilde u$.
\end{proof}

The optimal diffusion and the first-variation function of the lower
problem rescale in the same way.  Let $\widetilde X$ solve
$\dd\widetilde X_q=m\Xi''(q)\alpha_\nu(q)\widetilde u_x(q,\widetilde X_q)\dd q
+\sqrt{\Xi''(q)}\dd B_q$ with $\widetilde X_0=0$.  Then $m\widetilde X$ is
the optimal diffusion of $(m^2\Xi,\nu)$, and the first-variation function
$\frac12\int_0^q\Xi''(s)(\E\widetilde u_x(s,\widetilde X_s)^2-s)\dd s$ of
the lower problem equals $m^{-2}f_{m^2\Xi,\nu}$ on $[0,q_{\mathrm c}]$
(\zcref{lem:temperature-mass-diffusion}).

\subsection{Entropy and the critical coefficient}\label{s:limiting-objective}

The boundary derivative of $J_\lambda$ is determined by the entropy of
the unperturbed model.  Let
$\mathfrak h(x)=\log2+\log\cosh x-x\tanh x$, the entropy of a $\pm1$ spin with
mean $\tanh x$; thus $0<\mathfrak h\leq\log2$.  The following identity, which we
prove in \zcref{s:scaling-proofs}, expresses the $m$-derivative of
$\frac1m\Pp_{m^2\Xi}(\nu)$ at $m=1$ as the negative of the first variation from moving mass
to $q=1$, minus the entropy $\E \mathfrak h(X_1)$ of a spin with mean $\tanh X_1$.

\begin{lemma}\label{lem:entropy-identity}
Let $\Xi$ be admissible, let $\nu=\nu_z$ for some $z\in\mathcal Z_r$, and
let $X$ be the optimal diffusion of $\nu$.  Then
$m\mapsto\frac1m\Pp_{m^2\Xi}(\nu)$ is smooth near $m=1$, and its derivative
at $m=1$ equals
\[
 f_{\Xi,\nu}(1)-\int_0^1f_{\Xi,\nu}\dd\nu-\E \mathfrak h(X_1).
\]
\end{lemma}

\begin{corollary}\label{cor:positive-entropy}
Let $X$ be the optimal diffusion of the Parisi measure $\mu_*$ of $\Xi$.
Then
\[
 s_\Xi=\mathfrak p(1)-\mathfrak p'(1)=\E \mathfrak h(X_1)-f_{\Xi,\mu_*}(1)>0.
\]
\end{corollary}

\begin{proof}
By the stationarity of $z_*$, when differentiating
$\mathfrak p(b)/b=b^{-1}\Pp_{b^2\Xi}(\nu_{z(b)})$ at $b=1$, the term
involving $z'(1)$ vanishes.  Using \zcref{lem:entropy-identity}, we obtain
\[
 \mathfrak p'(1)-\mathfrak p(1)
 =\left.\frac{\dd}{\dd b}\frac{\mathfrak p(b)}b\right|_{b=1}
 =\left.\frac{\dd}{\dd m}\frac1m\Pp_{m^2\Xi}(\mu_*)\right|_{m=1}
 =f_{\Xi,\mu_*}(1)-\E \mathfrak h(X_1),
\]
since $f_{\Xi,\mu_*}=0$ on $\supp\mu_*$.  Thus
$s_\Xi=\E \mathfrak h(X_1)-f_{\Xi,\mu_*}(1)>0$, because $\mathfrak h>0$ and
$f_{\Xi,\mu_*}\leq0$ (\zcref{s:nondegenerate}).
\end{proof}

The function $\mathfrak p$ is convex, because the finite-volume
pressures $\mathfrak p_N(b)=N^{-1}\E\log\sum_\sigma e^{bH_N(\sigma)}$ are
convex by H\"older's inequality and, by the Parisi formula for the
covariance $b^2\Xi$ \cite{Talagrand2006,Panchenko2014}, converge pointwise
to $\mathfrak p$.
Set $s(m)=\mathfrak p(m)-m\mathfrak p'(m)$ and
$\lambda_{\mathrm c}=2s_\Xi$.  Differentiation gives
\begin{equation}\label{eq:J-derivative}
 J_\lambda'(m)=\frac\lambda2-\frac{s(m)}{m^2},
 \qquad
 J_\lambda''(m)=\frac{\mathfrak p''(m)}m+\frac{2s(m)}{m^3}.
\end{equation}

In particular,
\begin{equation}\label{eq:J1}
 J_\lambda'(1)=\frac{\lambda-\lambda_{\mathrm c}}2,
 \qquad J_\lambda''(1)=\kappa_\Xi
 =\mathfrak p''(1)+2s_\Xi>0.
\end{equation}
Here $\kappa_\Xi>0$, because $\mathfrak p$ is convex and smooth near $1$,
so that $\mathfrak p''(1)\geq0$, and $s_\Xi>0$ by
\zcref{cor:positive-entropy}.  Thus $J_\lambda$ is strictly convex near
$m=1$.  For
$\lambda\leq\lambda_{\mathrm c}$, its minimum over masses $m\leq1$ near
$1$ is attained at $m=1$.  For $\lambda$ slightly above
$\lambda_{\mathrm c}$, the unique minimizer $m$ satisfies $m<1$ and
$J_\lambda'(m)=0$.  This identifies the limiting transition at twice the
entropy per spin (\zcref{rem:entropy}).

Since $\nu_{z(m)}$ minimizes $\Pp_{m^2\Xi}$ and $\mathcal Z_r$ is open, the
point $z(m)$ is a critical point of $F_\lambda(\cdot,m)$.  This is the first
identity below, and the chain rule applied to
$J_\lambda(m)=F_\lambda(z(m),m)$ gives the second:
\begin{equation}\label{eq:limiting-geometry}
 \partial_zF_\lambda(z(m),m)=0,
 \qquad
 \partial_mF_\lambda(z(m),m)=J_\lambda'(m).
\end{equation}
Moreover,
$D_z^2F_\lambda(z_*,1)=D_z^2\Pp_\Xi(\nu_z)|_{z=z_*}$ is positive definite by
\zcref{lem:atomic-hessian}, since $z_*$ is a critical point and
$f_{\Xi,\mu_*}''(q_i)<0$ for $1\leq i\leq r$ by
\zcref{def:nondegenerate}(ii).

\begin{remark}\label{rem:entropy}
The quantity $s_\Xi$ is the limit of the expected entropy per spin
of the Gibbs measure $G_N(\sigma)\propto e^{H_N(\sigma)}$.  Indeed, write
$Z_N=\sum_\sigma e^{H_N(\sigma)}$ and $\langle\cdot\rangle$ for the average
under $G_N$.  The entropy
$S_N=-\sum_\sigma G_N(\sigma)\log G_N(\sigma)$ equals
$\log Z_N-\langle H_N\rangle$, and $\mathfrak p_N'(1)=N^{-1}\E\langle
H_N\rangle$, where differentiation under $\E$ is justified because
$|\partial_b\log\sum_\sigma e^{bH_N(\sigma)}|\leq\max_\sigma|H_N(\sigma)|$,
which is integrable.  Hence $N^{-1}\E S_N=\mathfrak p_N(1)-\mathfrak p_N'(1)$.  By
convexity,
$\frac1\varepsilon(\mathfrak p_N(1)-\mathfrak p_N(1-\varepsilon))
\leq\mathfrak p_N'(1)\leq
\frac1\varepsilon(\mathfrak p_N(1+\varepsilon)-\mathfrak p_N(1))$ for
$\varepsilon>0$.  We let $N\to\infty$ and then $\varepsilon\to0$.  Since
$\mathfrak p$ is differentiable at $1$ (see also
\cite{Panchenko2008Differentiability}), we find that
$\mathfrak p_N'(1)\to\mathfrak p'(1)$, and hence
$N^{-1}\E S_N\to s_\Xi$.
\end{remark}

\section{Adding one RSB step}\label{sec:high-degree-perturbation}

We now prove the analogue of the limiting transition for large finite $p$
and identify
the resulting candidate minimizers as Parisi measures.  First, we eliminate the
location $Q$ of the new atom from the variational problem.  Then we optimize the lower atoms and their
total mass.  Finally, we show that the first-variation inequality holds on all of
$[0,1]$.

\subsection{Candidate minimizer and parameter ranges}\label{s:transition-setup}

Fix $\Xi$, $\mu_*=\nu_{z_*}$, $z(m)$, and $\lambda_{\mathrm c}=2s_\Xi$
as in \zcref{sec:limiting-problem}, and write
$a(z)=\max\supp\nu_z$.  We fix the following choices, which depend only
on $\Xi$.  When $r=0$, the vector $z$ of coordinates of the lower atoms is empty, and
$\nu_z=\delta_0$; the conventions for this case are stated at the end of
(S2).

These choices keep the minimizations over the lower atom coordinates
and their total mass strictly convex, and provide a fixed cutoff above
the lower atoms where the limiting first-variation function is uniformly
negative.  All choices remain fixed as $p$ tends to infinity.

\begin{enumerate}[label=\textup{(S\arabic*)},leftmargin=*]
 \item \emph{Lower coordinates.}  Let $r\geq1$.  By
 \zcref{prop:finite-step-regularity}, $D_z^2F_\lambda$ is continuous, and it
 does not depend on $\lambda$; by \zcref{lem:atomic-hessian}, it is positive
 definite at $(z_*,1)$.  Fix a closed ball $\cN\subset\mathcal Z_r$
 of radius $2\rho_0$ centered at $z_*$, a constant $\kappa_{\mathrm z}>0$, and
 $m_0\in(0,1)$ such that $D_z^2F_\lambda\succeq2\kappa_{\mathrm z}I$ on
 $\cN\times[m_0,1]$.
 \item \emph{Mass interval.}  Increasing $m_0$ within $(0,1)$ if necessary, we may assume
 that, for $m\in[m_0,1]$, the point $z(m)$ is defined,
 $\nu_{z(m)}$ is the nondegenerate $r$-RSB Parisi measure of $m^2\Xi$,
 $|z(m)-z_*|\leq\rho_0$, and $J_\lambda''(m)\geq2\kappa_{\mathrm m}$ for a constant
 $\kappa_{\mathrm m}>0$; this is possible because $J_\lambda''(1)=\kappa_\Xi>0$ by \eqref{eq:J1}.  In
 particular, $\operatorname{dist}(z(m),\partial\cN)\geq\rho_0$.  If
 $r=0$, we choose $m_0\in(0,1)$ with the properties not involving $\cN$ and set
 $\cN=\{z_*\}$; all assertions about $z$-derivatives and $\partial\cN$
 below are then vacuous.
 \item \emph{Compactness.}  The masses and support gaps of $\nu_z$ are
 bounded away from $0$ uniformly for $z\in\cN$, and there exists $a_*<1$
 such that $a(z)\leq a_*$ on $\cN$.
 \item \emph{Coefficient interval.}  Since $J_\lambda''>0$ on $[m_0,1]$ and
 $J_\lambda''$ does not depend on $\lambda$, the number
 $d_0=J_\lambda'(1)-J_\lambda'(m_0)$ is positive and does not depend on
 $\lambda$.  Fix
 \[
  0<\eta<\min\Bigl\{\frac{\lambda_{\mathrm c}}4,\frac{d_0}2\Bigr\},
  \qquad
  \Lambda=[\lambda_-,\lambda_+]
  =[\lambda_{\mathrm c}-2\eta,\ \lambda_{\mathrm c}+2\eta].
 \]
 Then $\Lambda\subset(0,\infty)$, and, by \eqref{eq:J-derivative}, for
 every $\lambda\in\Lambda$,
 \begin{equation}\label{eq:window}
  J_\lambda'(m_0)=\tfrac12(\lambda-\lambda_{\mathrm c})-d_0<-\eta,
  \qquad
  J_\lambda'(1)=\tfrac12(\lambda-\lambda_{\mathrm c})
  =\begin{cases}-\eta,&\lambda=\lambda_-,\\ \eta,&\lambda=\lambda_+.\end{cases}
 \end{equation}
 \item \emph{Cutoff and margin.}  For $m\in[m_0,1]$, let
 $f^\circ_m=m^{-2}f_{m^2\Xi,\nu_{z(m)}}$.  By
 \zcref{lem:temperature-mass-scaling,lem:temperature-mass-diffusion} 
 with $q_{\mathrm c}=1$, this is the first-variation function of the
limiting lower problem for $\nu=\nu_{z(m)}$, with terminal function
$\psi_m$ at $q=1$.  By (S2) and (S3), $f^\circ_m\leq0$ with
 equality exactly on $\supp\nu_{z(m)}\subset[0,a_*]$, and
 $(f^\circ_m)''<0$ at the positive support points.

 Since $m\mapsto m^2\Xi$
 is continuous into $C^3([0,1])$ and $z(\cdot)$ is continuous,
 the map $m\mapsto f^\circ_m$ is continuous into $C^2([0,1])$ by
 \zcref{lem:persistence}.  Since $f^\circ_m(1)<0$ for every $m$ by
 \zcref{def:nondegenerate}(i), by compactness there exists $\gamma>0$ such
 that $f^\circ_m(1)\leq-5\gamma$ for all $m\in[m_0,1]$.  Since
 $(f^\circ_m)'=\frac12\Xi''\cdot(\Gamma_{m^2\Xi,\nu_{z(m)}}-q)$ by
 \eqref{eq:f-derivatives}, $0\leq\Gamma\leq1$, and $0<\Xi''\leq\Xi''(1)$ on
 $(0,1]$, we have
 $|(f^\circ_m)'|\leq\frac12\Xi''(1)$, and hence
 $f^\circ_m(q)\leq-5\gamma+\frac12\Xi''(1)(1-q)$.  We may therefore fix a
 cutoff $\bar q\in(a_*,1)$, with $1-\bar q$ so small that
 \begin{equation}\label{eq:cutoff-choice}
  f^\circ_m\leq-4\gamma\ \text{ on }[\bar q,1]\quad(m\in[m_0,1]),
  \qquad
  \frac12\int_{\bar q}^1(1-q)\Xi''(q)\dd q\leq\gamma.
 \end{equation}
\end{enumerate}

In this section and in \zcref{app:uniform}, the constants $c,C,A>0$ and $C_j$ may
change from line to line, but depend only on $\Xi$ and (S1)--(S5), not on
$p$, $\lambda$, $z$, $m$, or $Q$.

For large even $p$, let $\Psi=\Xi_{p,\lambda}=\Xi+\lambda q^p$, where we
suppress $p$ and $\lambda$ from the notation. The covariance $\Psi$ is admissible once $p$
exceeds the degree of $\Xi$.  Consider
\[
 \mu_{z,m,Q}=m\nu_z+(1-m)\delta_Q,
 \qquad z\in\cN,\quad m\in[m_0,1],\quad a(z)<Q<1.
\]
The parameter $m$ is the total mass of the lower atoms, the atoms of
$\nu_z$.  Write
$\theta=(z,\lambda,m)$, $\Theta=\cN\times\Lambda\times[m_0,1]$, and
$\Pp_p(\theta,Q)=\Pp_\Psi(\mu_{z,m,Q})$.
Below the fixed cutoff, the perturbation and its derivatives are
exponentially small:
\begin{equation}\label{eq:high-term-small-fixed}
 \sup_{0\leq q\leq\bar q}\bigl|\partial_q^j(\lambda q^p)\bigr|
 \leq C_jp^j\bar q^{\,p-j},
 \qquad \lambda\in\Lambda.
\end{equation}

On the closed
parameter sets used below, $C^k$ means that the interior derivatives
through order $k$ extend continuously to the boundary, and derivatives at
$m=1$ are from the left.  Identities between derivatives extend to the boundary by continuity.
In particular, at $m=1$, the continuous extension of the $m$-derivative
agrees with the left derivative of the function. 
If $r=0$, we identify $\Theta$ with $\Lambda\times[m_0,1]$.

\subsection{Locating the new atom}\label{s:root}

For $m<1$, the first identity in \eqref{eq:location-derivative},
applied to the atom at $Q$ with mass $1-m$, and \eqref{eq:f-derivatives}
give
\begin{equation}\label{eq:Q-derivative}
 \partial_Q\Pp_p(\theta,Q)
 =
 -\frac{1-m}{2}\Psi''(Q)
 \bigl(\Gamma_{\Psi,\mu_{z,m,Q}}(Q)-Q\bigr).
\end{equation}

Since $\Psi''(Q)>0$, stationarity in $Q$ is equivalent to the
\emph{location equation}
\begin{equation}\label{eq:top-equation}
 Q=\Gamma_{\Psi,\mu_{z,m,Q}}(Q).
\end{equation}

At $m=1$, the point $Q$ has zero mass, and the functional does not depend
on $Q$.  We still use this equation to choose $Q$;
the value of $f$ there will determine when the new atom acquires mass.

The equation \eqref{eq:top-equation} has a useful interpretation.  Let
$\mathcal A_p(\theta,Q)=1-\Gamma_{\Psi,\mu_{z,m,Q}}(Q)$ be the
\emph{magnetization deficit} at $Q$.  Then stationarity reads
\begin{equation}\label{eq:balance}
 1-Q=\mathcal A_p(\theta,Q).
\end{equation}

The derivative $\partial_Q\mathcal A_p$ varies both the evaluation point and
the candidate minimizer $\mu_{z,m,Q}$.  By contrast, $\Gamma'(q)$ in
\eqref{eq:Gamma-derivative} differentiates the overlap variable with
the measure fixed.

As in \zcref{sec:limiting-problem}, for $Q$ exponentially close to $1$,
the variance between $\bar q$ and $Q$ is approximately $\lambda p$, and
this large variance makes $|X_Q|$ typically large, so the magnetization
$\tanh X_Q$ at $Q$ is close to $\pm1$.  We will
show that $\mathcal A_p$ is of order $p^{-1/2}e^{-m^2\lambda p/2}$ (see
\eqref{eq:endpoint-two-sided} below), and we define the interval
\[
 S_p=[1-e^{-m_0^2\lambda_-p/4},\,1-e^{-\lambda_+p}].
\]
Since $m_0^2\lambda_-/4<m^2\lambda/2<\lambda_+$ on $\Theta$, the scale of $\mathcal A_p$ lies
strictly between $e^{-\lambda_+p}$ and $e^{-m_0^2\lambda_-p/4}$, with
exponentially large room on either side.

We solve the location equation using two estimates, which we prove in
\zcref{app:uniform}.  The first makes the heuristic at the beginning of
\zcref{sec:limiting-problem} precise.  It compares the Parisi functional
$\Pp_p(\theta,Q)$ of the candidate minimizer with
$F(\theta)=F_\lambda(z,m)$.  The difference consists of the
term $\mathcal T_p$ below, which is $1-m$ times the integral
$\frac12\int_Q^1(1-q)\Psi''(q)\dd q$ above $Q$, and an exponentially small
$C^4$ remainder.

\begin{lemma}\label{lem:lower-recursion}
For every sufficiently large even $p$ and every
$(\theta,Q)\in\Theta\times S_p$,
\begin{equation}\label{eq:candidate-decomposition}
 \Pp_p(\theta,Q)
 =
 F(\theta)
 +\mathcal T_p(\theta,Q)
 +\mathcal R_p(\theta,Q),
 \qquad
 \mathcal T_p(\theta,Q)=\frac{1-m}{2}\int_Q^1(1-q)\Psi''(q)\dd q,
\end{equation}
where the remainder $\mathcal R_p=\Pp_p-F-\mathcal T_p$ is $C^4$ on
$\Theta\times S_p$, with
\begin{equation}\label{eq:lower-remainder-bound}
 \max_{|\mathbf j|\leq4}
 \sup_{\Theta\times S_p}
 \bigl|
 \partial_{z,\lambda,m,Q}^{\mathbf j}\mathcal R_p
 \bigr|
 \leq Cp^Ae^{-cp}.
\end{equation}
\end{lemma}

The term $\mathcal T_p$ is at most $Cp^2(1-Q)^2$, but it cannot be
absorbed into the remainder, since its second $Q$-derivative
$\frac12(1-m)(\Psi''(Q)-(1-Q)\Psi'''(Q))$ is of order $(1-m)p^2$.  Its
first $Q$-derivative is also needed explicitly.  The first identity below
follows from a direct computation.  Since $F$ does not depend on $Q$, we
obtain the second by comparing the $Q$-derivative of
\eqref{eq:candidate-decomposition} with \eqref{eq:Q-derivative}:
\begin{equation}\label{eq:balance-derivatives}
 \partial_Q\mathcal T_p=-\frac{1-m}2\Psi''(Q)(1-Q),
 \qquad
 \partial_Q\mathcal R_p=\frac{1-m}2\Psi''(Q)\mathcal A_p(\theta,Q).
\end{equation}
At $m=1$, both identities hold because all of their terms vanish.  By
\eqref{eq:balance}, the location equation states that the two derivatives
cancel.

The second estimate concerns the deficit.  Its size comes from the large
variance above $\bar q$, through a comparison of the optimal diffusion on
$[\bar q,Q]$ with an explicit mixture of two Gaussians
(\zcref{lem:doob-kernel}).  The same comparison bounds $1-\Gamma(q)$ for
$\bar q\leq q\leq Q$,
which we use in \zcref{s:global}.  The regularity of $\mathcal A_p$ comes from the second
identity in \eqref{eq:balance-derivatives}, which expresses
$(1-m)\mathcal A_p$ through $\partial_Q\mathcal R_p$.  Since
$\partial_Q\mathcal R_p$ vanishes at $m=1$, writing it as an integral of
its $m$-derivative removes the factor $1-m$.  Then the $C^4$ bound
\eqref{eq:lower-remainder-bound} implies $C^2$ bounds on $\mathcal A_p$
that are uniform up to $m=1$.

\begin{lemma}\label{lem:deficit}
For every sufficiently large even $p$, the magnetization deficit $\mathcal A_p$ is $C^2$
on $\Theta\times S_p$, and, for all $(\theta,Q)\in\Theta\times S_p$,
\begin{align}
 \max_{|\mathbf j|\leq2}
 \bigl|\partial_{z,\lambda,m,Q}^{\mathbf j}\mathcal A_p(\theta,Q)\bigr|
 &\leq Cp^Ae^{-cp},
 \label{eq:endpoint-crude}\\
 cp^{-1/2}e^{-m^2\lambda p/2}
 \leq\mathcal A_p(\theta,Q)
 &\leq Cp^{-1/2}e^{-m^2\lambda p/2}.
 \label{eq:endpoint-two-sided}
\end{align}
Moreover, the mean squared magnetization $\Gamma$ of $\mu_{z,m,Q}$ for $\Psi$
satisfies
\begin{equation}\label{eq:deficit-intermediate}
 1-\Gamma(q)\leq4e^{-m^2(\Psi'(q)-\Psi'(\bar q))/2},
 \qquad \bar q\leq q\leq Q.
\end{equation}
\end{lemma}

We can now locate the new atom.

\begin{proposition}\label{prop:uniform-reduction}
For every sufficiently large even $p$ and every $\theta\in\Theta$, the
location equation \eqref{eq:top-equation} has exactly one solution $Q=Q_p(\theta)$
in $S_p$.  It lies in the interior of $S_p$ and satisfies
\begin{equation}\label{eq:transition-Q-root}
 C^{-1}p^{-1/2}e^{-m^2\lambda p/2}
 \leq
 1-Q_p(\theta)
 \leq
 Cp^{-1/2}e^{-m^2\lambda p/2}.
\end{equation}
\end{proposition}

\begin{proof}
Set
\[
 \phi_p(\theta,Q)=Q-1+\mathcal A_p(\theta,Q).
\]
By \eqref{eq:balance}, the location equation is
$\phi_p(\theta,Q)=0$.  By \zcref{lem:deficit}, $\phi_p$ is $C^2$ on
$\Theta\times S_p$, and by \eqref{eq:endpoint-crude}, we have
\[
 \partial_Q\phi_p=1+\partial_Q\mathcal A_p\geq\frac12
\]
on $\Theta\times S_p$ for sufficiently large $p$.

At the left endpoint of $S_p$, we have
$1-Q=e^{-m_0^2\lambda_-p/4}$, whereas \eqref{eq:endpoint-two-sided} and
$m^2\lambda\geq m_0^2\lambda_-$ give
\[
 \mathcal A_p(\theta,Q)
 \leq Cp^{-1/2}e^{-m_0^2\lambda_-p/2}.
\]
Hence $\phi_p<0$ there for large $p$.
At the right endpoint, $1-Q=e^{-\lambda_+p}$, while
\eqref{eq:endpoint-two-sided} and $m^2\lambda\leq\lambda_+$ give
\[
 \mathcal A_p(\theta,Q)
 \geq cp^{-1/2}e^{-\lambda_+p/2},
\]
so $\phi_p>0$ there.
Strict monotonicity therefore gives exactly one solution
$Q_p(\theta)$ in $S_p$, and it lies in the interior of $S_p$.
Evaluating \eqref{eq:endpoint-two-sided} at this solution and using
$1-Q_p=\mathcal A_p(\theta,Q_p)$ proves
\eqref{eq:transition-Q-root}.
\end{proof}

Substituting $Q=Q_p(\theta)$, we obtain the reduced functional
$\widehat\Pp_p(\theta)=\Pp_p(\theta,Q_p(\theta))$.  Since $Q_p(\theta)$
solves the location equation \eqref{eq:top-equation}, \eqref{eq:Q-derivative}
gives $\partial_Q\Pp_p(\theta,Q_p(\theta))=0$.
The chain rule therefore implies that the first-order partial derivatives of
$\widehat\Pp_p$ with respect to the variables in $\theta$ are those of $\Pp_p$ at fixed $Q$.  At fixed $Q$, the $\theta$-derivatives of
$\Pp_p-F=\mathcal T_p+\mathcal R_p$ through third order are exponentially small, by
\zcref{lem:lower-recursion} and because $1-Q$ is exponentially small on
$S_p$.  By \eqref{eq:balance} and \eqref{eq:endpoint-crude}, the first and second derivatives of $Q_p$ with respect to $\theta$ are
exponentially small as well.  Then the additional chain-rule terms in the derivatives of
$(\Pp_p-F)(\theta,Q_p(\theta))$ (see \eqref{eq:reduced-chain-rule})
remain exponentially small through third
order, yielding $C^3$ closeness of the reduced functional $\widehat \Pp_p$ to $F$. The next
proposition makes this precise; we prove it in
\zcref{s:reduction-proofs}.

\begin{proposition}\label{prop:reduced-functional}
For every sufficiently large even $p$, the following hold.
\begin{enumerate}[label=\textup{(\roman*)}]
 \item The reduced functional $\widehat\Pp_p$ is $C^3$ and satisfies
\begin{equation}\label{eq:envelope}
 \partial_j\widehat\Pp_p(\theta)
 =
 \partial_j\Pp_p\bigl(\theta,Q_p(\theta)\bigr)
 \quad\text{for every coordinate $\theta_j$ of $\theta$}.
\end{equation}
 \item The reduced functional satisfies
\begin{equation}\label{eq:transition-C3-limit}
 \bigl\|\widehat\Pp_p-F\bigr\|_{C^3(\Theta)}
 \leq Cp^Ae^{-cp}.
\end{equation}
\end{enumerate}
\end{proposition}

\subsection{Optimizing the lower atoms and their mass}\label{s:branches}

For each $(\lambda,m)$, we next optimize the lower coordinates $z$. The
resulting objective is $J_p(\lambda,m)$.  This step uses the strict curvature condition \eqref{eq:strict-replicon}, through (S1).  By (S1) and
\zcref{prop:reduced-functional}(ii), the reduced functional
$\widehat\Pp_p$ is strongly convex in $z$ for large $p$, and $z(m)$ stays at
distance at least $\rho_0$ from $\partial\cN$ by (S2).  This suggests a unique
minimizer in the interior of $\cN$, close to $z(m)$.  The next
proposition, which we prove in \zcref{s:reduction-proofs}, confirms this
and shows that $J_p$ converges to the limiting objective $J_\lambda$.

\begin{proposition}\label{prop:stationary-branches}
For every sufficiently large even $p$, the following hold.
\begin{enumerate}[label=\textup{(\roman*)}]
 \item If $r\geq1$, there exists a unique map
 $z_p\colon\Lambda\times[m_0,1]\to\operatorname{int}\cN$ such that
 $\partial_z\widehat\Pp_p(z_p(\lambda,m),\lambda,m)=0$.  It is $C^2$ and
 converges to $(\lambda,m)\mapsto z(m)$ in $C^2(\Lambda\times[m_0,1])$, and
 $D_z^2\widehat\Pp_p\succeq\kappa_{\mathrm z}I$ on $\Theta$.  If $r=0$, we set
 $z_p=z_*$.

 \item Let
 \[
  J_p(\lambda,m)=\widehat\Pp_p\bigl(z_p(\lambda,m),\lambda,m\bigr),
  \qquad
  g_p(\lambda,m)=\partial_mJ_p(\lambda,m).
 \]
 Then
 $g_p(\lambda,m)=\partial_m\widehat\Pp_p(z_p(\lambda,m),\lambda,m)$, $J_p\to
 J_\lambda$ in $C^3(\Lambda\times[m_0,1])$, and $g_p\to J_\lambda'$ in
 $C^2(\Lambda\times[m_0,1])$.
\end{enumerate}
\end{proposition}

Like $J_\lambda$ in \zcref{s:limiting-objective}, the function $J_p$ has a
unique critical coefficient, above which its minimizing mass is strictly
less than $1$.

\begin{proposition}\label{prop:scalar-bifurcation}
For every sufficiently large even $p$, we have
$\partial_mg_p\geq\kappa_{\mathrm m}$ and $\partial_\lambda g_p\geq\frac14$
on $\Lambda\times[m_0,1]$.  There exists a unique
$\lambda_{\mathrm c,p}\in\Lambda$ such that
$g_p(\lambda_{\mathrm c,p},1)=0$. The sign of $g_p(\lambda,1)$ is that of
$\lambda-\lambda_{\mathrm c,p}$.  Moreover,
$\lambda_{\mathrm c,p}\to\lambda_{\mathrm c}$ as $p\to\infty$.  For every
$\lambda\in\Lambda$, the function $J_p(\lambda,\cdot)$ has a unique
minimizer $m_{p,\lambda}$ on $[m_0,1]$.  If $\lambda\leq\lambda_{\mathrm c,p}$,
then $m_{p,\lambda}=1$; if $\lambda>\lambda_{\mathrm c,p}$, then
$m_0<m_{p,\lambda}<1$ and $g_p(\lambda,m_{p,\lambda})=0$.
\end{proposition}

\begin{proof}
By \zcref{prop:stationary-branches}(ii), $g_p\to J_\lambda'$ in
$C^2(\Lambda\times[m_0,1])$.  Since $J_\lambda''\geq2\kappa_{\mathrm m}$ on
$[m_0,1]$ by (S2) and $\partial_\lambda J_\lambda'=\frac12$ by
\eqref{eq:J-derivative}, we have $\partial_mg_p\geq\kappa_{\mathrm m}$ and
$\partial_\lambda g_p\geq\frac14$ for large $p$.  By \eqref{eq:window},
$J_\lambda'(m_0)<-\eta$ for every $\lambda\in\Lambda$, while
$J_{\lambda_-}'(1)=-\eta$ and $J_{\lambda_+}'(1)=\eta$.  By the uniform convergence $g_p\to J_\lambda'$ on
$\Lambda\times[m_0,1]$ from
\zcref{prop:stationary-branches}(ii), we have
$g_p(\lambda,m_0)<0$ for every $\lambda\in\Lambda$ and
$g_p(\lambda_-,1)<0<g_p(\lambda_+,1)$ for large $p$. Since $g_p(\cdot,1)$
is strictly increasing, it has exactly one zero
$\lambda_{\mathrm c,p}\in\Lambda$, and the sign of $g_p(\lambda,1)$ is that
of $\lambda-\lambda_{\mathrm c,p}$.  Since
$J_{\lambda_{\mathrm c}}'(1)=0$, the bound $\partial_\lambda g_p\geq\frac14$
implies that
\[
 |\lambda_{\mathrm c,p}-\lambda_{\mathrm c}|
 \leq4|g_p(\lambda_{\mathrm c},1)|
 =4\bigl|g_p(\lambda_{\mathrm c},1)-J_{\lambda_{\mathrm c}}'(1)\bigr|
 \longrightarrow0.
\]
Now fix $\lambda\in\Lambda$.  The derivative $g_p(\lambda,\cdot)$ of $J_p(\lambda,\cdot)$ is strictly
increasing.  If $\lambda\leq\lambda_{\mathrm c,p}$, then
$g_p(\lambda,1)\leq0$, and $J_p(\lambda,\cdot)$ is strictly decreasing on
$[m_0,1]$, with unique minimizer $m=1$.  Otherwise,
$g_p(\lambda,m_0)<0<g_p(\lambda,1)$, and $g_p(\lambda,\cdot)$ has a unique
zero $m_{p,\lambda}\in(m_0,1)$.  Then $J_p(\lambda,\cdot)$ decreases
before $m_{p,\lambda}$ and increases after it, and $m_{p,\lambda}$ is its
unique minimizer.
\end{proof}

\begin{definition}\label{def:candidates}
For $(\lambda,m)\in\Lambda\times[m_0,1]$, write
\[
 \mu_p(\lambda,m)=m\nu_{z_p(\lambda,m)}+(1-m)\delta_Q,
 \qquad Q=Q_p(z_p(\lambda,m),\lambda,m).
\]
This is the candidate minimizer whose lower atoms have total mass $m$.  The location $Q$ and the
lower coordinates $z$ are stationary; the mass $m$ remains free until
the scalar minimization selects $m=m_{p,\lambda}$.  When $m=1$, the
point $Q$ has zero mass.
\end{definition}

The next lemma relates the mass derivative to the point $Q$. Here, $g_p$ is the
value of the first-variation function at $Q$.  In particular, if the
minimizing mass is interior, then $f(Q)=0$.

\begin{lemma}\label{lem:stationary-candidates}
For every sufficiently large even $p$ and every
$(\lambda,m)\in\Lambda\times[m_0,1]$, let $z=z_p(\lambda,m)$ and
$Q=Q_p(z,\lambda,m)$, and let $f$ be the first-variation function of
$\mu_p(\lambda,m)$ for $\Psi$.  Then $f=0$ on $\supp\nu_z$, $f'=0$ on
$\supp\nu_z\setminus\{0\}$, $f'(Q)=0$, and $g_p(\lambda,m)=f(Q)$.
\end{lemma}

\begin{proof}
Let $\Gamma$ be the mean squared magnetization of $\mu_p(\lambda,m)$ for $\Psi$; at
$m=1$, we have $\mu_p(\lambda,m)=\nu_z$.  By \zcref{prop:uniform-reduction}, we have
$\Gamma(Q)=Q$.  Then $f'(Q)=\frac12\Psi''(Q)(\Gamma(Q)-Q)=0$.  By \eqref{eq:envelope}, the
$z$- and $m$-derivatives of $\widehat\Pp_p$ at $\theta=(z,\lambda,m)$ are
those of $\Pp_p$ at $(\theta,Q)$.  The $z$-derivatives vanish by the
definition of $z_p$, and \zcref{cor:atomic-gradient} implies that $f=0$ on
$\supp\nu_z$ and $f'=0$ on
$\supp\nu_z\setminus\{0\}$.  By \zcref{prop:stationary-branches}(ii),
$g_p(\lambda,m)=\partial_m\widehat\Pp_p(\theta)=\partial_m\Pp_p(\theta,Q)$.
Finally, let $m<1$, write $\mu=\mu_{z,m,Q}$, and let $0<s<1-m$.  Then
$\mu_{z,m+s,Q}=(1-t)\mu+t\nu_z$ with $t=s/(1-m)$, and
$\nu_z-\mu=(1-m)(\nu_z-\delta_Q)$.  By \eqref{eq:directional-derivative},
the right derivative in $m$ is
\[
 \frac1{1-m}\Bigl(-\int_0^1f\dd(\nu_z-\mu)\Bigr)
 =f(Q)-\int_0^1f\dd\nu_z=f(Q).
\]
Since $\Pp_p=F+\mathcal T_p+\mathcal R_p$ is differentiable in $m$ by
\zcref{lem:lower-recursion}, this right derivative is
$\partial_m\Pp_p(\theta,Q)$.
At $m=1$, the same identity follows from
\eqref{eq:new-point-derivative}, with $\mu=\nu_z$, using the derivative
from the left.
\end{proof}

We have therefore constructed candidate minimizers and a unique scalar critical
coefficient.  To establish an RSB transition for the Parisi measure, it
remains to verify $f\leq0$ at every overlap value.

\subsection{The sign of the first-variation function}\label{s:global}

We verify the first-variation inequality in three regions: below the
fixed cutoff $\bar q$, on an intermediate interval, and very close to
$1$.  For a candidate minimizer, let $f$ and $\Gamma$ denote its
first-variation function and mean squared magnetization for $\Psi$.  We have
$f'=\frac12\Psi''\cdot(\Gamma-q)$ and $f\in C^2([0,1])$ by
\eqref{eq:f-derivatives}.  The next lemma shows that, on $[0,\bar q]$, the
first-variation function of the finite-$p$ candidate minimizer behaves like
$f^\circ_m$: it is nonpositive, vanishes exactly at the lower atoms, and
has strictly negative second derivative at the positive ones.
There, the covariance perturbation is exponentially small by
\eqref{eq:high-term-small-fixed}, and the Parisi solution at $\bar q$ is
$\psi_m$ up to a constant and an exponentially small error
(\zcref{lem:cutoff-splitting}).  The proof of the next
lemma, in \zcref{s:lower-geometry-proof}, is a compactness argument: as
$p\to\infty$, the first-variation functions converge to $f^\circ_m$ in
$C^2([0,\bar q])$ by \zcref{lem:observable-continuity}, and their
zeros and signs then match those of $f^\circ_m$ for large $p$ by
\zcref{lem:contact-persistence}.

\begin{lemma}\label{lem:lower-geometry}
For every sufficiently large even $p$ and every
$(\lambda,m)\in\Lambda\times[m_0,1]$, the first-variation function $f$ of
the candidate minimizer $\mu_p(\lambda,m)$, with $z=z_p(\lambda,m)$,
satisfies $f\leq0$ on $[0,\bar q]$,
$\{q\in[0,\bar q]:f(q)=0\}=\supp\nu_z$, $f''(q)<0$ at every positive
$q\in\supp\nu_z$, and $f(\bar q)\leq-3\gamma$.
\end{lemma}

It remains to analyze the high-overlap region $(\bar q,1]$.  For all
sufficiently large $p$, we have $(10/m_0^2)\log p<\lambda_-p$.  Since
$q\mapsto\lambda pq^{p-1}$ increases from $0$ to $\lambda p$ on $[0,1]$,
there then exists a unique $\hat q_p\in(0,1)$, which depends on $\lambda$,
such that
\begin{equation}\label{eq:bp-def}
 \lambda p\,\hat q_p^{\,p-1}=\frac{10}{m_0^2}\log p.
\end{equation}
The point $\hat q_p$ balances two requirements: below it, the perturbation raises
$f$ by less than $|f(\bar q)|$; above it, $1-\Gamma$ is small enough to give
$\Gamma'<1$.
It lies at distance of order $(\log p)/p$ from $1$, much farther
from $1$ than the new atom.  Since $(p-1)\log(1/\hat q_p)=\log(\lambda
pm_0^2/(10\log p))\leq\log(\lambda_+p)$ for large $p$, and
$1-x\leq\log(1/x)$ for $0<x\leq1$, we have
$1-\hat q_p\leq\log(\lambda_+p)/(p-1)$.  In particular, $\hat q_p\to1$
uniformly for $\lambda\in\Lambda$.

\begin{proposition}\label{prop:high-overlap-geometry}
For every sufficiently large even $p$ and every
$(\lambda,m)\in\Lambda\times[m_0,1]$, let $f$ and $\Gamma$ be the
first-variation function and the mean squared magnetization of the candidate
minimizer $\mu_p(\lambda,m)$, and let $z=z_p(\lambda,m)$ and
$Q=Q_p(z,\lambda,m)$.  Then $\bar q<\hat q_p<Q$,
\begin{equation}\label{eq:intermediate-negative}
 f(q)\leq-\gamma,
 \qquad
 \bar q\leq q\leq\hat q_p,
\end{equation}
with $\gamma$ from \eqref{eq:cutoff-choice}, and
\begin{equation}\label{eq:Gamma-slope-tail}
 \Gamma'(q)\leq\frac12,
 \qquad
 \hat q_p\leq q\leq1.
\end{equation}
Consequently, $f$ is strictly increasing on $[\hat q_p,Q]$ and strictly
decreasing on $[Q,1]$, and $f''(Q)<0$.
\end{proposition}

\begin{proof}
We divide the proof into three steps.

\medskip
\noindent\emph{Step 1: Location of the three regions.}
Since $\hat q_p\to1$ uniformly, we have $\bar q<\hat q_p$ for large $p$.
Further, since $Q\in S_p$, we have
$Q^{p-1}\geq1-p(1-Q)\geq1-pe^{-m_0^2\lambda_-p/4}\geq\frac12$, whereas
$\hat q_p^{\,p-1}=10\log p/(m_0^2\lambda p)\to0$.  This implies that
$\hat q_p<Q$ for large $p$.

\medskip
\noindent\emph{Step 2: The interval $[\bar q,\hat q_p]$.}
Since $\Gamma\leq1$ by \eqref{eq:curvature} and $\Psi''>0$ on $(0,1]$, we
have $f'(s)\leq\frac12\Psi''(s)(1-s)$.  For $\bar q\leq q\leq\hat q_p$, we
integrate this inequality from $\bar q$ and use the bound
$f(\bar q)\leq-3\gamma$ of \zcref{lem:lower-geometry},
\eqref{eq:cutoff-choice}, and $\Psi''=\Xi''+\lambda p(p-1)q^{p-2}$ to
obtain
\[
 f(q)
 \leq
 f(\bar q)
 +\frac12\int_{\bar q}^1(1-s)\Xi''(s)\dd s
 +\frac\lambda2\int_{\bar q}^{\hat q_p}(1-s)p(p-1)s^{p-2}\dd s
 \leq
 -2\gamma
 +\frac\lambda2\int_{\bar q}^{\hat q_p}(1-s)p(p-1)s^{p-2}\dd s.
\]
Since $\frac{\dd}{\dd s}\bigl[ps^{p-1}-(p-1)s^p\bigr]=p(p-1)s^{p-2}(1-s)$
and $ps^{p-1}-(p-1)s^p=s^{p-1}\bigl(p-(p-1)s\bigr)\geq0$ on $[0,1]$,
\eqref{eq:bp-def} gives
\[
 \lambda\int_{\bar q}^{\hat q_p}(1-s)p(p-1)s^{p-2}\dd s
 \leq
 \lambda\hat q_p^{\,p-1}\bigl(p-(p-1)\hat q_p\bigr)
 =
 \frac{10\log p}{m_0^2}\left(1-\hat q_p+\frac{\hat q_p}p\right)
 \leq
 C\frac{(\log p)^2}{p}.
\]
For large $p$, the right side is at most $2\gamma$, which proves
\eqref{eq:intermediate-negative}.

\medskip
\noindent\emph{Step 3: The interval $[\hat q_p,1]$.}
Let $t=\Psi'(\hat q_p)-\Psi'(\bar q)$.  By \eqref{eq:bp-def} and
$\lambda p\bar q^{\,p-1}\to0$, we have $t=(10/m_0^2)\log p+O(1)$.  By
Step~1, we may apply \eqref{eq:deficit-intermediate} at $q=\hat q_p$.
Since $m\geq m_0$, we obtain
$1-\Gamma(\hat q_p)\leq4e^{-m^2t/2}\leq p^{-4}$ for large $p$.

By \eqref{eq:Gamma-curvature}, we have $\Gamma'\geq0$.  Then $1-\Gamma$ is
nonincreasing, and $1-\Gamma(q)\leq p^{-4}$ for $q\geq\hat q_p$.  We now
use the upper bound $\Gamma'\leq\Psi''(1-\Gamma)$ of
\eqref{eq:Gamma-curvature}.  Since $\sup_{[0,1]}\Psi''\leq Cp^2$ uniformly for
$\lambda\in\Lambda$, this bound gives $\Gamma'(q)\leq Cp^{-2}\leq\frac12$
for
$\hat q_p\leq q\leq1$ and large $p$, which is \eqref{eq:Gamma-slope-tail}.
By the location equation \eqref{eq:top-equation}, we have $\Gamma(Q)=Q$, and
$\hat q_p<Q$ by Step~1.  Since the derivative of $q\mapsto\Gamma(q)-q$ is
at most $-\frac12$ on $[\hat q_p,1]$, this function is positive on
$[\hat q_p,Q)$ and negative on $(Q,1]$, and the monotonicity of $f$ follows
from $f'=\frac12\Psi''\cdot(\Gamma-q)$.  Since $\Gamma\in C^1([0,1])$ and
$\Gamma(Q)=Q$, we have $f''(Q)=\frac12\Psi''(Q)(\Gamma'(Q)-1)<0$.
\end{proof}

\begin{corollary}\label{cor:global-sign}
In the setting of \zcref{prop:high-overlap-geometry}, if $f(Q)<0$, then
$f\leq0$ on $[0,1]$ and $\{f=0\}=\supp\nu_z$; if $f(Q)=0$, then $f\leq0$
on $[0,1]$ and $\{f=0\}=\supp\nu_z\cup\{Q\}$.
\end{corollary}

\begin{proof}
Let $f(Q)\leq0$, and cover $[0,1]$ by $[0,\bar q]$, $[\bar q,\hat q_p]$,
and $[\hat q_p,1]$.  By \zcref{lem:lower-geometry}, we have $f\leq0$ on
$[0,\bar q]$, with zero set $\supp\nu_z$ there.  On $[\bar q,\hat q_p]$, we have
$f\leq-\gamma$ by \eqref{eq:intermediate-negative}.  By
\zcref{prop:high-overlap-geometry}, we have $f\leq f(Q)\leq0$ on
$[\hat q_p,1]$, and the maximum $f(Q)$ is attained only at $Q$.
Thus $f$ vanishes on this interval if and only if $f(Q)=0$, and then only
at $Q$.  Both assertions follow.
\end{proof}

\subsection{The transition}\label{s:transition-conclusion}

We now combine \zcref{prop:scalar-bifurcation} with
\zcref{cor:global-sign} to identify the Parisi measures.

\begin{theorem}\label{thm:gardner}
Let $r\geq0$, and let $\Xi$ be an admissible covariance whose Parisi measure
$\mu_*$ is $r$-RSB and nondegenerate.  Define
$\mathfrak p(b)=\inf_{\mu\in\M}\Pp_{b^2\Xi}(\mu)$,
$s_\Xi=\mathfrak p(1)-\mathfrak p'(1)$, and $\lambda_{\mathrm c}=2s_\Xi$.
There exist $\eta\in(0,\lambda_{\mathrm c})$, $a_*<1$, and $C<\infty$ such
that, for every sufficiently large even integer $p$, there exists
$\lambda_{\mathrm c,p}\in(\lambda_{\mathrm c}-\eta,\lambda_{\mathrm c}+\eta)$
with the following properties.  For
$\lambda\in[\lambda_{\mathrm c}-\eta,\lambda_{\mathrm c}+\eta]$, the
covariance $\Xi_{p,\lambda}(q)=\Xi(q)+\lambda q^p$ has a unique Parisi
measure $\mu_{p,\lambda}$, and:
\begin{enumerate}[label=\textup{(\roman*)}]
 \item If $\lambda<\lambda_{\mathrm c,p}$, then $\mu_{p,\lambda}$ is
 $r$-RSB and nondegenerate, and $\supp\mu_{p,\lambda}\subset[0,a_*]$.

 \item If $\lambda=\lambda_{\mathrm c,p}$, then $\mu_{p,\lambda}$ is
 $r$-RSB and $\supp\mu_{p,\lambda}\subset[0,a_*]$.  Its first-variation
 function is nonpositive and vanishes exactly on $\supp\mu_{p,\lambda}$ and
 at one further point $Q\in(a_*,1)$.  This point satisfies, with $m=1$,
 \begin{equation}\label{eq:gardner-location}
  C^{-1}p^{-1/2}e^{-m^2\lambda p/2}
  \leq
  1-Q
  \leq
  Cp^{-1/2}e^{-m^2\lambda p/2}.
 \end{equation}
 In particular, condition~(i) of \zcref{def:nondegenerate}
 fails, so $\mu_{p,\lambda}$ is degenerate.

 \item If $\lambda>\lambda_{\mathrm c,p}$, then
 \[
  \mu_{p,\lambda}=m\nu+(1-m)\delta_Q,
 \]
 where $m\in(0,1)$, $\nu$ is an $r$-RSB probability measure with
 $\supp\nu\subset[0,a_*]$, and $Q\in(a_*,1)$ satisfies
 \eqref{eq:gardner-location} with this value of $m$.  Thus $\mu_{p,\lambda}$ is $(r+1)$-RSB.  It
 is also nondegenerate.
\end{enumerate}
Moreover, $\lambda_{\mathrm c,p}\to\lambda_{\mathrm c}$ as $p\to\infty$.
By~(i)--(iii), the Parisi measure of $\Xi_{p,\lambda}$ is $r$-RSB for
$\lambda\in[\lambda_{\mathrm c}-\eta,\lambda_{\mathrm c,p}]$ and
$(r+1)$-RSB for $\lambda\in(\lambda_{\mathrm c,p},\lambda_{\mathrm c}+\eta]$,
so $\lambda_{\mathrm c,p}$ is the unique point of
$[\lambda_{\mathrm c}-\eta,\lambda_{\mathrm c}+\eta]$ at which
the number of RSB steps changes.
\end{theorem}

\begin{proof}
Take $\eta$ and $a_*$ as in (S4) and (S3); then
$[\lambda_{\mathrm c}-\eta,\lambda_{\mathrm c}+\eta]\subset\Lambda$.  Let
$p$ be large enough that all preceding conclusions hold and
$\lambda_{\mathrm c,p}$ lies in $(\lambda_{\mathrm c}-\eta,\lambda_{\mathrm c}+\eta)$.
Fix $\lambda\in[\lambda_{\mathrm c}-\eta,\lambda_{\mathrm c}+\eta]$ and set $m=m_{p,\lambda}$, $z=z_p(\lambda,m)$,
$Q=Q_p(z,\lambda,m)$, and $\mu=\mu_p(\lambda,m)$.
By \zcref{lem:stationary-candidates}, $f(Q)=g_p(\lambda,m)$.

By \zcref{prop:scalar-bifurcation}, if $\lambda\leq\lambda_{\mathrm c,p}$,
the chosen mass is $m=1$;
the value $f(Q)$ is negative below the critical coefficient and zero at
it.  If $\lambda>\lambda_{\mathrm c,p}$, the chosen mass is interior and
$f(Q)=0$.  By \zcref{cor:global-sign}, in every case $f\leq0$ and its
zero set contains the support of $\mu$.  The first-variation criterion in
\zcref{prop:first-variation-criterion} therefore identifies $\mu$ as the
unique Parisi measure $\mu_{p,\lambda}$.

The measure $\nu_z$ has exactly $r+1$ atoms, all in $[0,a_*]$.  For $m<1$,
$Q>a_*$ is a further atom with positive mass; for $m=1$, the point $Q$ has
zero mass.  This proves the stated support counts.  The scale of $Q$ follows
from \zcref{prop:uniform-reduction}.
Away from the critical coefficient, the zero set is exactly the support,
and strict curvature at every positive atom follows from
\zcref{lem:lower-geometry,prop:high-overlap-geometry}.  Thus these measures
are nondegenerate.  At the critical coefficient, there is one extra
zero without mass, so condition~(i) of \zcref{def:nondegenerate} fails.
Finally, the convergence of $\lambda_{\mathrm c,p}$ and its uniqueness in
$[\lambda_{\mathrm c}-\eta,\lambda_{\mathrm c}+\eta]$ follow from the scalar minimization and the
three cases just proved.
\end{proof}

\section{Iteration and open phases}\label{sec:consequences}

Starting from the replica-symmetric model of
\zcref{lem:replica-symmetric-seed}, each application of
\zcref{thm:gardner} above the critical coefficient adds one atom and
preserves nondegeneracy.
Local stability then gives an open phase around every constructed model.

\subsection{The iteration}\label{s:iteration}

\begin{proposition}\label{prop:iteration}
For every integer $k\geq0$, there exist even integers
$4=p_0<p_1<\cdots<p_k$ and positive coefficients $c_0,\ldots,c_k$ such that
$\sum_{j=0}^kc_jq^{p_j}$ has a nondegenerate $k$-RSB Parisi measure.
\end{proposition}

\begin{proof}
We argue by induction on $k$.  Every covariance of the form
$\sum_{j=0}^kc_jq^{p_j}$ with $c_j>0$ and even $p_j\geq4$ is admissible,
because its second derivative vanishes at $0$.

For $k=0$, take the covariance of
\zcref{lem:replica-symmetric-seed}.

Now suppose that $\Xi^{(k)}=\sum_{j=0}^kc_jq^{p_j}$ has a nondegenerate
$k$-RSB Parisi measure, and let $\lambda_{\mathrm c}^{(k)}$, $\eta^{(k)}$,
and $\lambda_{\mathrm c,p}^{(k)}$ be the quantities of \zcref{thm:gardner}
for $\Xi^{(k)}$.  Fix an even integer $p_{k+1}>p_k$ so large that the
conclusions of \zcref{thm:gardner} hold for $p=p_{k+1}$, and set
$c_{k+1}=\lambda_{\mathrm c}^{(k)}+\eta^{(k)}$.  Then $c_{k+1}>0$, and
$c_{k+1}>\lambda_{\mathrm c,p_{k+1}}^{(k)}$ because
$\lambda_{\mathrm c,p_{k+1}}^{(k)}<\lambda_{\mathrm c}^{(k)}+\eta^{(k)}$.
By \zcref{thm:gardner}(iii), the covariance
$\Xi^{(k+1)}=\Xi^{(k)}+c_{k+1}q^{p_{k+1}}$ has a nondegenerate
$(k+1)$-RSB Parisi measure.
\end{proof}

\subsection{Proofs of the main results}\label{s:mainresults}

\begin{proof}[Proof of \zcref{thm:main}]
Fix $k\geq1$, and let $4=p_0<\cdots<p_k$ and $\mathbf c^*=(c_0^*,\ldots,c_k^*)$
be given by \zcref{prop:iteration}.  For these exponents, let
$\mathcal O_k$ be the set of $\mathbf c\in(0,\infty)^{k+1}$ for which
$\Xi_{\mathbf c}=\sum_{j=0}^kc_jq^{p_j}$ has a nondegenerate $k$-RSB Parisi
measure.  Since all exponents are at least $4$, every
$\mathbf c\in(0,\infty)^{k+1}$ gives an admissible covariance.  By construction,
$\mathbf c^*\in\mathcal O_k$, and $\mathcal O_k$ is open by
\zcref{prop:stability}, applied at each of its points with
$\mathscr P=\{p_0,\ldots,p_k\}$ and $c_{p_j}=c_j$.
\end{proof}

\begin{proof}[Proof of \zcref{thm:gardner-main}]
By \zcref{cor:positive-entropy}, we have $s_\Xi>0$.  The remaining assertions
follow from \zcref{thm:gardner}, in which $\lambda_{\mathrm c}=2s_\Xi$:
cases~(i) and~(ii) give the $r$-RSB phase for $\lambda\leq\lambda_{\mathrm c,p}$,
case~(iii) gives the $(r+1)$-RSB phase for $\lambda>\lambda_{\mathrm c,p}$,
and the final two assertions give the convergence and the uniqueness of
$\lambda_{\mathrm c,p}$.
\end{proof}

\subsection{The mass of the new atom}\label{s:new-atom-mass}

Taylor expansion of the equation $g_p(\lambda,m)=0$ of
\zcref{prop:scalar-bifurcation} at
$(\lambda_{\mathrm c,p},1)$ shows that the mass of the new atom grows
linearly to first order in the distance above the transition.  The
following corollary records the uniform estimate, together with the
location of the new atom.

\begin{corollary}\label{cor:gardner-asymptotics}
Let $\eta$ be as in \zcref{thm:gardner}, and recall that
$\kappa_\Xi=\mathfrak p''(1)+2s_\Xi>0$ (see \eqref{eq:J1}).
For every sufficiently large even $p$, let
$g_p$ be the function of \zcref{prop:stationary-branches}(ii), and define
\[
 \kappa_p
 =
 \frac{g_p^m}{2g_p^\lambda},
 \qquad
 g_p^\lambda=\partial_\lambda g_p(\lambda_{\mathrm c,p},1),
 \qquad
 g_p^m=\partial_mg_p(\lambda_{\mathrm c,p},1).
\]
Then $\kappa_p>0$ and
$\kappa_p\to\kappa_\Xi$.  Moreover, there exists $C<\infty$ such that, for
every
sufficiently large even $p$ and every
$\lambda=\lambda_{\mathrm c,p}+\delta$ with $0<\delta\leq\eta/2$, the mass
$1-m$ of the new atom in the Parisi measure
$\mu_{p,\lambda}=m\nu+(1-m)\delta_Q$ of \zcref{thm:gardner}(iii) satisfies
\begin{equation}\label{eq:gardner-uniform-mass-law}
 \left|1-m-\frac{\delta}{2\kappa_p}\right|
 \leq C\delta^2,
\end{equation}
and its location $Q$ satisfies
\begin{equation}\label{eq:gardner-location-log}
 \Bigl|\log\frac1{1-Q}-\frac{m^2\lambda p}2-\frac12\log p\Bigr|\leq C.
\end{equation}
\end{corollary}

\begin{proof}
By \zcref{prop:stationary-branches}(ii), $g_p\to J_\lambda'$ in
$C^2(\Lambda\times[m_0,1])$, with derivatives at $m=1$ from the left; in
particular, the derivatives of $g_p$ of order at most two are bounded
uniformly in $p$.  Since $\partial_mJ_\lambda'=J_\lambda''$ is continuous and
does not depend on $\lambda$, and $\lambda_{\mathrm c,p}\to\lambda_{\mathrm c}$
by \zcref{prop:scalar-bifurcation}, \eqref{eq:J1} gives
$g_p^m\to\kappa_\Xi$, and \eqref{eq:J-derivative} gives
$g_p^\lambda\to1/2$.  Both derivatives are positive, so
$\kappa_p>0$ and $\kappa_p\to\kappa_\Xi$.

For large $p$, $|\lambda_{\mathrm c,p}-\lambda_{\mathrm c}|<\eta/2$;
thus every $\lambda=\lambda_{\mathrm c,p}+\delta$ in the statement lies
in $[\lambda_{\mathrm c}-\eta,\lambda_{\mathrm c}+\eta]$.  By the proof of
\zcref{thm:gardner},
$m=m_{p,\lambda}$, and by \zcref{prop:scalar-bifurcation}, $m<1$ and
$g_p(\lambda,m)=0$.  Set
$\varepsilon=1-m$.  The bound
$\partial_mg_p\geq\kappa_{\mathrm m}$ of \zcref{prop:scalar-bifurcation}
and the uniform bound $\partial_\lambda g_p\leq C$ give the preliminary estimate
\[
 \kappa_{\mathrm m}\varepsilon
 \leq g_p(\lambda,1)-g_p(\lambda,m)
 =g_p(\lambda,1)-g_p(\lambda_{\mathrm c,p},1)
 \leq C\delta.
\]
The second derivatives of $g_p$ are uniformly bounded.  Taylor expansion
along the segment in $\Lambda\times[m_0,1]$ from
$(\lambda_{\mathrm c,p},1)$ to $(\lambda,1-\varepsilon)$ yields
\[
 0=g_p(\lambda,1-\varepsilon)
 =g_p^\lambda\delta-g_p^m\varepsilon+R,
 \qquad |R|\leq C(\delta+\varepsilon)^2\leq C\delta^2.
\]
Dividing by $g_p^m\geq\kappa_{\mathrm m}$ and using
$g_p^\lambda/g_p^m=1/(2\kappa_p)$ proves
\eqref{eq:gardner-uniform-mass-law}.  Finally, \eqref{eq:gardner-location-log}
follows by taking logarithms in \eqref{eq:gardner-location}.
\end{proof}

\begin{remark}\label{rem:gardner-asymptotics}
By \eqref{eq:gardner-uniform-mass-law}, $1/(2\kappa_p)$ is the right
derivative at $\lambda_{\mathrm c,p}$ of the mass $1-m_{p,\lambda}$ of the
new atom of the Parisi measure $\mu_{p,\lambda}$.  Since $\lambda_{\mathrm c,p}$ is
the unique point of $[\lambda_{\mathrm c}-\eta,\lambda_{\mathrm c}+\eta]$ at which
the number of RSB steps changes, $\kappa_p$ depends only on
$\Xi$ and $p$, and not on the choices made in the proof.  If $p_j\to\infty$ through even integers and
$\lambda_j=\lambda_{\mathrm c,p_j}+\delta_j$ with $\delta_j>0$ and
$\delta_j\to0$, then \eqref{eq:gardner-uniform-mass-law} and
$\kappa_{p_j}\to\kappa_\Xi$ give
$1-m_{p_j,\lambda_j}=\delta_j/(2\kappa_\Xi)+o(\delta_j)$.  Thus, to first
order, the mass of the new atom grows at a rate determined by the
unperturbed model alone.
\end{remark}

\appendix

\section{Parisi recursions and regularity}\label{app:recursion}

This appendix proves the recursion, curvature, continuity,
first-variation, and entropy results used in the main text.  The general terminal functions used here allow
us to compare the recursion below a cutoff for a candidate minimizer with
the lower problem of \zcref{s:tempmass}.

Let $\mathsf p_t(x,y)=(2\pi t)^{-1/2}\exp(-(y-x)^2/(2t))$ be the
Gaussian kernel.  The Cole--Hopf transform and heat semigroup are those
of \eqref{eq:T-def}.  An atom at $q=1$ affects the PDE only through the
values of the cumulative mass on $[0,1)$, so it requires no separate
treatment in the recursion.

\subsection{One Cole--Hopf step and finite recursions}\label{s:recursions}

The transition kernel for one Cole--Hopf step is a Doob $h$-transform
\cite[Lemma~8.5 and Corollary~8.7]{JagannathTobasco2017}.  We record its
properties and the smooth dependence of finite recursions on their data.

For a $C^3$ function $g$ on $\R$ whose derivatives of orders one to three
are bounded and uniformly continuous, write
$\|g\|_3=|g(0)|+\sum_{j=1}^3\|g^{(j)}\|_\infty$, and let $\mathcal G$ be the
space of such functions, normed by $\|\cdot\|_3$.  The terminal functions
used below, such as $\log\cosh$ and $\psi_m(x)=m^{-1}\log\cosh(mx)$, belong
to $\mathcal G$ and are one-Lipschitz.  For $a\geq0$, $V>0$, and
$g\in\mathcal G$, define the probability kernel
\begin{equation}\label{eq:finite-transition-kernel}
 K_{a,V,g}(x,\dd y)
 =\mathsf p_V(x,y)\exp\bigl\{a\bigl[g(y)-T_{a,V}g(x)\bigr]\bigr\}\dd y,
\end{equation}
and set $K_{a,0,g}(x,\cdot)=\delta_x$.  For $a=0$, the kernel $K_{a,V,g}$
is the Gaussian kernel; for $a>0$, it is the Gaussian kernel tilted by
$e^{ag}$.

\begin{lemma}\label{lem:one-step}
Let $a,V\geq0$, let $g\in\mathcal G$ be one-Lipschitz, and write
$T=T_{a,V}$ and $K=K_{a,V,g}$.
\begin{enumerate}[label=\textup{(\roman*)}]
 \item We have $T(g+c)=Tg+c$ and $K_{a,V,g+c}=K$ for every constant $c$.
 Moreover, $T_{a,s}T_{a,t}=T_{a,s+t}$ for $s,t\geq0$.
 \item The function $Tg$ is one-Lipschitz and belongs to $\mathcal G$, and
 \[
  (Tg)'(x)=\int K(x,\dd y)g'(y),
  \qquad
  (Tg)''(x)=\int K(x,\dd y)g''(y)+a\Var_{K(x,\cdot)}(g').
 \]
 \item Let $w(s,\cdot)=T_{a,V-s}g$ for $0\leq s\leq V$, and let $\xi$ solve
 $\dd\xi_s=aw_x(s,\xi_s)\dd s+\dd B_s$ with $\xi_0=x$.  Then $\xi_V$ has
 law $K(x,\cdot)$.
 \item Let $\mathcal U=\{h\in\mathcal G:\|h'\|_\infty<2\}$.  The map
 $(a,V,h)\mapsto T_{a,V}h$ is $C^\infty$ from $(0,\infty)^2\times\mathcal U$
 into $(\mathcal G,\|\cdot\|_3)$ in the Fr\'echet sense.  For
 $0<a_-\leq a_+$, $0<V_-\leq V_+$, and $M<\infty$, its derivatives are
 bounded on $[a_-,a_+]\times[V_-,V_+]\times\{h\in\mathcal U:\|h\|_3\leq M\}$
 by constants that depend only on $a_\pm$, $V_\pm$, $M$, and the order.  If
 $a=0$ or $V=0$, the map $h\mapsto T_{a,V}h$ is linear and bounded on
 $\mathcal G$.  In all cases, the derivative of $h\mapsto T_{a,V}h$ at $g$
 in a direction $\dot g\in\mathcal G$ is $\int K(\cdot,\dd y)\dot g(y)$.
\end{enumerate}
\end{lemma}

\begin{proof}
(i) The first two identities follow from \eqref{eq:T-def} and
\eqref{eq:finite-transition-kernel}, and the third
from $\mathsf H_s\mathsf H_t=\mathsf H_{s+t}$.  (ii) For $V=0$ there is
nothing to prove, and for $a=0$ the operator $T$ is linear.  For $a,V>0$,
we have
$K(x,\dd y)=\mathsf p_V(x,y)e^{ag(y)}\dd y/\mathsf H_V(e^{ag})(x)$.  We
write $\E_K$, $\Var_K$, and $\operatorname{Cov}_K$ for the mean, variance,
and covariance under $K(x,\cdot)$, and $\Delta_xg(y)=g(x+y)-g(x)$, so that
$\E_Kh=\E[h(x+\sqrt VZ)e^{a\Delta_xg(\sqrt VZ)}]/\E[e^{a\Delta_xg(\sqrt VZ)}]$.
Differentiating this expression in $x$ gives
$(\E_Kh)'=\E_Kh'+a\operatorname{Cov}_K(h,g')$, and differentiating
$Tg=a^{-1}\log\mathsf H_V(e^{ag})$ once and twice in $x$ gives the two
displayed formulas.  The first one shows that $(Tg)'$ is an average of
$g'$, and in particular $|(Tg)'|\leq1$.  Differentiating once more, we find
\[
 (Tg)'''=\E_Kg'''+3a\operatorname{Cov}_K(g',g'')
 +a^2\E_K\bigl[(g'-\E_Kg')^3\bigr].
\]
Since $|g'|\leq1$, the centered variable $Y=g'-\E_Kg'$ satisfies
$|Y|\leq2$ and $\E_KY^2\leq\E_Kg'^2\leq1$.  Hence $\Var_K(g')\leq1$,
$|\operatorname{Cov}_K(g',g'')|=|\E_K[Yg'']|\leq\|g''\|_\infty$, and
$\E_K|Y|^3\leq2\E_KY^2\leq2$, so
\begin{equation}\label{eq:one-step-spatial-bounds}
 \begin{aligned}
  |(Tg)''|&\leq\|g''\|_\infty+a,\\
  |(Tg)'''|&\leq\|g'''\|_\infty+3a\|g''\|_\infty+2a^2.
 \end{aligned}
\end{equation}
The derivatives of
$Tg$ of orders one to three are uniformly continuous, because the tilted
averages above are: their integrands are uniformly continuous in $x$, and
$|\Delta_xg(y)-\Delta_{x'}g(y)|\leq\|g''\|_\infty|y||x-x'|$.  We conclude that $Tg\in\mathcal G$.  (iii)
By (ii), we have $|w_x|\leq1$ and $|w_{xx}|\leq\|g''\|_\infty+a$.  Then
the drift is bounded and Lipschitz, and $\xi$ exists and is unique.  If $a=0$, then
$\xi$ is a Brownian motion, and the claim is clear.  If $a>0$, the function
$\varphi=e^{aw}=\mathsf H_{V-s}(e^{ag})$ is a positive solution of the
backward heat equation with bounded logarithmic derivative $aw_x$.  For a
Brownian motion $\tilde B$ started at $x$, the process
$\varphi(s,\tilde B_s)/\varphi(0,x)$ is therefore the stochastic
exponential of a bounded integrand.  By
Girsanov's theorem, the law of $\xi$ is that of $\tilde B$ weighted by
$\varphi(V,\tilde B_V)/\varphi(0,x)$.  Then $\xi_V$ has law
$\varphi(V,y)\mathsf p_V(x,y)\varphi(0,x)^{-1}\dd y=K(x,\dd y)$.

(iv) The set $\mathcal U$ is open in $(\mathcal G,\|\cdot\|_3)$ and contains
the one-Lipschitz functions.  We work on $\mathcal U$ because derivatives
in the terminal function involve $h+t\dot h$, which need not be
one-Lipschitz.  For $a,V>0$ and $h\in\mathcal U$, we rewrite \eqref{eq:T-def} as
\begin{equation}\label{eq:centered-cole-hopf}
 T_{a,V}h(x)=h(x)+\frac1a\log
 \int e^{a\Delta_xh(y)}\mathsf p_V(0,y)\dd y.
\end{equation}
Since $|\Delta_xh(y)|\leq2|y|$, the exponential in the integrand is at
most $e^{2a|y|}$.  Then the integral is at most
$\E e^{2a\sqrt V|Z|}\leq2e^{2a^2V}$, and, by Jensen's inequality and
$\E|\sqrt VZ|\leq\sqrt V$, it is at least $e^{-2a\sqrt V}$.  We may take
derivatives of \eqref{eq:centered-cole-hopf} in $a$, in $V$, and in the
terminal function along directions $\dot h\in\mathcal G$, together with up
to three spatial derivatives, under the integral.  The $V$-derivatives fall
on $\mathsf p_V(0,y)$ and produce polynomials in $y$ whose coefficients are
polynomials in $V^{-1}$; these coefficients are bounded because
$V\geq V_-$.  The derivatives in $a$ and $h$ produce factors
$\Delta_xh(y)$ and $\Delta_x\dot h(y)$, of size at most $2|y|$ and
$\|\dot h\|_3|y|$, respectively.  The spatial derivatives produce factors
$h^{(j)}(x+y)-h^{(j)}(x)$ and $\dot h^{(j)}(x+y)-\dot h^{(j)}(x)$ with
$1\leq j\leq3$, which are bounded by $2\|h\|_3$ and $2\|\dot h\|_3$.  For
$(a,V)\in[a_-,a_+]\times[V_-,V_+]$ and $\|h\|_3\leq M$, all integrands are
therefore dominated, uniformly in $x$, by the product of the norms of the
directions and $C(1+|y|)^A\exp(-y^2/(2V_+)+2a_+|y|)$, where $A$ and $C$ depend
only on $a_\pm$, $V_\pm$, and $M$.  The resulting multilinear forms in the
directions are bounded in $\|\cdot\|_3$.  When $(a,V,h)$ changes by
$(\Delta a,\Delta V,\Delta h)$ within this set, their integrands change by
at most $C(1+|y|)^A(|\Delta a|+|\Delta V|+\|\Delta h\|_3)$ times this
envelope.  The forms therefore depend continuously on $(a,V,h)$ in
operator norm.  They take values in $\mathcal G$, and so does $T_{a,V}h$:
each third spatial derivative is a tilted Gaussian average of products of
derivatives of order at most three of $h$ and of the directions, and it is
uniformly continuous in $x$ as in the proof of~(ii).  Since the
G\^ateaux derivatives of all orders exist and are continuous in operator
norm, the map is $C^\infty$ in the Fr\'echet sense, with the asserted
bounds.  If $a=0$ or $V=0$, then $h\mapsto T_{a,V}h$
is $\mathsf H_V$ or the identity, which is linear and bounded on
$\mathcal G$.  Finally, differentiating \eqref{eq:T-def} at $g$ in the
direction $\dot g$, we obtain $\int K(\cdot,\dd y)\dot g(y)$.
\end{proof}

Part~(iii) says that $K_{a,V,g}$ is the transition kernel of the optimal
diffusion across one step of constant level, measured in the variance
variable.

Given levels $a_i\geq0$, variances $V_i\geq0$, $0\leq i<L$, and a
one-Lipschitz terminal function $g\in\mathcal G$, the \emph{backward
recursion} is given by $u_L=g$ and $u_i=T_{a_i,V_i}u_{i+1}$, and the
\emph{forward
chain} is the Markov chain with $x_0=0$ and
$x_{k+1}\sim K_{a_k,V_k,u_{k+1}}(x_k,\cdot)$.  By \zcref{lem:one-step}(iv),
each step is Fr\'echet differentiable in its terminal function, with
derivative given by integration against its kernel.  Applying the chain
rule from $i=L-1$ down to $i=0$, we find that the derivative of $u_0(x)$
with respect to the terminal function in the direction $\dot g\in\mathcal G$
is $\E[\dot g(x_L)\mid x_0=x]$, where the chain is started at $x$.  In
particular, the derivative of $u_0(0)$ is
\begin{equation}\label{eq:terminal-direction}
 \E\dot g(x_L).
\end{equation}
In \zcref{s:truncated}, we identify the forward chain with the optimal
diffusion evaluated at the endpoints of the intervals used in the recursion.
The next lemma, proved in
\zcref{s:smooth-dependence}, shows that the backward recursion depends
smoothly on its data.  The lower bound $V_->0$ in the lemma allows derivatives in the
variance to fall on the Gaussian kernel rather than on the terminal
function, and the lower bound $a_->0$ simplifies the estimates in the
proof.

\begin{lemma}\label{lem:cole-hopf-smooth}
Fix integers $L\geq1$ and $\ell\geq0$, and constants $0<a_-\leq a_+$,
$0<V_-\leq V_+$, and $C_0<\infty$.  Let $D\subset\R^d$ be the compact closure of a convex open set; a
function on $D$ is $C^\ell$ if its derivatives through order $\ell$ on the
interior of $D$ extend continuously to $D$.  For
$\theta\in D$, let levels
$a_i(\theta)$ and variances $V_i(\theta)$, $0\leq i<L$, be $C^\ell$, and
let $\theta\mapsto g_\theta$ be a $C^\ell$ map from $D$ into
$(\mathcal G,\|\cdot\|_3)$ whose values are one-Lipschitz.  Assume that, on
$D$ and for every multi-index $\mathbf j$ with $|\mathbf j|\leq\ell$,
\[
 a_-\leq a_i\leq a_+,\qquad
 V_-\leq V_i\leq V_+,\qquad
 |\partial_\theta^{\mathbf j}a_i|+|\partial_\theta^{\mathbf j}V_i|
 +\|\partial_\theta^{\mathbf j}g_\theta\|_3\leq C_0.
\]
\begin{enumerate}[label=\textup{(\roman*)}]
 \item Each $u_i$ is one-Lipschitz, and its mixed $\theta$-derivatives
 through order $\ell$ are continuous in $\|\cdot\|_3$.  They are bounded
 on $D$ by a constant $C$ that depends only on $L$, $\ell$, $d$, $a_\pm$,
 $V_\pm$, and $C_0$, and not on $D$.  In particular, $\theta\mapsto u_0(0)$ is $C^\ell$.
 \item Suppose that a second family of data satisfies the same
 hypotheses.  Suppose also that the two families, together with their mixed
 derivatives through order $\ell$, differ by at most $\varepsilon$ on $D$,
 with terminal functions compared in $\|\cdot\|_3$.  Then the outputs $u_0$ and
 $\tilde u_0$ of the two recursions satisfy
 $\sup_D\|\partial_\theta^{\mathbf j}(u_0-\tilde u_0)\|_3\leq C\varepsilon$
 for $|\mathbf j|\leq\ell$, with $C$ depending only on the quantities
 listed in~(i).
\end{enumerate}
\end{lemma}

\subsection{Smooth dependence on recursion data}\label{s:smooth-dependence}

\begin{proof}[Proof of \zcref{lem:cole-hopf-smooth}]
By \zcref{lem:one-step}(ii), one step preserves the one-Lipschitz bound.
By the bounds on the integral in the proof of \zcref{lem:one-step}(iv),
$|T_{a,V}g-g|$ is bounded by a constant that depends only on $a_\pm$ and
$V_+$, and, by \eqref{eq:one-step-spatial-bounds},
$\|T_{a,V}g\|_3\leq(1+3a_+)\|g\|_3+C$.  By induction on $L-i$, we have
$\|u_i\|_3\leq M$ for every $i$, with $M$ depending only on $L$, $a_\pm$,
$V_+$, and $C_0$.  By \zcref{lem:one-step}(iv), the derivatives of one
step along the recursion are then bounded by constants that depend only on
the quantities listed in~(i).  The chain rule and induction on $L-i$ give~(i).  At each interior point
$\theta$, the chain rule expresses $\partial_\theta^{\mathbf j}u_i(\theta)$
through derivatives of one step at $(a_i(\theta),V_i(\theta),u_{i+1}(\theta))$
and derivatives of the data at $\theta$, so the bound involves only the
constants in the hypotheses.  It extends to $D$ by continuity, and the
constant does not depend on $D$.

For~(ii), let $a_i^s$, $V_i^s$, and $g^s$, $0\leq s\leq1$, be the linear
interpolations between the two families, so that $s=0$ gives the first
family and $s=1$ the second.  They satisfy the hypotheses of the lemma
with the same constants, and we write $u_i^s$ for their backward functions.
Applying~(i) to the interpolated family as a family in $(\theta,s)$ would
require $\ell+1$ derivatives of the data.  Instead, we use that one step is
$C^\infty$ in $(a,V,g)$ by \zcref{lem:one-step}(iv) and that the data are affine in $s$.  The
functions $\dot u_i^s=\partial_su_i^s$ satisfy $\dot u_L^s=\tilde g-g$ and
\[
 \dot u_i^s
 =\partial_aT\,(\tilde a_i-a_i)
 +\partial_VT\,(\tilde V_i-V_i)
 +D_gT\,[\dot u_{i+1}^s],
\]
where $\partial_aT$, $\partial_VT$, and the derivative $D_gT$ in the
terminal function are evaluated at $(a_i^s,V_i^s,u_{i+1}^s)$, and tildes
denote the second family.  A $\theta$-derivative of order at most $\ell$ of
the right side is a sum of products of a derivative of one step, of
$\theta$-derivatives of $a_i^s$, $V_i^s$, and $u_{i+1}^s$ of order at most
$\ell$, and of exactly one $\theta$-derivative of order at most $\ell$ of a
difference of the data or of $\dot u_{i+1}^s$.  By induction on $L-i$,
the derivatives $\partial_\theta^{\mathbf j}\dot u_i^s$ with
$|\mathbf j|\leq\ell$ exist and are continuous in $(\theta,s)$ in
$\|\cdot\|_3$: the right side is a $C^\infty$ function of
$(a_i^s,V_i^s,u_{i+1}^s)$, linear in
$(\tilde a_i-a_i,\tilde V_i-V_i,\dot u_{i+1}^s)$, and each of these has this
property.  By~(i) and induction on $L-i$, we find that
$\|\partial_\theta^{\mathbf j}\dot u_i^s\|_3\leq C\varepsilon$ for
$|\mathbf j|\leq\ell$, uniformly in $s$.  Since
$\partial_\theta^{\mathbf j}(\tilde u_0-u_0)
=\int_0^1\partial_\theta^{\mathbf j}\dot u_0^s\dd s$, this proves~(ii).
\end{proof}

\begin{proof}[Proof of \zcref{prop:finite-step-regularity}]
Set $q_{r+1}=1$, $a_i=\nu_z([0,q_i])$, and
$V_i=\Xi_{\mathbf c}'(q_{i+1})-\Xi_{\mathbf c}'(q_i)$.  By
\eqref{eq:cole-hopf-step}, $u_{\nu_z}(0,\cdot)$ is the output $u_0$ of the
backward recursion with these levels and variances and with terminal
function $\log\cosh$.  Fix a closed ball
$D\subset\mathcal C_{\mathscr P}\times\mathcal Z_r$.  On $D$, the levels
lie in $[\min_Dw_0,1]$, and, by compactness, the variances lie in a compact
subset of $(0,\infty)$, because the support gaps are positive and
$\Xi_{\mathbf c}''>0$ on $(0,1]$.  The levels are affine and the variances
polynomial in $(\mathbf c,z)$, so all their derivatives are bounded on $D$,
and the terminal function does not depend on $(\mathbf c,z)$.  Then
\zcref{lem:cole-hopf-smooth} applies for every $\ell$, and $u_{\nu_z}(0,0)$
is smooth.  Integrating the deterministic part of the Parisi functional, we
find that
\[
 \Pp_{\Xi_{\mathbf c}}(\nu_z)
 =
 \log2+u_{\nu_z}(0,0)
 -\frac12\sum_{i=0}^ra_i
 \bigl[q\Xi_{\mathbf c}'(q)-\Xi_{\mathbf c}(q)\bigr]_{q_i}^{q_{i+1}},
\]
whose last term is also smooth.
\end{proof}

\subsection{Curvature and martingale identities}\label{s:curvature-proof}

\begin{proof}[Proof of \zcref{lem:curvature}]
Let $0\leq a\leq1$ and $V\geq0$, and let $g\in\mathcal G$ satisfy
$|g'|\leq1$ and $0<g''\leq1-g'^2$.  Set $\tilde g=T_{a,V}g$, and write
$\E_K$ and $\Var_K$ for the mean and variance under $K_{a,V,g}(x,\cdot)$.
By \zcref{lem:one-step}(ii), $\tilde g'=\E_K g'$ and
$\tilde g''=\E_K g''+a\Var_K(g')$.  Then $\tilde g''\geq\E_Kg''>0$,
$|\tilde g'|=|\E_K g'|\leq1$, and
\[
 (1-\tilde g'^2)-\tilde g''
 =\E_K\bigl[(1-g'^2)-g''\bigr]+(1-a)\Var_K(g')\geq0.
\]
A Cole--Hopf step therefore preserves the
bounds $|g'|\leq1$ and $0<g''\leq1-g'^2$.
Since the terminal function $g=\log\cosh$ satisfies $g'=\tanh$ and
$g''=\sech^2=1-g'^2$, we may apply this observation successively to the
Cole--Hopf steps \eqref{eq:cole-hopf-step}, which proves
\eqref{eq:curvature}.  The averages under $K_{a,V,g}(x,\cdot)$ can be
written as tilted Gaussian averages, as in the proof of
\zcref{lem:one-step}(ii); this representation also holds for $V=0$.  For
a bounded continuous function $h$ and $0\leq V\leq V_*:=\Xi'(1)$, the integrands of the
numerator and the denominator of the tilted average of $h$ are continuous
in $(V,x)$ and bounded by $\|h\|_\infty e^{a\sqrt{V_*}|Z|}$ and
$e^{a\sqrt{V_*}|Z|}$, and the denominator is at least $e^{-a\sqrt{V_*}}$ by
Jensen's inequality.  By dominated convergence, $\tilde g'$ and $\tilde g''$
therefore depend continuously on $(V,x)$.  Since the variances
$\Xi'(q_+)-\Xi'(q)$ in \eqref{eq:cole-hopf-step} depend continuously on $q$,
$u_{\mu,x}$ and $u_{\mu,xx}$ are bounded and continuous on
$[0,1]\times\R$.

Between consecutive atoms of $\mu$, the variance of each step is positive,
so $u=u_\mu$ is smooth in $x$ with bounded derivatives, and $\partial_qu_x$
is continuous by \eqref{eq:parisi-pde}; this justifies It\^o's formula
below.  We differentiate \eqref{eq:parisi-pde} in $x$ and find that
$\partial_qu_x=-\frac12\Xi''(u_{xxx}+2\alpha_\mu u_xu_{xx})$.  By It\^o's
formula, we have
\[
 \dd U_q
 =\Bigl(\partial_qu_x+\Xi''\alpha_\mu u_xu_{xx}+\frac12\Xi''u_{xxx}\Bigr)
 \dd q+\sqrt{\Xi''}u_{xx}\dd B_q
 =\sqrt{\Xi''(q)}u_{\mu,xx}(q,X_q)\dd B_q,
\]
with all derivatives of $u_\mu$ evaluated at $(q,X_q)$.  Letting the
endpoints of these intervals tend to the atoms and using the continuity of
$U$, we obtain $U_q=U_0+\int_0^q\sqrt{\Xi''}u_{\mu,xx}(s,X_s)\dd B_s$ on
$[0,1]$.  Since the integrand is bounded, $U$ is a martingale.
The symmetry $\mathsf p_V(-x,-y)=\mathsf p_V(x,y)$ shows that each
Cole--Hopf step preserves evenness.  Hence the Parisi solution is even
in $x$, since its terminal function is $\log\cosh$.  Then $U_0=u_{\mu,x}(0,0)=0$, and $U_1=\tanh X_1$.  Applying It\^o's formula to
$U_q^2$, we find \eqref{eq:Gamma-derivative} between atoms.  Since $X$ has continuous
paths and $u_{\mu,xx}$ is bounded and continuous, the right side of
\eqref{eq:Gamma-derivative} is continuous in $q$ by bounded convergence.
This implies that $\Gamma\in C^1([0,1])$ and that
\eqref{eq:Gamma-derivative} holds on $[0,1]$.  Finally, by
\eqref{eq:curvature}, we have
$\Gamma_2=\E u_{\mu,xx}^2\leq\E(1-u_{\mu,x}^2)^2\leq\E(1-u_{\mu,x}^2)=1-\Gamma$,
with derivatives evaluated at $(q,X_q)$, which together with
\eqref{eq:Gamma-derivative} gives the first bound in
\eqref{eq:Gamma-curvature}.  Since $\Gamma(0)=0$,
$\Xi'(0)=0$, and $\Gamma'\leq\Xi''$, integration yields $\Gamma\leq\Xi'$.
\end{proof}

\subsection{Truncated data and continuity}\label{s:truncated}

In \zcref{sec:high-degree-perturbation}, we compare the Parisi solution of a
candidate minimizer below a cutoff $q_{\mathrm c}<1$ with that of the lower
problem of \zcref{s:tempmass}, whose terminal function at $q_{\mathrm c}$ is
$\psi_m$ rather than $\log\cosh$.  We
therefore extend the definitions above to finite recursions with a general
terminal function.

Fix $r\geq0$ and $q_{\mathrm c}\in(0,1]$.  By \emph{truncated Parisi
data} we mean $\mathfrak d=(\Psi,(q_i,a_i)_{i=0}^r,g)$, consisting of a
covariance $\Psi\in C^3([0,q_{\mathrm c}])$ with $\Psi''>0$ on
$(0,q_{\mathrm c}]$, points
$0=q_0<\cdots<q_r<q_{r+1}=q_{\mathrm c}$ with associated levels
$a_i\in[0,1]$, and a one-Lipschitz terminal function $g\in\mathcal G$.  Set
$\alpha^{\mathfrak d}=a_i$ on $[q_i,q_{i+1})$, and define the backward
functions
\[
 u^{\mathfrak d}(q,\cdot)
 =T_{a_i,\Psi'(q_{i+1})-\Psi'(q)}u^{\mathfrak d}(q_{i+1},\cdot),
 \qquad q_i\leq q\leq q_{i+1},
 \qquad u^{\mathfrak d}(q_{\mathrm c},\cdot)=g.
\]
Let $X^{\mathfrak d}$ be the diffusion
$\dd X_q=\Psi''(q)\alpha^{\mathfrak d}(q)u^{\mathfrak d}_x(q,X_q)\dd q
+\sqrt{\Psi''(q)}\dd B_q$, $X_0=0$, and set
\[
 \Gamma^{\mathfrak d}(q)=\E u^{\mathfrak d}_x(q,X^{\mathfrak d}_q)^2,\qquad
 \Gamma_2^{\mathfrak d}(q)=\E u^{\mathfrak d}_{xx}(q,X^{\mathfrak d}_q)^2,\qquad
 f^{\mathfrak d}(q)=\frac12\int_0^q
 \Psi''(s)(\Gamma^{\mathfrak d}(s)-s)\dd s.
\]
When $\mathfrak d$ is fixed, we drop the superscript.  For an atomic
$\mu\in\M$, the \emph{data of $(\Xi,\mu)$} are
$(\Xi,(q_i,\alpha_\mu(q_i))_{i=0}^r,\log\cosh)$ with $q_{\mathrm c}=1$,
where $q_1<\cdots<q_r$ are the atoms of $\mu$ in $(0,1)$; by
\eqref{eq:cole-hopf-step}, these data give the same backward function
$u_\mu$, diffusion $X$, and functions $\Gamma_{\Xi,\mu}$ and $f_{\Xi,\mu}$.  The
level $a_0=\mu(\{0\})$ may vanish, and we also allow additional points
with zero mass, at which the level does not change.

On $[q_i,q_{i+1}]$, we pass to the variance variable
$s=\Psi'(q)-\Psi'(q_i)$, which increases from $0$ to
$V_i=\Psi'(q_{i+1})-\Psi'(q_i)$ because $\Psi''>0$ on $(0,q_{\mathrm c}]$;
let $q(s)$ denote the inverse map.  By L\'evy's characterization,
$\tilde B_s=\int_{q_i}^{q(s)}\sqrt{\Psi''}\dd B$ is a standard Brownian
motion independent of $\mathcal F_{q_i}$, and $\xi_s=X_{q(s)}$ solves
$\dd\xi_s=a_iw_x(s,\xi_s)\dd s+\dd\tilde B_s$ with $\xi_0=X_{q_i}$, where
$w(s,\cdot)=u(q(s),\cdot)=T_{a_i,V_i-s}u(q_{i+1},\cdot)$.  By pathwise
uniqueness, $X_{q_{i+1}}=\xi_{V_i}$ is a measurable function of $X_{q_i}$
and $\tilde B$, and \zcref{lem:one-step}(iii) shows that, conditionally on
$\mathcal F_{q_i}$, its law is $K_{a_i,V_i,u(q_{i+1},\cdot)}(X_{q_i},\cdot)$.
It follows that
\begin{equation}\label{eq:chain-identification}
 \text{$(X_{q_k})_{0\leq k\leq r+1}$ is the forward chain of the levels
 $a_i$, variances $V_i$, and terminal function $g$.}
\end{equation}
Splitting $[q_i,q_{i+1}]$ at a point $q\in(q_i,q_{i+1})$, with level
$a_i$ on both new intervals, changes neither $u$ nor $X$.
Thus $u(q,\cdot)$ and $X_q$ are, respectively, the backward function
and the forward-chain state at $q$.
As in the proof of \zcref{lem:curvature}, $u_x$ and $u_{xx}$ are continuous
on $[0,q_{\mathrm c}]\times\R$.  They are bounded, because $|u_x|\leq1$ by
\zcref{lem:one-step}(ii), and each step increases $\|u_{xx}\|_\infty$ by at
most its level, by \eqref{eq:one-step-spatial-bounds}.  By
It\^o's formula, we have
$\dd u_x(q,X_q)=\sqrt{\Psi''(q)}u_{xx}(q,X_q)\dd B_q$ on each interval
$(q_i,q_{i+1})$.
This gives
\begin{equation}\label{eq:truncated-regularity}
 \Gamma^{\mathfrak d}\in C^1([0,q_{\mathrm c}]),\qquad
 (\Gamma^{\mathfrak d})'=\Psi''\Gamma_2^{\mathfrak d},\qquad
 f^{\mathfrak d}\in C^2([0,q_{\mathrm c}]),
\end{equation}
with $(f^{\mathfrak d})'$ and $(f^{\mathfrak d})''$ given by
\eqref{eq:f-derivatives} with $(\Xi,\Gamma,\Gamma_2)$ replaced by
$(\Psi,\Gamma^{\mathfrak d},\Gamma_2^{\mathfrak d})$.

The next lemma gives continuity of $\Gamma$, $\Gamma_2$, and $f$ when the data vary.
It extends the
continuity of these quantities under weak convergence of the measure
\cite[Proposition~1(ii)]{AuffingerChen2015properties} to varying
covariances and terminal data.  Unlike \zcref{lem:cole-hopf-smooth}, it
requires no lower bound on the levels, so levels may vanish or coincide in
the limit; it allows the points $q_i$ to vary; and it concerns $\Gamma$,
$\Gamma_2$, and $f$ rather than the value $u^{\mathfrak d}(0,0)$.  These
features are needed, for example, in the proofs of
\zcref{prop:first-variation-criterion,lem:deficit}.

\begin{lemma}\label{lem:observable-continuity}
Let $\mathfrak d_n=(\Psi_n,(q_i^n,a_i^n)_{i=0}^r,g_n)$,
$n\in\mathbb N\cup\{\infty\}$, be truncated data with the same $r$ and
$q_{\mathrm c}$; in particular, the limiting points satisfy
$q_0^\infty<\cdots<q_r^\infty<q_{\mathrm c}$.  If $\Psi_n\to\Psi_\infty$ in $C^3([0,q_{\mathrm c}])$,
$q_i^n\to q_i^\infty$ and $a_i^n\to a_i^\infty$ for every $i$, and
$g_n^{(j)}\to g_\infty^{(j)}$ uniformly for $1\leq j\leq3$, then
$\Gamma^{\mathfrak d_n}\to\Gamma^{\mathfrak d_\infty}$ and
$\Gamma_2^{\mathfrak d_n}\to\Gamma_2^{\mathfrak d_\infty}$ uniformly, and
$f^{\mathfrak d_n}\to f^{\mathfrak d_\infty}$ in $C^2([0,q_{\mathrm c}])$.
\end{lemma}

\begin{proof}
We divide the proof into four steps.

\medskip
\noindent\emph{Step 1: Continuity of one step.}
Fix $V_*<\infty$.  For
$a\in[0,1]$, $V\in[0,V_*]$, and one-Lipschitz $g$,
\begin{equation}\label{eq:cole-hopf-increment}
 T_{a,V}g(x)-g(x)
 =\int_0^1
 \frac{\E[\Delta_xg(\sqrt VZ)e^{ta\Delta_xg(\sqrt VZ)}]}
 {\E e^{ta\Delta_xg(\sqrt VZ)}}\dd t,
\end{equation}
which follows from \eqref{eq:centered-cole-hopf} when $a,V>0$ and is
trivial otherwise.  Let $\E_K$ denote the average under the tilt
$e^{a\Delta_xg(\sqrt VZ)}$, that is,
$\E_Kh=\E[h(x+\sqrt VZ)e^{a\Delta_xg(\sqrt VZ)}]/\E e^{a\Delta_xg(\sqrt VZ)}$,
and let $\Var_K$ and $\operatorname{Cov}_K$ be the corresponding variance
and covariance.  By the proof of \zcref{lem:one-step}(ii), whose formulas
also hold when $a=0$ or $V=0$,
\begin{align*}
 (Tg)'&=\E_Kg',\qquad
 (Tg)''=\E_Kg''+a\Var_K(g'),\\
 (Tg)'''&=\E_Kg'''+3a\operatorname{Cov}_K(g',g'')
 +a^2\E_K\bigl[(g'-\E_Kg')^3\bigr],
\end{align*}
where $T=T_{a,V}$.  The tilts lie between $e^{\pm\sqrt{V_*}|Z|}$, and
their averages are at least $e^{-\sqrt{V_*}}$.  Let $a^n\to a$ and
$V^n\to V$ in these ranges, and let $g_n$ be one-Lipschitz with
$g_n^{(j)}\to g^{(j)}$ uniformly for $1\leq j\leq3$.  Then
\[
 \bigl|g_n^{(j)}(x+\sqrt{V^n}Z)-g^{(j)}(x+\sqrt VZ)\bigr|
 \leq\|g_n^{(j)}-g^{(j)}\|_\infty
 +\omega_j\bigl(|\sqrt{V^n}-\sqrt V||Z|\bigr),
\]
with $\omega_j$ a bounded modulus of continuity of $g^{(j)}$, and
$|a^n\Delta_xg_n(\sqrt{V^n}Z)-a\Delta_xg(\sqrt VZ)|
\leq(|a^n-a|+\|g_n'-g'\|_\infty)\sqrt{V_*}|Z|+|\sqrt{V^n}-\sqrt V||Z|$.
These bounds do not depend on $x$, and they tend to zero for each fixed
$Z$.  The differences of the integrands in the tilted averages above and in
the formula for $T_{a,V}g-g$ are bounded by such quantities times
$Ce^{C|Z|}$, uniformly in $x$.  By dominated convergence, applied to these
bounds, $T_{a^n,V^n}g_n-g_n$ and the first three derivatives of
$T_{a^n,V^n}g_n$ converge uniformly in $x$.

\medskip
\noindent\emph{Step 2: Refinement and backward functions.}
Since subtracting $g_n(0)$ from $g_n$ changes $u$ by a constant and
changes neither $u_x$, $u_{xx}$, nor $X$, we may assume that $g_n(0)=0$.  Let
$q_n\to q$ in $[0,q_{\mathrm c}]$.  We show that
$\Gamma^{\mathfrak d_n}(q_n)\to\Gamma^{\mathfrak d_\infty}(q)$ and
$\Gamma_2^{\mathfrak d_n}(q_n)\to\Gamma_2^{\mathfrak d_\infty}(q)$.  It
suffices to show this along a further subsequence of every subsequence.
After passing to a subsequence, we may therefore assume that
$q_n\in[q_i^n,q_{i+1}^n]$ for a fixed $i$, where
$q_{r+1}^n=q_{\mathrm c}$; then $q\in[q_i^\infty,q_{i+1}^\infty]$.  We
work with the backward recursions and forward chains of
\zcref{s:recursions}, which allow variances equal to zero.  We split the
step of $\mathfrak d_n$ on $[q_i^n,q_{i+1}^n]$ at $q_n$, with level $a_i^n$
on both parts, and similarly split the step of $\mathfrak d_\infty$ at $q$;
if $q_n$ or $q$ is an endpoint of the step, one of the two parts has
variance zero and acts as the identity.  The refined recursions have $r+2$
steps each, and their levels and variances converge to those of the
refined limit; some limiting variances may vanish.  We take
$V_*=\sup_n(\Psi_n'(q_{\mathrm c})-\Psi_n'(0))$, which is finite, so that
all variances lie in $[0,V_*]$.  By the continuity
of one step and induction, the first three spatial derivatives of the
backward functions converge
uniformly, and the backward functions converge locally uniformly.

\medskip
\noindent\emph{Step 3: Forward chains.}
Let $Z_0,Z_1,\ldots$ be independent standard Gaussians,
$y_0=0$, and $y_{k+1}=y_k+\sqrt{V_k}Z_k$.  By
\eqref{eq:finite-transition-kernel}, for bounded measurable $\varphi$,
\[
 \E\varphi(x_k)=\E\bigl[\varphi(y_k)\mathcal W_k\bigr],\qquad
 \mathcal W_k=\prod_{i<k}\exp\bigl\{a_i\bigl[u_{i+1}(y_{i+1})-u_i(y_i)\bigr]\bigr\},
\]
which also holds when some $V_i=0$, since then $y_{i+1}=y_i$ and
$u_i=u_{i+1}$, so that the corresponding factor is $1$.  We write the
exponent of the $i$th factor as
$a_i[u_{i+1}(y_{i+1})-u_{i+1}(y_i)]-a_i(u_i-u_{i+1})(y_i)$.  The first
term is at most $\sqrt{V_*}|Z_i|$, because $a_i\leq1$ and $u_{i+1}$ is
one-Lipschitz.  For the second, \eqref{eq:cole-hopf-increment}
bounds $|u_i-u_{i+1}|$ by a constant that
depends only on $V_*$; unlike \eqref{eq:centered-cole-hopf}, this formula
does not degenerate as $a_i\to0$.  Hence
$\mathcal W_k\leq C\exp(C\sum_i|Z_i|)$, uniformly in $n$.  For bounded
continuous $\varphi$, the integrands converge pointwise, because the
backward functions converge locally uniformly.  By dominated
convergence, the forward chains converge in law.  By
\eqref{eq:chain-identification}, applied to the refined data, the law of
$X^n_{q_n}$, where $X^n=X^{\mathfrak d_n}$ and $X^\infty=X^{\mathfrak d_\infty}$, converges weakly to that of $X^\infty_q$.

\medskip
\noindent\emph{Step 4: Convergence of $\Gamma$, $\Gamma_2$, and $f$.}
The function $\varphi_n=u^{\mathfrak d_n}_x(q_n,\cdot)^2$ converges uniformly to
$\varphi=u^{\mathfrak d_\infty}_x(q,\cdot)^2$, and we have
\[
 \bigl|\E\varphi_n(X^n_{q_n})-\E\varphi(X^\infty_q)\bigr|
 \leq\|\varphi_n-\varphi\|_\infty
 +\bigl|\E\varphi(X^n_{q_n})-\E\varphi(X^\infty_q)\bigr|\to0,
\]
that is, $\Gamma^{\mathfrak d_n}(q_n)\to\Gamma^{\mathfrak d_\infty}(q)$ along
the subsequence.  In the same way, with $u_{xx}^2$ in place of $u_x^2$,
$\Gamma_2^{\mathfrak d_n}(q_n)\to\Gamma_2^{\mathfrak d_\infty}(q)$.  Thus
the functions $\Gamma$ and $\Gamma_2$ converge continuously: $\Gamma^{\mathfrak d_n}(q_n)\to
\Gamma^{\mathfrak d_\infty}(q)$ whenever $q_n\to q$, and likewise for
$\Gamma_2$.  Taking $\mathfrak d_n=\mathfrak d_\infty$ shows that the
limiting functions are continuous, and continuous convergence to a
continuous limit on the compact interval $[0,q_{\mathrm c}]$ is uniform
convergence.  Finally, $f^{\mathfrak d_n}(0)=0$, and, by
\eqref{eq:truncated-regularity} and \eqref{eq:f-derivatives},
\[
 (f^{\mathfrak d_n})'=\frac12\Psi_n''\cdot(\Gamma^{\mathfrak d_n}-q),
 \qquad
 (f^{\mathfrak d_n})''=\frac12\Psi_n'''\cdot(\Gamma^{\mathfrak d_n}-q)
 +\frac12\Psi_n''\cdot(\Psi_n''\Gamma_2^{\mathfrak d_n}-1).
\]
Since $\Psi_n\to\Psi_\infty$ in $C^3$ and $\Gamma$ and $\Gamma_2$ converge
uniformly, $f^{\mathfrak d_n}\to f^{\mathfrak d_\infty}$ in
$C^2([0,q_{\mathrm c}])$.
\end{proof}

The next lemma shows that $C^2$ convergence of first-variation functions,
together with stationarity at the prescribed points, preserves their zero
sets, negative second derivatives at positive zeros, and strict negativity
elsewhere.

\begin{lemma}\label{lem:contact-persistence}
Let $\mathfrak d_n$, $n\in\mathbb N\cup\{\infty\}$, be truncated data
satisfying the hypotheses of \zcref{lem:observable-continuity}, and write
$f_n=f^{\mathfrak d_n}$.  Assume that:
\begin{enumerate}[label=\textup{(\roman*)}]
 \item $f_\infty\leq0$ on $[0,q_{\mathrm c}]$,
 $\{f_\infty=0\}=\{q_0^\infty,\ldots,q_r^\infty\}$, and $f_\infty''(q_i^\infty)<0$ for
 $1\leq i\leq r$;
 \item there exists $\varepsilon>0$ such that
 $\Gamma^{\mathfrak d_n}(q)\leq(1-\varepsilon)q$ for $0\leq q\leq\varepsilon$ and
 all large $n$;
 \item $f_n(q_i^n)=f_n'(q_i^n)=0$ for $1\leq i\leq r$ and all large $n$.
\end{enumerate}
Then, for all large $n$, $f_n\leq0$ on $[0,q_{\mathrm c}]$,
$\{f_n=0\}=\{q_0^n,\ldots,q_r^n\}$, and $f_n''(q_i^n)<0$ for
$1\leq i\leq r$.
\end{lemma}

\begin{proof}
Since $\Psi_n''>0$ on $(0,q_{\mathrm c}]$ and
$\Gamma^{\mathfrak d_n}(q)-q\leq-\varepsilon q$ on $[0,\varepsilon]$, we have
$f_n'<0$ on $(0,\varepsilon]$, and therefore $f_n<0$ there.  Fix
$\delta,c>0$ such
that the intervals $I_i=[q_i^\infty-2\delta,q_i^\infty+2\delta]$,
$1\leq i\leq r$, are pairwise disjoint subsets of $(0,q_{\mathrm c})$ on
which $f_\infty''\leq-2c$, and decrease $\varepsilon$ so that
$\varepsilon<q_1^\infty-2\delta$ when $r\geq1$.  The compact set
$S=[\varepsilon,q_{\mathrm c}]\setminus
\bigcup_i(q_i^\infty-\delta,q_i^\infty+\delta)$ contains no zero of
$f_\infty$.  After decreasing $c$, we then have $f_\infty\leq-2c$ on $S$.  For large $n$,
we have $|q_i^n-q_i^\infty|<\delta$ and, by
\zcref{lem:observable-continuity}, $\|f_n-f_\infty\|_{C^2}<c$.  Then
$f_n\leq-c$ on $S$, and $f_n''\leq-c$ on each $I_i$.  By condition~(iii)
and Taylor's theorem at $q_i^n$, we have
$f_n(q)\leq-\frac c2(q-q_i^n)^2$ for $q\in I_i$.  Since $[0,\varepsilon]$,
$S$, and the $I_i$ cover $[0,q_{\mathrm c}]$, the assertions follow.
\end{proof}

\begin{remark}\label{rem:origin-condition}
Condition~(ii) of \zcref{lem:contact-persistence} holds, in particular, if
there exist an admissible covariance $\Xi_n$ with
$\Xi_n''\leq1-\varepsilon$ on $[0,\varepsilon]$ and an atomic measure $\mu_n$
such that $\Gamma^{\mathfrak d_n}=\Gamma_{\Xi_n,\mu_n}$ on $[0,\varepsilon]$.
Indeed, by \eqref{eq:Gamma-curvature}, we then have
$\Gamma^{\mathfrak d_n}(q)\leq\Xi_n'(q)=\int_0^q\Xi_n''\leq(1-\varepsilon)q$
for $0\leq q\leq\varepsilon$.
\end{remark}

\begin{proof}[Proof of \zcref{lem:persistence}]
Let $\mu_n=\nu_{z_n}$, and let $q_0^n<\cdots<q_r^n$ and
$w_0^n,\ldots,w_r^n$ be its atoms and masses.  The data of
$(\Xi_n,\mu_n)$ have covariance $\Xi_n$, points $q_i^n$, levels
$w_0^n+\cdots+w_i^n$, and terminal function $\log\cosh$, with
$q_{\mathrm c}=1$.  Since $z_n\to z_\infty\in\mathcal Z_r$, the atom
locations and levels converge, and the limiting locations satisfy
$q_0^\infty<\cdots<q_r^\infty<1$.  By \zcref{lem:observable-continuity},
$f_n\to f_\infty$ in $C^2([0,1])$.

For the second assertion, we check the conditions of
\zcref{lem:contact-persistence}.  Since $\mu_\infty$ is nondegenerate,
condition~(i) holds.  Since $\Xi_\infty''(0)<1$ and $\Xi_n''\to\Xi_\infty''$
uniformly, there exists $\varepsilon>0$ such that $\Xi_n''\leq1-\varepsilon$ on
$[0,\varepsilon]$ for all large $n$, and condition~(ii) holds by
\zcref{rem:origin-condition}.  Since $z_n$ is a critical point,
condition~(iii) follows from \zcref{cor:atomic-gradient} with $m=1$ (and
any $Q\in(q_r^n,1)$).
By \zcref{lem:contact-persistence}, for all large $n$, we have
$f_n\leq0$ on $[0,1]$, $\{f_n=0\}=\supp\mu_n$, and $f_n''(q_i^n)<0$ for
$1\leq i\leq r$.  By \zcref{prop:first-variation-criterion}, $\mu_n$ is the
Parisi measure of $\Xi_n$, and it is nondegenerate.
\end{proof}

\subsection{First-variation formulas and the Hessian in atom coordinates}\label{s:variation-proofs}

Let $\mu,\nu\in\M$ be atomic; atoms at $0$ or $1$ are allowed.  Since
$f_{\Xi,\mu}(0)=0$ and $\alpha_\nu(1)=\alpha_\mu(1)=1$, integration by parts
in the form
$\int_{[0,1]}f_{\Xi,\mu}\dd\nu=f_{\Xi,\mu}(1)-\int_0^1f_{\Xi,\mu}'\alpha_\nu\dd q$
and \eqref{eq:f-derivatives} give
\begin{equation}\label{eq:f-integral}
 \int_0^1f_{\Xi,\mu}\dd(\nu-\mu)
 =-\frac12\int_0^1\Xi''(q)\bigl(\Gamma_{\Xi,\mu}(q)-q\bigr)
 (\alpha_\nu-\alpha_\mu)(q)\dd q.
\end{equation}

The proofs of \zcref{prop:first-variation-criterion,lem:atomic-hessian}
use the following strengthened first-variation inequality.  Its
remainder term is the error in the first-order Taylor expansion of
$\log\cosh$ at $X_1$, where $X$ is the optimal diffusion of $\mu$; it is
nonnegative, and it gives a strictly positive second-order
term when only the masses change.  The proof evaluates the variational formula for $\nu$ using the
magnetization process associated with $\mu$
\cite[Theorem~3]{AuffingerChen2015unique},
\cite[Lemma~18]{JagannathTobasco2016}.

\begin{lemma}\label{lem:sharp-first-variation}
Let $\Xi$ be admissible, let $\mu,\nu\in\M$ be atomic, let $X$ be the
optimal diffusion of $\mu$, and set
$I=\int_0^1\Xi''(q)(\alpha_\nu-\alpha_\mu)(q)u_{\mu,x}(q,X_q)\dd q$.  Then
\begin{equation}\label{eq:sharp-first-variation}
 \Pp_\Xi(\nu)-\Pp_\Xi(\mu)
 \geq-\int_0^1f_{\Xi,\mu}\dd(\nu-\mu)
 +\E\bigl[\log\cosh(X_1+I)-\log\cosh X_1-\tanh(X_1)I\bigr],
\end{equation}
and the expectation on the right is nonnegative.
\end{lemma}

\begin{proof}
Write $U_q=u_{\mu,x}(q,X_q)$ and $\Gamma=\Gamma_{\Xi,\mu}$.  Driven by
the same Brownian motion as $X$, let $Y$ solve
\[
 \dd Y_q=\Xi''(q)\alpha_\nu(q)U_q\dd q+\sqrt{\Xi''(q)}\dd B_q,
 \qquad Y_0=0.
\]
Then $Y_1=X_1+I$.
By It\^o's formula and the Parisi PDE for $u_\nu$, we have, between the
atoms of $\nu$,
\[
 \dd u_\nu(q,Y_q)
 =\Xi''\alpha_\nu\Bigl(U_qu_{\nu,x}-\frac12u_{\nu,x}^2\Bigr)\dd q+\dd M_q
 \leq\frac12\Xi''\alpha_\nu U_q^2\dd q+\dd M_q,
\]
where $u_{\nu,x}$ is evaluated at $(q,Y_q)$ and $M$ is a martingale.  Here
$u_\nu$ is $C^{1,2}$ between consecutive atoms of $\nu$ and continuous on
$[0,1]\times\R$, as in the proof of \zcref{lem:curvature}, and the
integrand $\sqrt{\Xi''}u_{\nu,x}(q,Y_q)$ of $M$ is bounded by
$\sqrt{\Xi''(1)}$.  We apply It\^o's formula on closed intervals strictly
between consecutive atoms and pass to their endpoints by continuity and
dominated convergence, using $|u_\nu(q,x)|\leq C+|x|$.
Taking expectations, we have
$u_\nu(0,0)\geq\E\log\cosh(X_1+I)-\frac12\int_0^1\Xi''\alpha_\nu\Gamma\dd q$,
with equality when $\nu=\mu$, since then $Y=X$.  Subtracting the identity
for $\mu$ from the inequality for $\nu$ and using the definition
\eqref{eq:parisi-functional} of $\Pp_\Xi$, we obtain
\begin{align*}
 \Pp_\Xi(\nu)-\Pp_\Xi(\mu)
 &\geq\E\bigl[\log\cosh(X_1+I)-\log\cosh X_1\bigr]
 -\frac12\int_0^1\Xi''(\alpha_\nu-\alpha_\mu)(\Gamma+q)\dd q\\
 &=\E\bigl[\log\cosh(X_1+I)-\log\cosh X_1-\tanh(X_1)I\bigr]
 +\frac12\int_0^1\Xi''(\alpha_\nu-\alpha_\mu)(\Gamma-q)\dd q.
\end{align*}
In the second line, we used that
$\E[\tanh(X_1)I]=\int_0^1\Xi''(\alpha_\nu-\alpha_\mu)\Gamma\dd q$.  This
holds because $U$ is a martingale with $U_1=\tanh X_1$ by
\zcref{lem:curvature}, and therefore
$\E[\tanh(X_1)U_q]=\E U_q^2=\Gamma(q)$.  By
\eqref{eq:f-integral}, the last term equals
$-\int f_{\Xi,\mu}\dd(\nu-\mu)$, which proves
\eqref{eq:sharp-first-variation}.  The expectation is nonnegative because
$\log\cosh$ is convex with derivative $\tanh$.
\end{proof}

\begin{proof}[Proof of \zcref{prop:first-variation-criterion}]
Let $\mu_t=(1-t)\mu+t\nu$.  By \zcref{lem:sharp-first-variation}, applied
to $(\mu,\mu_t)$ and to $(\mu_t,\mu)$,
\[
 -t\int_0^1f_{\Xi,\mu}\dd(\nu-\mu)
 \leq\Pp_\Xi(\mu_t)-\Pp_\Xi(\mu)
 \leq-t\int_0^1f_{\Xi,\mu_t}\dd(\nu-\mu).
\]
Let $0=\tau_0<\cdots<\tau_{r'}$ be $0$ together with the points of
$\supp\mu\cup\supp\nu$ in $(0,1)$.  For every $t\in[0,1]$, the data of
$(\Xi,\mu_t)$ using these same points have levels $\alpha_{\mu_t}(\tau_i)$,
which converge to $\alpha_\mu(\tau_i)$ as $t\downarrow0$.  By
\zcref{lem:observable-continuity}, $f_{\Xi,\mu_t}\to f_{\Xi,\mu}$
uniformly, and \eqref{eq:directional-derivative} follows.  For
\eqref{eq:new-point-derivative}, we take $\nu=\delta_Q$ and $t=1-m$.  Next,
for $|\varepsilon|$ small, let
$\mu_\varepsilon=\mu+w_i(\delta_{q_i+\varepsilon}-\delta_{q_i})$.  Then
\[
 -w_i\bigl(f_{\Xi,\mu}(q_i+\varepsilon)-f_{\Xi,\mu}(q_i)\bigr)
 \leq\Pp_\Xi(\mu_\varepsilon)-\Pp_\Xi(\mu)
 \leq-w_i\bigl(f_{\Xi,\mu_\varepsilon}(q_i+\varepsilon)
 -f_{\Xi,\mu_\varepsilon}(q_i)\bigr).
\]
The data of $(\Xi,\mu_\varepsilon)$, with the atom at $q_i+\varepsilon$
in place of $q_i$, converge to those of $(\Xi,\mu)$ as $\varepsilon\to0$.
By \zcref{lem:observable-continuity}, we have
$f_{\Xi,\mu_\varepsilon}\to f_{\Xi,\mu}$ in $C^1$.  Since $f_{\Xi,\mu}\in C^1$, dividing by
$\varepsilon$ and letting $\varepsilon\to0$ proves the first formula in
\eqref{eq:location-derivative}.
For the second, let $\nu=\mu+w_j(\delta_{q_i}-\delta_{q_j})\in\M$.  Then
$(1-t)\mu+t\nu=\mu+tw_j(\delta_{q_i}-\delta_{q_j})$, and
\eqref{eq:directional-derivative} gives the right derivative
$f(q_j)-f(q_i)$ in the variable $s=tw_j$.  Exchanging the roles of $i$ and
$j$, which is possible because $w_i>0$, gives the same formula for the left
derivative.  Now assume that
\eqref{eq:first-variation-criterion} holds.  Then $\int f_{\Xi,\mu}\dd\mu=0$
and $\int f_{\Xi,\mu}\dd\nu\leq0$, and \eqref{eq:sharp-first-variation}
implies that $\Pp_\Xi(\nu)\geq\Pp_\Xi(\mu)$ for every atomic $\nu$.  For
general $\nu\in\M$, $u_\nu$ is the weak solution of \eqref{eq:parisi-pde}
\cite{JagannathTobasco2016}.  Atomic measures are dense in $\M$ for
$d_1(\nu,\nu')=\int_0^1|\alpha_\nu-\alpha_{\nu'}|\dd q$, and $\Pp_\Xi$ is
Lipschitz for $d_1$.  Indeed,
$\|u_\nu-u_{\nu'}\|_\infty\leq\Xi''(1)d_1(\nu,\nu')$ by
\cite[Lemma~14]{JagannathTobasco2016}, and, since
$0\leq q\Xi''(q)\leq\Xi''(1)$ on $[0,1]$, the last term of $\Pp_\Xi$ changes
by at most $\frac12\Xi''(1)d_1(\nu,\nu')$.  We conclude that $\mu$ minimizes
$\Pp_\Xi$ over
$\M$, and it is the unique minimizer \cite{AuffingerChen2015unique}.
\end{proof}

\begin{proof}[Proof of \zcref{lem:atomic-hessian}]
The assertion is vacuous for $r=0$, so let $r\geq1$, and write $\mu=\nu_z$,
$f=f_{\Xi,\mu}$, and $\mathcal H=D_z^2\Pp_\Xi(\nu_z)$.  By
\zcref{cor:atomic-gradient}, $f(q_i)=0$ for all $i$ and $f'(q_i)=0$ for
$i\geq1$.  Fix a nonzero direction $v$ in the atomic parameter space and set
$\mu_t=\nu_{z+tv}$.  Write $\dot q_i$ and $\dot w_i$ for the
derivatives at $t=0$ of the atom locations and masses, respectively.
Since $q_0=0$ is fixed, $\dot q_0=0$.   Since $z$
is a critical point,
$\Pp_\Xi(\mu_t)-\Pp_\Xi(\mu)=\frac12t^2v^\top\mathcal Hv+o(t^2)$.

\medskip
\noindent\emph{A support point moves ($\dot q\neq0$).}
We apply \zcref{lem:sharp-first-variation} with
$\nu=\mu_t$ and drop the nonnegative expectation in
\eqref{eq:sharp-first-variation}.  Since $f$ is $C^2$ by
\eqref{eq:f-derivatives}, and $f$ and $f'$ vanish at the atoms as recorded
above, we have
\[
 -\int_0^1f\dd(\mu_t-\mu)=-\sum_{i=1}^rw_i(t)f(q_i+t\dot q_i)
 =-\frac{t^2}2\sum_{i=1}^rw_if''(q_i)\dot q_i^2+o(t^2),
\]
which implies that $v^\top\mathcal Hv\geq-\sum_{i}w_if''(q_i)\dot q_i^2>0$.

\medskip
\noindent\emph{Only masses move ($\dot q=0$).}
Then $\int f\dd(\mu_t-\mu)=0$ and
$\alpha_{\mu_t}-\alpha_\mu=t\dot\alpha$ with
$\dot\alpha(q)=\sum_{i\colon q_i\leq q}\dot w_i$.  Then $I=t\mathcal L$ in
\eqref{eq:sharp-first-variation}, where
$\mathcal L=\int_0^1\Xi''(q)\dot\alpha(q)U_q\dd q$ and $U_q=u_{\mu,x}(q,X_q)$.
We will show that $\E\mathcal L^2>0$, so that the convexity remainder
contributes a strictly positive second-order term.  By Taylor's theorem
with integral remainder, the expectation in
\eqref{eq:sharp-first-variation} equals
$t^2\E[\mathcal L^2\int_0^1(1-\tau)\sech^2(X_1+\tau t\mathcal L)\dd\tau]$.  Since $|U|\leq1$, we have
$|\mathcal L|\leq\Xi'(1)\sum_i|\dot w_i|$, and dominated convergence gives
$v^\top\mathcal Hv\geq\E[\sech^2(X_1)\mathcal L^2]$.  By
\zcref{lem:curvature}, we have
$U_q=\int_0^q\sqrt{\Xi''(s)}u_{\mu,xx}(s,X_s)\dd B_s$.  Exchanging the order
of integration by the stochastic Fubini theorem, we find that
\[
 \mathcal L=\int_0^1\Bigl(\int_s^1\Xi''(q)\dot\alpha(q)\dd q\Bigr)
 \sqrt{\Xi''(s)}u_{\mu,xx}(s,X_s)\dd B_s,
\]
and the It\^o isometry gives
\[
 \E\mathcal L^2=\int_0^1\Bigl(\int_s^1\Xi''(q)\dot\alpha(q)\dd q\Bigr)^2
 \Xi''(s)\E u_{\mu,xx}(s,X_s)^2\dd s.
\]
Since $\dot w\neq0$ and $\sum_i\dot w_i=0$, the step function $\dot\alpha$
vanishes on $[q_r,1]$ but not identically on $[0,q_r)$.  It follows that
the inner integral, whose derivative in $s$ is $-\Xi''\dot\alpha$, is not
identically zero.  As $\Xi''>0$ on $(0,1]$ and $u_{\mu,xx}>0$ by
\eqref{eq:curvature}, we have $\E\mathcal L^2>0$.  Since
$\sech^2(X_1)>0$, we conclude that $v^\top\mathcal Hv>0$.
\end{proof}

\subsection{Scaling and entropy}\label{s:scaling-proofs}

In the notation of \zcref{s:truncated}, the lower problem of
\zcref{lem:temperature-mass-scaling} has the truncated data
\[
 \mathfrak d^\circ_{\nu,m}
 =\bigl(\Xi,(q_i,m\alpha_\nu(q_i))_{i=0}^r,\psi_m\bigr),
 \qquad \psi_m(x)=m^{-1}\log\cosh(mx),
\]
with cutoff $q_{\mathrm c}$, where $0=q_0<\cdots<q_r$ are $0$ together
with the atoms of $\nu$ in $(0,1)$.  Comparing \eqref{eq:lower-problem}
with the definition of $u^{\mathfrak d}$, we see that
$\widetilde u=u^{\mathfrak d^\circ_{\nu,m}}$.  The next lemma shows that
the diffusion and the functions $\Gamma$ and $f$ associated with these data rescale in the same way.

\begin{lemma}\label{lem:temperature-mass-diffusion}
In the setting of \zcref{lem:temperature-mass-scaling}, let $Y$ be the
optimal diffusion of $(m^2\Xi,\nu)$, and let
$X^\circ=X^{\mathfrak d^\circ_{\nu,m}}$ be driven by the same Brownian
motion.  Then $X^\circ=m^{-1}Y$ on $[0,q_{\mathrm c}]$, and
\begin{equation}\label{eq:temperature-mass-Gamma-f}
 \Gamma^{\mathfrak d^\circ_{\nu,m}}=\Gamma_{m^2\Xi,\nu},
 \qquad
 f^{\mathfrak d^\circ_{\nu,m}}=\frac1{m^2}f_{m^2\Xi,\nu}
 \qquad\text{on }[0,q_{\mathrm c}].
\end{equation}
\end{lemma}

\begin{proof}
Let $u=u_{m^2\Xi,\nu}$.  By \zcref{lem:temperature-mass-scaling}, we have
$u^{\mathfrak d^\circ_{\nu,m}}_x(q,x)=u_x(q,mx)$.  Since the levels of
$\mathfrak d^\circ_{\nu,m}$ are $m\alpha_\nu$, the process $mX^\circ$
satisfies
\[
 \dd(mX^\circ_q)
 =m^2\Xi''(q)\alpha_\nu(q)u_x(q,mX^\circ_q)\dd q
 +m\sqrt{\Xi''(q)}\dd B_q,
 \qquad mX^\circ_0=0.
\]
This is the equation of the optimal diffusion of $(m^2\Xi,\nu)$.  Since
$m^2\Xi$ is admissible, \zcref{lem:curvature} implies that
$0<u_{xx}\leq1$.  Then the drift of this equation is globally Lipschitz in space, and
$mX^\circ=Y$ by pathwise uniqueness.  Then
$\Gamma^{\mathfrak d^\circ_{\nu,m}}(q)=\E u_x(q,Y_q)^2=\Gamma_{m^2\Xi,\nu}(q)$.
Integrating $\Gamma(q)-q$ against $\frac12\Xi''=\frac1{2m^2}(m^2\Xi)''$, we
obtain the second identity in \eqref{eq:temperature-mass-Gamma-f}.
\end{proof}

The terminal functions $\psi_m$ enter \zcref{lem:cole-hopf-smooth} as
parameter-dependent data, so we record their regularity.

\begin{remark}\label{rem:psi-regularity}
Each $\psi_m$ is one-Lipschitz with $\psi_m(0)=0$, and for every compact
interval $I\subset(0,\infty)$ and all integers $k\geq0$ and $j\geq1$,
$\sup_{m\in I,\,x\in\R}|\partial_m^k\partial_x^j\psi_m(x)|<\infty$.  Indeed,
$\partial_x\psi_m(x)=\tanh(mx)$, and for $j\geq1$ the derivative
$\partial_m^k\partial_x^j\psi_m$ is a finite sum of terms
$x^im^{i'}\tanh^{(n)}(mx)$ with $n\geq1$ whenever $i\geq1$; since
$\tanh^{(n)}$ decays exponentially for $n\geq1$, these terms are bounded
uniformly for $m\in I$ and $x\in\R$.  Since also $\partial_m^k\psi_m(0)=0$,
the functions $\partial_m^k\psi_m$ are bounded in $\|\cdot\|_3$ uniformly for
$m\in I$, and they belong to $\mathcal G$ because their fourth
$x$-derivatives are bounded.  By Taylor's theorem in $m$, applied to each
$\partial_m^k\partial_x^j\psi_m(x)$, we have
$\|\partial_m^k\psi_{m+\delta}-\partial_m^k\psi_m-\delta\partial_m^{k+1}\psi_m\|_3
\leq\delta^2\sup_{m'\in I}\|\partial_m^{k+2}\psi_{m'}\|_3$ whenever
$m,m+\delta\in I$.  Hence $m\mapsto\psi_m$ is $C^\infty$ from $I$ into
$(\mathcal G,\|\cdot\|_3)$, with derivatives $\partial_m^k\psi_m$.
\end{remark}

\begin{proof}[Proof of \zcref{lem:entropy-identity}]
The coefficient vector of $m^2\Xi$ depends smoothly on $m$ and is
admissible for $m$ near $1$.  By \zcref{prop:finite-step-regularity}, the
map $m\mapsto\frac1m\Pp_{m^2\Xi}(\nu)$ is smooth near $m=1$, and it suffices to compute its derivative from the
left, which we do because \zcref{lem:temperature-mass-scaling} and
\eqref{eq:new-point-derivative} are available only for $m\leq1$.  Let $f=f_{\Xi,\nu}$, let $0=q_0<\cdots<q_r$ be
the atoms of $\nu$, and set $q_{r+1}=1$.  For $0<m\leq1$ and a one-Lipschitz
$g\in\mathcal G$, let
\begin{equation}\label{eq:Rm-def}
 R_m(g)=u_0(0),
 \qquad
 u_{r+1}=g,
 \qquad
 u_i=T_{m\alpha_\nu(q_i),\,\Xi'(q_{i+1})-\Xi'(q_i)}u_{i+1},
\end{equation}
be the output of the backward recursion with levels $m\alpha_\nu(q_i)$ and
variances $\Xi'(q_{i+1})-\Xi'(q_i)$.  By \zcref{lem:temperature-mass-scaling}
with $q_{\mathrm c}=1$, we have $\frac1mu_{m^2\Xi,\nu}(0,0)=R_m(\psi_m)$,
and the definition \eqref{eq:parisi-functional} of the Parisi functional gives
\[
 \frac1m\Pp_{m^2\Xi}(\nu)=\frac{\log2}m+R_m(\psi_m)
 -\frac m2\int_0^1q\Xi''\alpha_\nu\dd q.
\]
The measure $\tilde\nu_m=m\nu+(1-m)\delta_1$ has cumulative mass $m\alpha_\nu$ on
$[0,1)$.  By \eqref{eq:cole-hopf-step}, we have
$\Pp_\Xi(\tilde\nu_m)=\log2+R_m(\log\cosh)-\frac m2\int_0^1q\Xi''\alpha_\nu\dd q$.
Since $\psi_1=\log\cosh$, it follows that
\[
 \frac1m\Pp_{m^2\Xi}(\nu)
 =\Pp_\Xi(\tilde\nu_m)+\Bigl(\frac1m-1\Bigr)\log2+R_m(\psi_m)-R_m(\psi_1).
\]
We differentiate the three terms from the left at $m=1$.  By
\eqref{eq:new-point-derivative} with $\mu=\nu$ and $Q=1$, the derivative
of the first term is $f(1)-\int f\dd\nu$, and that of the second is
$-\log2$.  For the third, we consider $(m,b)\mapsto R_m(\psi_b)$ on
$D=[1-\varepsilon,1]\times[1-\varepsilon,1+\varepsilon]$ for a small $\varepsilon>0$.
This is a backward recursion with $r+1$ steps, levels
$m\alpha_\nu(q_i)\in[(1-\varepsilon)\nu(\{0\}),1]$, positive variances
$\Xi'(q_{i+1})-\Xi'(q_i)$ that do not depend on $(m,b)$, and terminal
functions $\psi_b$.  By \zcref{rem:psi-regularity}, the $\psi_b$ are
one-Lipschitz, and their $b$-derivatives of every order are bounded and
continuous in $\|\cdot\|_3$.  Since $\nu(\{0\})>0$,
\zcref{lem:cole-hopf-smooth} with $\ell=1$ applies and shows that
$(m,b)\mapsto R_m(\psi_b)$ is $C^1$ on $D$.  Since $D$ is convex and the
first derivatives of $R(m,b)=R_m(\psi_b)$ extend continuously to $D$, the
mean value theorem along segments shows that $R$ is differentiable
relative to $D$ at every point of $D$.  The chain rule therefore applies
to $m\mapsto(m,m)$ and $m\mapsto(m,1)$ on $[1-\varepsilon,1]$, and gives
$\frac{\dd}{\dd m}[R(m,m)-R(m,1)]=\partial_mR(m,m)+\partial_bR(m,m)
-\partial_mR(m,1)$, which equals $\partial_bR(1,1)$ at $m=1$.  This is the
derivative of $R_1$ in the terminal direction
$\partial_b\psi_b|_{b=1}=x\tanh x-\log\cosh x$, which belongs to
$\mathcal G$ because its derivatives of orders one to four are bounded.  At $m=1$, the
recursion \eqref{eq:Rm-def} is that of the data of $(\Xi,\nu)$, whose
forward chain ends at $X_1$ by \eqref{eq:chain-identification}.  By
\eqref{eq:terminal-direction}, this derivative equals
$\E[X_1\tanh X_1-\log\cosh X_1]$.  Adding the three contributions and using
the definition of $\mathfrak h$ proves the lemma.
\end{proof}

\section{Uniform perturbation estimates}\label{app:uniform}

We prove the results of \zcref{sec:high-degree-perturbation} whose
proofs are technical, namely
\zcref{lem:lower-recursion,lem:deficit,prop:reduced-functional,prop:stationary-branches,lem:lower-geometry}.
We first split the recursion at the fixed cutoff $\bar q$ and compare the
part below it with the lower problem of
\zcref{lem:temperature-mass-scaling}.  We then prove the decomposition of
the Parisi functional of the candidate minimizers and the estimates for
the magnetization deficit.  Finally, we establish the regularity of the
successive reductions and compare the first-variation functions of the
candidate minimizers on $[0,\bar q]$ with $f^\circ_m$.  All notation is that of
\zcref{sec:high-degree-perturbation}, including the choices (S1)--(S5)
and the conventions on constants.

On $[a(z),1]$, the cumulative mass is $m$ below $Q$ and $1$ above it.
By \eqref{eq:cole-hopf-step} and \eqref{eq:T-logcosh}, the Parisi solution
$u_p$ of $\mu_{z,m,Q}$ for $\Psi$ satisfies
\begin{equation}\label{eq:top-solution}
 u_p(q,x)
 =\frac12\bigl(\Psi'(1)-\Psi'(Q)\bigr)
 +T_{m,\Psi'(Q)-\Psi'(q)}\log\cosh x,
 \qquad a(z)\leq q\leq Q.
\end{equation}

\subsection{Splitting at the fixed cutoff}

For $\theta\in\Theta$ and $Q\in S_p$, let
$W_p=\Psi'(Q)-\Psi'(\bar q)$ be the variance increment from the cutoff to
$Q$; it depends on $(\lambda,Q)$.  Since
\[
 W_p-\lambda p
 =\bigl[\Xi'(Q)-\Xi'(\bar q)\bigr]-\lambda p\bigl(1-Q^{p-1}\bigr)
 -\lambda p\bar q^{\,p-1},
\]
with $0\leq\lambda p(1-Q^{p-1})\leq\lambda p^2(1-Q)\to0$ and
$\lambda p\bar q^{\,p-1}\to0$, we have
\begin{equation}\label{eq:endpoint-input-W}
 |W_p-\lambda p|\leq C,
\end{equation}
while every derivative of $W_p$ in $(\lambda,Q)$ of total order $k$ is
bounded by $Cp^{k+1}$.  Indeed,
$W_p=\Xi'(Q)-\Xi'(\bar q)+\lambda p(Q^{p-1}-\bar q^{\,p-1})$ is affine in
$\lambda$, $\partial_Q^kW_p=\Psi^{(k+1)}(Q)=O(p^{k+1})$, and
$\partial_\lambda\partial_Q^{k-1}W_p=O(p^k)$.

We use the estimates of \zcref{app:large-variance} for the Parisi solution
and the magnetization on $[\bar q,Q]$.  We first identify
the terminal function at the cutoff.

\begin{lemma}\label{lem:cutoff-splitting}
For every sufficiently large even $p$ and every
$(\theta,Q)\in\Theta\times S_p$, the Parisi solution $u_p$ of
$\mu_{z,m,Q}$ for $\Psi$ satisfies $u_p(\bar q,\cdot)=k_p+\psi^{(p)}_m$,
where $k_p=k_p(\lambda,m,Q)$ does not depend on $x$ (see the proof) and
\[
 \psi^{(p)}_m=T_{m,W_p}\log\cosh-\frac{mW_p}{2}-\Bigl(\frac1m-1\Bigr)\log2.
\]
The dependence of $\psi^{(p)}_m$ on $(\lambda,Q)$ is through $W_p$.
It is one-Lipschitz and belongs to $\mathcal G$ by
\zcref{lem:one-step}(ii).
The map $(\lambda,m,Q)\mapsto\psi^{(p)}_m$ has parameter derivatives through
order four that are continuous in $\|\cdot\|_3$, and
\begin{equation}\label{eq:lower-terminal-approximation}
 \max_{|\mathbf j|\leq4}
 \bigl\|
 \partial_{\lambda,m,Q}^{\mathbf j}(\psi^{(p)}_m-\psi_m)
 \bigr\|_3
 \leq Cp^Ae^{-cp}.
\end{equation}
\end{lemma}

\begin{proof}
Since $a(z)\leq a_*<\bar q<Q$, \eqref{eq:top-solution} at $q=\bar q$
gives
$u_p(\bar q,\cdot)=\frac12(\Psi'(1)-\Psi'(Q))+T_{m,W_p}\log\cosh$, which is
the claim with
$k_p=\frac12(\Psi'(1)-\Psi'(Q))+\frac{mW_p}2+(\frac1m-1)\log2$.  By
\eqref{eq:endpoint-input-W}, we have $W_p\geq\overline V$ for large $p$,
with $\overline V$ from \zcref{lem:large-V} for $I=[m_0,1]$, and that lemma
implies that
$\psi^{(p)}_m-\psi_m=\mathcal E_{m,W_p}$.  We use \eqref{eq:large-V-error}
in two ways.  With at most three $x$-derivatives and at most four
derivatives in $(m,V)$, it bounds the quantities that enter the parameter
derivatives of order at most four in $\|\cdot\|_3$.  With one further
derivative, so of total order at most eight, it shows by the mean value
theorem that these parameter derivatives are continuous in $\|\cdot\|_3$.
It also shows that their third $x$-derivatives are uniformly continuous, as
membership in $\mathcal G$ requires.  By the chain rule, each
derivative $\partial_x^i\partial^{\mathbf j}_{\lambda,m,Q}\mathcal E_{m,W_p}$
with $i\leq3$ and $|\mathbf j|\leq4$ is a sum of products of two kinds of
factors.  The first is a derivative
$\partial_x^i\partial_m^{j'}\partial_V^k\mathcal E_{m,V}$ at $V=W_p$, with
$i+j'+k\leq7$, and the others are exactly $k$ derivatives of $W_p$ in
$(\lambda,Q)$, whose orders add up to at most $4$.  Since $\lambda_-p-C\leq W_p\leq Cp$ by
\eqref{eq:endpoint-input-W}, the first factor is at most
$Cp^7e^{-\frac12m_0^2\lambda_-p}$ by \eqref{eq:large-V-error} with
$I=[m_0,1]$.  Since the derivatives of $W_p$ of order $j$ are
$O(p^{j+1})$, the product of the derivatives of $W_p$ is at most $Cp^8$.
This gives \eqref{eq:lower-terminal-approximation}, and the same argument
with one more derivative gives the regularity of $\psi^{(p)}_m-\psi_m$.  The regularity of $\psi_m$ follows
from \zcref{rem:psi-regularity}.
\end{proof}

Below the cutoff, three truncated problems appear: the candidate minimizer
itself, the same problem with terminal function $\psi_m$, and the lower
problem with covariance $\Xi$.

\begin{lemma}\label{lem:below-cutoff}
For every sufficiently large even $p$ and every
$(\theta,Q)\in\Theta\times S_p$, let $q_0<\cdots<q_r$ be the atoms of
$\nu_z$, let $G_\theta=(q_i,m\alpha_{\nu_z}(q_i))_{i=0}^r$ record
these atoms and the corresponding cumulative masses of $m\nu_z$,
and consider the truncated data with cutoff $\bar q$
\[
 \mathfrak d_p(\theta,Q)=\bigl(\Psi,G_\theta,\psi^{(p)}_m\bigr),
 \qquad
 \mathfrak d^\Psi_\theta=\bigl(\Psi,G_\theta,\psi_m\bigr),
 \qquad
 \mathfrak d^\circ_\theta=\bigl(\Xi,G_\theta,\psi_m\bigr).
\]
The first two differ only in their terminal functions, and the last two
only in their covariances.
\begin{enumerate}[label=\textup{(\roman*)}]
 \item On $[0,\bar q]$, the Parisi solution of $\mu_{z,m,Q}$ for $\Psi$ is
 $k_p+u^{\mathfrak d_p(\theta,Q)}$, its optimal diffusion is
 $X^{\mathfrak d_p(\theta,Q)}$, and
 $f_{\Psi,\mu_{z,m,Q}}=f^{\mathfrak d_p(\theta,Q)}$.
 \item $u^{\mathfrak d^\Psi_\theta}(0,0)
 =\frac1mu_{m^2\Psi,\nu_z}(0,0)-\frac m2\bigl(\Psi'(1)-\Psi'(\bar q)\bigr)$.
 \item $f^{\mathfrak d^\circ_\theta}=m^{-2}f_{m^2\Xi,\nu_z}$ on
 $[0,\bar q]$; in particular, $f^\circ_m=f^{\mathfrak d^\circ_\theta}$ on
 $[0,\bar q]$ when $z=z(m)$.
\end{enumerate}
\end{lemma}

\begin{proof}
(i) On $[0,\bar q]$, the cumulative mass of $\mu_{z,m,Q}$ is
$m\alpha_{\nu_z}$.  By \eqref{eq:cole-hopf-step} and
\zcref{lem:cutoff-splitting}, the Parisi solution minus $k_p$ is obtained
from $\psi^{(p)}_m$ by the backward steps of $\mathfrak d_p(\theta,Q)$,
because $T(g+c)=Tg+c$ (\zcref{lem:one-step}(i)).  The two diffusions
therefore solve the same equation, whose drift is Lipschitz because
$u_{xx}$ is bounded, and they coincide by pathwise uniqueness.  In
particular, $f_{\Psi,\mu_{z,m,Q}}=f^{\mathfrak d_p(\theta,Q)}$ on
$[0,\bar q]$.  (ii) is \zcref{lem:temperature-mass-scaling} with
$q_{\mathrm c}=\bar q$, applied to the admissible covariance $\Psi$, and
(iii) is \zcref{lem:temperature-mass-diffusion} with $q_{\mathrm c}=\bar q$,
applied to $\Xi$; in the notation of these lemmas,
$\mathfrak d^\circ_\theta=\mathfrak d^\circ_{\nu_z,m}$ with cutoff
$\bar q$.  The last assertion of~(iii) follows from the definition of
$f^\circ_m$ in (S5).
\end{proof}

Between the cutoff and $Q$, the cumulative mass of the candidate minimizer
is $m$,
and there the optimal diffusion is governed by a single Cole--Hopf step.

\begin{lemma}\label{lem:deficit-representation}
For every sufficiently large even $p$ and every
$(\theta,Q)\in\Theta\times S_p$, let $X$ and $\Gamma$ be the optimal
diffusion and the mean squared magnetization of
$\mu_{z,m,Q}$ for $\Psi$, and let $v_t=T_{m,W_p-t}\log\cosh$ for
$0\leq t\leq W_p$.  For $\bar q<q\leq Q$, set
$t=\Psi'(q)-\Psi'(\bar q)$.  Then, conditionally on $X_{\bar q}=x$, the law
of $X_q$ is $K_{m,t,v_t}(x,\cdot)$, and
\begin{equation}\label{eq:deficit-representation}
 1-\Gamma(q)=\E\bigl[1-v_t'(X_q)^2\bigr].
\end{equation}
At $q=Q$, we have $t=W_p$ and $v_{W_p}'=\tanh$, so $1-\Gamma(Q)=\E\sech^2(X_Q)$.
\end{lemma}

\begin{proof}
On $[\bar q,Q)$, the cumulative mass is $m$, and, by \eqref{eq:top-solution},
$u_{p,x}(q',\cdot)=v_s'$ for $\bar q\leq q'<Q$, where
$s=\Psi'(q')-\Psi'(\bar q)$.  We pass to the variance variable $s$, for
which $\dd s=\Psi''(q')\dd q'$, and let $q(s)$ be the inverse map.  Then
$\xi_s=X_{q(s)}$ solves $\dd\xi_s=mv_s'(\xi_s)\dd s+\dd\tilde B_s$ on $[0,t]$ for
a standard Brownian motion $\tilde B$.  Since $v_s=T_{m,t-s}v_t$ for
$0\leq s\leq t$ by \zcref{lem:one-step}(i), this is
the diffusion of \zcref{lem:one-step}(iii) with level $m$, variance $t$, and
terminal function $v_t$.  Then, conditionally on $X_{\bar q}=x$, the law
of $X_q$ is $K_{m,t,v_t}(x,\cdot)$.  Finally,
$1-\Gamma(q)=\E[1-u_{p,x}(q,X_q)^2]=\E[1-v_t'(X_q)^2]$.
\end{proof}

\subsection{The functional remainder}

\begin{proof}[Proof of \zcref{lem:lower-recursion}]
\noindent\emph{Exact decomposition.}
By \zcref{lem:below-cutoff}(i) and (ii),
\begin{align*}
 u_p(0,0)&=k_p+u^{\mathfrak d_p(\theta,Q)}(0,0),\\
 \frac1mu_{m^2\Psi,\nu_z}(0,0)
 &=u^{\mathfrak d^\Psi_\theta}(0,0)
 +\frac m2\bigl(\Psi'(1)-\Psi'(\bar q)\bigr).
\end{align*}
The cumulative mass of $\mu_{z,m,Q}$ is
$m\alpha_{\nu_z}+(1-m)\mathbf 1_{[Q,1]}$.  Inserting these facts and the
formula for $k_p$ from the proof of \zcref{lem:cutoff-splitting} into the
definitions of $\Pp_p(\theta,Q)$ and of
$\frac1m\Pp_{m^2\Psi}(\nu_z)=\frac1m(\log2+u_{m^2\Psi,\nu_z}(0,0))
-\frac m2\int_0^1q\Psi''\alpha_{\nu_z}\dd q$, both from
\eqref{eq:parisi-functional}, we find that
$\Pp_p(\theta,Q)-\frac1m\Pp_{m^2\Psi}(\nu_z)$ equals
\begin{align*}
 &\Bigl[\log2-\frac{\log2}m+\Bigl(\frac1m-1\Bigr)\log2\Bigr]
 +\Bigl[\frac12\bigl(\Psi'(1)-\Psi'(Q)\bigr)+\frac m2W_p
 -\frac m2\bigl(\Psi'(1)-\Psi'(\bar q)\bigr)\Bigr]\\
 &\qquad-\frac{1-m}2\int_Q^1q\Psi''(q)\dd q
 +\bigl(u^{\mathfrak d_p(\theta,Q)}-u^{\mathfrak d^\Psi_\theta}\bigr)(0,0).
\end{align*}
The first bracket vanishes, and since $W_p=\Psi'(Q)-\Psi'(\bar q)$, the
second equals $\frac{1-m}2(\Psi'(1)-\Psi'(Q))$; together with the integral,
it gives $\mathcal T_p$.  We obtain
\begin{equation}\label{eq:candidate-identity}
 \Pp_p(\theta,Q)=\frac1m\Pp_{m^2\Psi}(\nu_z)+\mathcal T_p(\theta,Q)
 +\bigl(u^{\mathfrak d_p(\theta,Q)}-u^{\mathfrak d^\Psi_\theta}\bigr)(0,0).
\end{equation}

\medskip
\noindent\emph{Changing the terminal function.}
The last term compares two recursions that differ only in their terminal
functions $\psi_m^{(p)}$ and $\psi_m$.  We apply
\zcref{lem:cole-hopf-smooth}(ii) with $\ell=4$ and $\varepsilon=Cp^Ae^{-cp}$
from \eqref{eq:lower-terminal-approximation} on the domain
$\Theta\times S_p$, which is the closure of a convex open set.  Its
hypotheses hold uniformly in
$p$, as follows.
\begin{itemize}
 \item \emph{Levels.}  The levels $m\alpha_{\nu_z}(q_i)$ depend smoothly on
 $(z,m)$, and they are at least $m_0$ times the mass of the atom of $\nu_z$
 at $0$, which is bounded below by (S3).
 \item \emph{Variances.}  The variances $\Psi'(q_{i+1})-\Psi'(q_i)$, with
 $q_{r+1}=\bar q$, depend only on $(z,\lambda)$.  They lie in a fixed
 compact subset of $(0,\infty)$, because the support gaps of $\nu_z$ and
 the gap $\bar q-a(z)\geq\bar q-a_*$ are bounded below and, by
 \eqref{eq:high-term-small-fixed}, the contributions of $\lambda q^p$ and
 its derivatives are exponentially small.
 \item \emph{Terminal functions.}  The functions $\psi_m^{(p)}$ and
 $\psi_m$ are one-Lipschitz, and their parameter derivatives are bounded in
 $\|\cdot\|_3$ by \zcref{lem:cutoff-splitting} and \zcref{rem:psi-regularity}.
\end{itemize}
Since the constant in \zcref{lem:cole-hopf-smooth} does not depend on the
domain, which here varies with $p$, the last term in
\eqref{eq:candidate-identity} and its derivatives through
order four are bounded by $Cp^Ae^{-cp}$.

\medskip
\noindent\emph{Changing the covariance from $\Psi$ to $\Xi$.}
It remains to compare $\frac1m\Pp_{m^2\Psi}(\nu_z)$ with
$F(\theta)=\frac1m\Pp_{m^2\Xi}(\nu_z)+\frac{\lambda m}2$.  Let
$\Upsilon$ be an admissible covariance.  Since $\alpha_{\nu_z}=1$ on
$[a(z),1]$, \eqref{eq:cole-hopf-step} and \eqref{eq:T-logcosh}, together
with \eqref{eq:parisi-functional}, split the Parisi functional of $\nu_z$
for $\Upsilon$ into a contribution determined by the recursion on
$[0,a(z)]$ and an explicit integral over $[a(z),1]$:
\[
 \Pp_\Upsilon(\nu_z)=\mathcal B_\Upsilon(z)
 +\frac12\int_{a(z)}^1(1-q)\Upsilon''(q)\dd q,
 \qquad
 \mathcal B_\Upsilon(z)=\log2+\tilde u^\Upsilon(0,0)
 -\frac12\int_0^{a(z)}q\Upsilon''\alpha_{\nu_z}\dd q,
\]
where $\tilde u^\Upsilon$ is obtained from $\log\cosh$ by the Cole--Hopf steps of
$\nu_z$ on $[0,a(z)]$ with covariance $\Upsilon$.  We apply this with
$\Upsilon=m^2\Psi$ and $\Upsilon=m^2\Xi$, which differ by $m^2\lambda q^p$.
If $r=0$, then $a(z)=0$, and both $\mathcal B$ terms equal $\log2$.  If $r\geq1$,
the two recursions have the same terminal function $\log\cosh$ and the same
levels $\alpha_{\nu_z}(q_i)$, $0\leq i<r$, and their variances are
$m^2(\Psi'(q_{i+1})-\Psi'(q_i))$ and $m^2(\Xi'(q_{i+1})-\Xi'(q_i))$.  As
in the list above, the levels are bounded below by (S3) and the variances
lie in a fixed compact subset of $(0,\infty)$, so the hypotheses of
\zcref{lem:cole-hopf-smooth} hold on $\Theta$ uniformly in $p$.  Since
$q_i\leq a(z)\leq a_*$, the variances differ by
$m^2\lambda p(q_{i+1}^{p-1}-q_i^{p-1})$, which, together with its
derivatives through order four, is bounded in absolute value by
$Cp^Aa_*^{p-5}$.  Indeed, for
$0\leq k\leq4$,
\[
 \frac{\dd^k}{\dd q^k}(pq^{p-1})
 =p\prod_{h=1}^k(p-h)\,q^{p-1-k},
 \qquad
 \sup_{0\leq q\leq a_*}
 \left|\frac{\dd^k}{\dd q^k}(pq^{p-1})\right|
 \leq Cp^5a_*^{p-5}.
\]
Derivatives of the prefactors in $(m,\lambda)$ are uniformly bounded, so
the same estimate applies to mixed parameter derivatives through order
four.  By
\zcref{lem:cole-hopf-smooth}(ii), the same bound holds for
$\tilde u^{m^2\Psi}(0,0)-\tilde u^{m^2\Xi}(0,0)$.  The integrals in the
$\mathcal B$ terms differ by
\[
 \frac{m^2\lambda}2\int_0^{a(z)}p(p-1)q^{p-1}\alpha_{\nu_z}(q)\dd q
 =\frac{m^2\lambda}2(p-1)\sum_{i=0}^{r-1}\alpha_{\nu_z}(q_i)
 \bigl(q_{i+1}^p-q_i^p\bigr),
\]
which obeys the same bound.  Then the $\mathcal B$ terms and their derivatives
through order four differ by at most $Cp^Aa_*^{p-5}\leq Cp^Ae^{-cp}$.

\medskip
\noindent\emph{The explicit integral contribution.}
The integral contributions differ by
\[
 \frac{m^2\lambda}2\int_{a(z)}^1(1-q)p(p-1)q^{p-2}\dd q
 =\frac{m^2\lambda}2\bigl(1-pa(z)^{p-1}+(p-1)a(z)^p\bigr).
\]
After division by $m$, the leading term of this difference is
$\frac{\lambda m}2$; this is the origin of the last term of $F(\theta)$.
By the same differentiation estimate, the rest and its derivatives
through order four are bounded by $Cp^Aa_*^{p-5}$.  Then
$\frac1m\Pp_{m^2\Psi}(\nu_z)-F(\theta)$ does not depend on $Q$ and is
bounded, together with its derivatives through order four, by
$Cp^Ae^{-cp}$; it is $C^4$ on $\Theta$, because the $\mathcal B$ terms are $C^4$ by
\zcref{lem:cole-hopf-smooth}(i) and the remaining integrals are explicit.  By
\eqref{eq:candidate-identity}, $\mathcal R_p$ is the sum of this function
and of the last term in \eqref{eq:candidate-identity}, which is $C^4$ by
\zcref{lem:cole-hopf-smooth}(i).  This proves the lemma.
\end{proof}

\subsection{The magnetization deficit}

\begin{proof}[Proof of \zcref{lem:deficit}]
At $m=1$, neither $\Pp_p$ nor $F$ nor $\mathcal T_p$ depends on $Q$, and
$\partial_Q\mathcal R_p$ therefore vanishes there.  For $m<1$, the
fundamental theorem of calculus and the substitution $\sigma=1-(1-m)\tau$
give
\[
 \partial_Q\mathcal R_p(z,\lambda,m,Q)
 =-\int_m^1\partial_m\partial_Q\mathcal R_p(z,\lambda,\sigma,Q)\dd\sigma
 =-(1-m)\int_0^1\partial_m\partial_Q\mathcal R_p\bigl(z,\lambda,
 1-(1-m)\tau,Q\bigr)\dd\tau.
\]
Dividing the second identity in \eqref{eq:balance-derivatives} by
$\frac12(1-m)\Psi''(Q)$, we obtain, for $m<1$,
\begin{equation}\label{eq:deficit-remainder-identity}
 \mathcal A_p(\theta,Q)
 =-\frac2{\Psi''(Q)}\int_0^1
 \partial_m\partial_Q\mathcal R_p\bigl(z,\lambda,1-(1-m)\tau,Q\bigr)\dd\tau.
\end{equation}

Both sides of \eqref{eq:deficit-remainder-identity} are continuous in $m$
up to $m=1$.  Continuity of the right side follows from continuity of
$\partial_m\partial_Q\mathcal R_p$ on the compact set $\Theta\times S_p$.
For the left side,
we use \zcref{lem:observable-continuity} with $q_{\mathrm c}=1$.  The data
of $(\Psi,\mu_{z,m,Q})$ have points $q_0<\cdots<q_r<Q$, levels
$m\alpha_{\nu_z}(q_i)$ and $1$, and terminal function $\log\cosh$.  As
$m\to1$, the levels converge to $\alpha_{\nu_z}(q_i)$ and $1$.  Since
$\alpha_{\nu_z}(q_r)=1$, the limiting data have the same level on both
sides of $Q$.  They are therefore the data of $(\Psi,\nu_z)$ with
the final interval split at $Q$, which changes neither the Parisi
solution nor the optimal diffusion.  Hence
$\mathcal A_p(z,\lambda,m,Q)\to1-\Gamma_{\Psi,\nu_z}(Q)=\mathcal
A_p(z,\lambda,1,Q)$, and \eqref{eq:deficit-remainder-identity} also holds
at $m=1$.

Let $\mathcal I(\theta,Q)$ be the integral in
\eqref{eq:deficit-remainder-identity}.  Its derivatives of order at most
two are integrals of derivatives of $\mathcal R_p$ of order at most four,
multiplied by powers of $\tau\in[0,1]$.  Then $\mathcal I$ is $C^2$, and its
derivatives of order at most two are bounded by $Cp^Ae^{-cp}$ by
\eqref{eq:lower-remainder-bound}.  On $S_p$, we have
$Q^{p-2}\geq1-p(1-Q)\geq1-pe^{-m_0^2\lambda_-p/4}\geq\frac12$ for large $p$.
Hence $\Psi''(Q)\geq\lambda_-p(p-1)Q^{p-2}\geq cp^2$.
The derivatives of $\Psi''(Q)$ in $(\lambda,Q)$ through order two
are bounded by $Cp^4$, so those of $1/\Psi''(Q)$ are bounded by $Cp^2$.
By the product rule,
$\mathcal A_p=-2\mathcal I/\Psi''(Q)$ is $C^2$ and satisfies
\eqref{eq:endpoint-crude}.

To estimate the size of $\mathcal A_p$, we write $W=W_p$ and let
$K^\infty_W(x,\cdot)$ be the comparison kernel of \zcref{lem:doob-kernel}.
By \zcref{lem:deficit-representation} at $q=Q$ and by
\eqref{eq:Doob-comparison}, the deficit $\mathcal A_p$ lies between $\frac14$
and $4$
times $\E\int_{\R}\sech^2y\,K^\infty_W(X_{\bar q},\dd y)$.
By \eqref{eq:Kinfty-mixture}, for each $x$,
\begin{equation}\label{eq:comparison-deficit-integral}
 \int_{\R}\sech^2y\,K^\infty_W(x,\dd y)
 =\frac{e^{-m^2W/2}}{\cosh(mx)}\int_{\R}\mathsf p_W(x,y)\sech^2y\cosh(my)\dd y.
\end{equation}
Since $\sech^2y\cosh(my)\leq\sech y$ and $\int\sech=\pi$, the left side is
at most $\pi(2\pi W)^{-1/2}e^{-m^2W/2}$ (compare
\cite[Lemma~5.1]{JagannathTobasco2017} and
\cite[Lemma~C.9]{Zhou2026GardnerII}).  For the lower bound, recall that the optimal diffusion starts at
$X_0=0$.  Since $|u_{p,x}|\leq1$ and the cumulative mass is at most $1$,
its integrated drift over $[0,\bar q]$ is bounded in absolute value by
$\Psi'(\bar q)$, which is also the variance of its martingale part.  For
large $p$, we have
$\Psi'(\bar q)\leq\Xi'(1)+\lambda_+p\bar q^{\,p-1}\leq\Xi'(1)+1$, and
hence $\E X_{\bar q}^2\leq2\Psi'(\bar q)^2+2\Psi'(\bar q)\leq\frac12\bar
x^2$ for a constant $\bar x>0$ that depends only on $\Xi$.  By Chebyshev's
inequality, $\mathbb P(|X_{\bar q}|\leq\bar x)\geq\frac12$ for all
$(\theta,Q)\in\Theta\times S_p$.  Let $|x|\leq\bar x$, and note that
$W\geq1$ for large $p$ by \eqref{eq:endpoint-input-W}.  For $|y|\leq1$, we
have $\mathsf p_W(x,y)\geq(2\pi W)^{-1/2}e^{-(\bar x+1)^2/2}$ and
$\sech^2y\cosh(my)\geq\sech^21$, and $\cosh(mx)\leq\cosh\bar x$.
Restricting the integral on the right of \eqref{eq:comparison-deficit-integral}
to $|y|\leq1$, we find that its left side is at least
$cW^{-1/2}e^{-m^2W/2}$, with $c>0$
depending only on $\bar x$.  Since $\mathbb P(|X_{\bar q}|\leq\bar x)\geq\frac12$, the expectation above is at least
$\frac c2W^{-1/2}e^{-m^2W/2}$.  Together with
\eqref{eq:endpoint-input-W}, these bounds imply
\eqref{eq:endpoint-two-sided}.

Finally, we prove \eqref{eq:deficit-intermediate}.  At $q=\bar q$, it holds
because $\Gamma\geq0$.  Let $\bar q<q\leq Q$ and
$t=\Psi'(q)-\Psi'(\bar q)$.  By \zcref{lem:tanh-comparison}, applied with
$V=W-t$ to $v_t=T_{m,W-t}\log\cosh$, we have
$1-v_t'(y)^2\leq\sech^2(my)$.  For every $x$, we also have
\[
 \int_{\R}\sech^2(my)K^\infty_t(x,\dd y)
 =
 e^{-m^2t/2}
 \frac{\mathsf H_t(\sech(m\,\cdot))(x)}{\cosh(mx)}
 \leq
 e^{-m^2t/2}.
\]
Combining these two bounds with \eqref{eq:Doob-comparison} and
\zcref{lem:deficit-representation}, we obtain
\eqref{eq:deficit-intermediate}.
\end{proof}

\subsection{Regularity of the reductions}\label{s:reduction-proofs}

On the closed parameter sets of \zcref{sec:high-degree-perturbation}, we
use the implicit function theorem in the following form.  It allows the
parameter to reach the boundary of its domain, provided that the solution
stays in the interior of the set where it is sought.

\begin{lemma}\label{lem:ift-closed}
Let $D\subset\R^d$ and $K\subset\R^n$ be the compact closures of convex
open sets, let $k\geq1$, and let $\Phi\colon D\times K\to\R^n$ be $C^k$.
Suppose that, for every $\theta\in D$, the equation $\Phi(\theta,y)=0$ has
exactly one solution $y=y(\theta)$ in $K$, that this solution lies in the
interior of $K$, and that $D_y\Phi(\theta,y(\theta))$ is invertible.  Then
$y$ is $C^k$ on $D$, and its derivatives are obtained by differentiating
the identity $\Phi(\theta,y(\theta))=0$.
\end{lemma}

\begin{proof}
The set $D\times K$ is the compact closure of a convex open set.  The
integral form of Taylor's theorem applies to $\Phi$ and its derivatives
through order $k-1$ along every segment in $D\times K$.
This is clear for segments in the interior, and
every other segment is a limit of such segments, obtained by contracting
it slightly toward an interior point.  Together with the uniform
continuity of the derivatives of order $k$ of $\Phi$ on the compact set
$D\times K$, these Taylor formulas give Whitney's compatibility
conditions, with remainders $o(|\xi-\xi'|^{k-|\mathbf j|})$ for the
derivatives of order $|\mathbf j|\leq k$, uniformly in
$\xi,\xi'\in D\times K$.  By Whitney's extension theorem \cite{Whitney1934}, $\Phi$
extends to a $C^k$ function $\tilde\Phi$ on an open neighborhood of
$D\times K$.  Fix $\theta_0\in D$.  Since
$D_y\tilde\Phi=D_y\Phi$ is invertible at $(\theta_0,y(\theta_0))$, the
implicit function theorem gives a neighborhood $U$ of $\theta_0$ and a
$C^k$ function $\tilde y$ on $U$ with $\tilde y(\theta_0)=y(\theta_0)$ and
$\tilde\Phi(\theta,\tilde y(\theta))=0$.  Because $y(\theta_0)$ lies in the
interior of $K$, we may shrink $U$ so that $\tilde y(U)\subset K$.  For
$\theta\in U\cap D$, the point $\tilde y(\theta)$ is then a solution of
$\Phi(\theta,\cdot)=0$ in $K$, and by uniqueness, we have
$\tilde y(\theta)=y(\theta)$.  We conclude that $y$ agrees near each point
of $D$ with a $C^k$ function whose derivatives are obtained by
differentiating $\tilde\Phi(\theta,\tilde y(\theta))=0$.  This completes
the proof.
\end{proof}

\begin{proof}[Proof of \zcref{prop:reduced-functional}]
Let $\phi_p(\theta,Q)=Q-1+\mathcal A_p(\theta,Q)$, as in the proof of
\zcref{prop:uniform-reduction}.  By \zcref{lem:deficit}, $\phi_p$ is $C^2$
on $\Theta\times S_p$, and $\partial_Q\phi_p\geq\frac12$ there for large
$p$ by \eqref{eq:endpoint-crude}.  By
\zcref{prop:uniform-reduction}, $\phi_p(\theta,\cdot)$ has exactly
one zero $Q_p(\theta)$ in $S_p$, which lies in the interior.

\medskip
\noindent\emph{Regularity of $Q_p$.}
We apply \zcref{lem:ift-closed} with $D=\Theta$, $K=S_p$, and
$\Phi=\phi_p$; when $r=0$, we drop the fixed coordinate $z=z_*$ from
$\theta$.  The hypotheses of \zcref{lem:ift-closed} were verified above, so $Q_p$ is $C^2$ on
$\Theta$, and, for each coordinate $\theta_j$,
\[
 \partial_jQ_p
 =-\frac{\partial_j\mathcal A_p}
 {1+\partial_Q\mathcal A_p},
\]
with the derivatives of $\mathcal A_p$ evaluated at
$(\theta,Q_p(\theta))$.  Differentiating the location equation twice gives
\[
 \begin{aligned}
 \partial_i\partial_jQ_p
 =-\frac{1}{1+\partial_Q\mathcal A_p}\Bigl(&
 \partial_i\partial_j\mathcal A_p
 +\partial_i\partial_Q\mathcal A_p\,\partial_jQ_p\\
 &+\partial_j\partial_Q\mathcal A_p\,\partial_iQ_p
 +\partial_Q^2\mathcal A_p\,\partial_iQ_p\,\partial_jQ_p\Bigr),
 \end{aligned}
\]
with all derivatives of $\mathcal A_p$ evaluated at
$(\theta,Q_p(\theta))$.  Since $1+\partial_Q\mathcal A_p\geq\frac12$,
\eqref{eq:endpoint-crude} and the first-derivative formula imply
\begin{equation}\label{eq:Q-parameter-bounds}
 \max_{1\leq|\mathbf j|\leq2}\sup_{\theta\in\Theta}
 \bigl|\partial_\theta^{\mathbf j}Q_p(\theta)\bigr|
 \leq Cp^Ae^{-cp}.
\end{equation}

\medskip
\noindent\emph{Reducing the functional.}
For $m<1$, \eqref{eq:Q-derivative} and the location equation give
$\partial_Q\Pp_p(\theta,Q_p(\theta))=0$.
The same identity holds at $m=1$ because the functional is then
independent of $Q$.
By the chain rule, we have
\[
 \partial_j\widehat\Pp_p(\theta)
 =\partial_j\Pp_p(\theta,Q_p(\theta))
 +\partial_Q\Pp_p(\theta,Q_p(\theta))\,\partial_jQ_p(\theta)
 =\partial_j\Pp_p(\theta,Q_p(\theta)),
\]
which is \eqref{eq:envelope}.
Since $F$ is smooth by \zcref{prop:finite-step-regularity} and
$\mathcal T_p$ is explicit, \zcref{lem:lower-recursion} shows that
$\Pp_p=F+\mathcal T_p+\mathcal R_p$ is $C^4$.
Since $Q_p$ is $C^2$, the right side of \eqref{eq:envelope} is $C^2$.
Then every first derivative of $\widehat\Pp_p$ is $C^2$, which proves
the $C^3$ regularity in part~(i).

For part~(ii), set
\[
 \Delta_p=\Pp_p-F=\mathcal T_p+\mathcal R_p,
 \qquad
 \widehat\Delta_p=\widehat\Pp_p-F.
\]
At fixed $Q$, $\mathcal T_p$ is independent of $z$ and affine
in each of $m$ and $\lambda$.
Since $\Psi''\leq Cp^2$, the derivatives of $\mathcal T_p$ in $\theta$ alone through
order three are bounded by $Cp^2(1-Q)^2$, hence are exponentially
small on $S_p$.
Its mixed derivatives involving $Q$ through order three are bounded
by powers of $p$, since $\Psi^{(j)}=O(p^j)$ on $[0,1]$ for $j\leq4$.
Together with \eqref{eq:lower-remainder-bound}, this gives exponentially
small bounds for the pure $\theta$-derivatives of $\Delta_p$ and
polynomial bounds for all its derivatives through order three.

Since $F$ is independent of $Q$, we have
$\partial_Q\Delta_p=\partial_Q\Pp_p=0$ at $Q=Q_p(\theta)$.
By the chain rule, we obtain
\[
 \partial_j\widehat\Delta_p(\theta)
 =\partial_j\Delta_p(\theta,Q_p(\theta)),
\]
where the derivative on the right is taken with $Q$ fixed. 
Differentiating this identity once and twice, we obtain
\begin{equation}\label{eq:reduced-chain-rule}
 \begin{aligned}
 \partial_k\partial_j\widehat\Delta_p
 &=\partial_k\partial_j\Delta_p
   +\partial_Q\partial_j\Delta_p\,\partial_kQ_p,\\
 \partial_l\partial_k\partial_j\widehat\Delta_p
 &=\partial_l\partial_k\partial_j\Delta_p
   +\partial_Q\partial_k\partial_j\Delta_p\,\partial_lQ_p
   +\partial_Q\partial_l\partial_j\Delta_p\,\partial_kQ_p\\
 &\quad
   +\partial_Q^2\partial_j\Delta_p\,\partial_kQ_p\,\partial_lQ_p
   +\partial_Q\partial_j\Delta_p\,\partial_l\partial_kQ_p,
 \end{aligned}
\end{equation}
where all derivatives of $\Delta_p$ are evaluated at
$(\theta,Q_p(\theta))$.
The pure $\theta$-derivative in each formula is exponentially small.
Every other term contains a polynomially bounded derivative of
$\Delta_p$ and at least one exponentially small first or second
derivative of $Q_p$, by \eqref{eq:Q-parameter-bounds}.
The value and first derivatives obey the same bound by the
decomposition and \eqref{eq:envelope}.
This proves \eqref{eq:transition-C3-limit}, without requiring third
derivatives of $Q_p$.
\end{proof}

\begin{proof}[Proof of \zcref{prop:stationary-branches}]
By \zcref{prop:reduced-functional}(ii), $\widehat\Pp_p\to F$ in $C^3(\Theta)$.

\medskip
\noindent\emph{Part (i).}
Let $r\geq1$.  Since $D_z^2F_\lambda\succeq2\kappa_{\mathrm z}I$
on $\cN\times[m_0,1]$ by (S1), and
$\partial_zF_\lambda(z(m),m)=0$ by
\eqref{eq:limiting-geometry}, for large $p$ we have
$D_z^2\widehat\Pp_p\succeq\kappa_{\mathrm z}I$ on $\Theta$ and
$\varepsilon_p=\sup|\partial_z\widehat\Pp_p(z(m),\lambda,m)|
<\kappa_{\mathrm z}\rho_0/2$.  Since $|z-z(m)|\geq\rho_0$ on $\partial\cN$, strong
convexity gives
$\widehat\Pp_p(z,\lambda,m)-\widehat\Pp_p(z(m),\lambda,m)
\geq|z-z(m)|(\frac{\kappa_{\mathrm z}}2|z-z(m)|-\varepsilon_p)>0$ there.  Then the
minimizer $z_p(\lambda,m)$ of $\widehat\Pp_p(\cdot,\lambda,m)$ on $\cN$ is
interior, is the unique zero of $\partial_z\widehat\Pp_p(\cdot,\lambda,m)$,
and satisfies $\kappa_{\mathrm z}|z_p(\lambda,m)-z(m)|\leq\varepsilon_p\to0$.  Since $\widehat\Pp_p$ is $C^3$, \zcref{lem:ift-closed}, applied with
$D=\Lambda\times[m_0,1]$, $K=\cN$, and
$\Phi(\lambda,m,z)=\partial_z\widehat\Pp_p(z,\lambda,m)$, shows that $z_p$
is $C^2$; its hypotheses hold because the zero is unique and interior and
$D_z^2\widehat\Pp_p$ is positive definite.
Differentiating $\partial_z\widehat\Pp_p(z_p(\lambda,m),\lambda,m)=0$ once
and twice expresses $Dz_p$ and $D^2z_p$ as polynomials in
$(D_z^2\widehat\Pp_p)^{-1}$ and the derivatives of $\widehat\Pp_p$ of order
at most three, evaluated at $(z_p(\lambda,m),\lambda,m)$.  The same
polynomials, with $F$ in place of $\widehat\Pp_p$ and evaluated at
$(z(m),\lambda,m)$, give the derivatives of $(\lambda,m)\mapsto z(m)$,
because $\partial_zF(z(m),\lambda,m)=0$.  The derivatives of $\widehat\Pp_p$
and of $F$ of order at most three differ by at most $Cp^Ae^{-cp}$ by
\zcref{prop:reduced-functional}.  Those of $F$ at the two points differ by
$o(1)$, because they are uniformly continuous on the compact set $\Theta$
and $|z_p(\lambda,m)-z(m)|\to0$ uniformly.  Writing
$H_p=D_z^2\widehat\Pp_p(z_p(\lambda,m),\lambda,m)$ and
$H=D_z^2F_\lambda(z(m),m)$, the inverses are bounded by
$\kappa_{\mathrm z}^{-1}$, and
\begin{align*}
 H_p^{-1}-H^{-1}&=H_p^{-1}(H-H_p)H^{-1},\\
 \|H_p^{-1}-H^{-1}\|
 &\leq\kappa_{\mathrm z}^{-2}\|H_p-H\|\longrightarrow0
\end{align*}
uniformly.  The derivative formulas above therefore give
$z_p\to z(m)$ in $C^2(\Lambda\times[m_0,1])$.

\medskip
\noindent\emph{Part (ii).}
Since $\partial_z\widehat\Pp_p$ vanishes along the graph of $z_p$, the
chain rule gives the following identities
$\partial_jJ_p(\lambda,m)=\partial_j\widehat\Pp_p(z_p(\lambda,m),\lambda,m)$
for $j\in\{\lambda,m\}$; the case $j=m$ is the formula for $g_p$.  Their
right sides are $C^2$ by part~(i).  It follows that $J_p$ is $C^3$ and
that its second and third derivatives involve only the derivatives of $\widehat\Pp_p$ through
order three and of $z_p$ through order two.  By
\eqref{eq:limiting-geometry}, the same identities hold for
$J_\lambda(m)=F_\lambda(z(m),m)$.  Using part~(i) and the $C^3$
convergence, we have $J_p\to J_\lambda$ in $C^3$, and in particular
$g_p\to J_\lambda'$ in $C^2$.
\end{proof}

\subsection{First-variation functions below the cutoff}\label{s:lower-geometry-proof}

\begin{proof}[Proof of \zcref{lem:lower-geometry}]
Suppose that the lemma fails.  Then there exist even $p_n\to\infty$ and
$(\lambda_n,m_n)\in\Lambda\times[m_0,1]$ at which one of its assertions
fails.  Passing to a subsequence, we may assume that $m_n\to m_\infty$;
then $z_n=z_{p_n}(\lambda_n,m_n)\to z(m_\infty)$ by
\zcref{prop:stationary-branches}(i).  We write $\Psi_n=\Xi+\lambda_nq^{p_n}$,
$\theta_n=(z_n,\lambda_n,m_n)$, and $Q_n=Q_{p_n}(\theta_n)$.  By
\zcref{lem:below-cutoff}(i), on $[0,\bar q]$ the first-variation function
$f_n$ of $\mu_{p_n}(\lambda_n,m_n)$ is $f^{\mathfrak d_n}$ with
$\mathfrak d_n=\mathfrak d_{p_n}(\theta_n,Q_n)$, and, by
\zcref{lem:below-cutoff}(iii), $f^\circ_{m_\infty}=f^{\mathfrak d^\circ}$
with $\mathfrak d^\circ=\mathfrak d^\circ_{\nu_{z(m_\infty)},m_\infty}$,
taken with cutoff $\bar q$.
The data $\mathfrak d_n$ converge to $\mathfrak d^\circ$ in the sense of
\zcref{lem:observable-continuity}.
Indeed, the covariances $\Psi_n$ converge to $\Xi$ in $C^3([0,\bar q])$ by
\eqref{eq:high-term-small-fixed}.  The atom locations and cumulative
masses of $m_n\nu_{z_n}$ converge to those of
$m_\infty\nu_{z(m_\infty)}$, and the limiting points, together with
the cutoff $\bar q$, are strictly increasing because
$z(m_\infty)\in\mathcal Z_r$ and $\bar q>a_*$.  Finally, by
\eqref{eq:lower-terminal-approximation} and
\zcref{rem:psi-regularity}, the derivatives of orders one to three of the
one-Lipschitz terminal functions $\psi^{(p_n)}_{m_n}$ converge uniformly
on $\R$ to those of $\psi_{m_\infty}$.

We now check the hypotheses of \zcref{lem:contact-persistence}.
Condition~(i) holds for $f_\infty=f^\circ_{m_\infty}$ by (S5).  For
condition~(ii), since $\Xi''(0)<1$, there exists $\varepsilon>0$ such that
$\Xi''\leq1-2\varepsilon$ on $[0,\varepsilon]$, and then $\Psi_n''\leq1-\varepsilon$
on $[0,\varepsilon]$ for large $n$.  Since $\Gamma^{\mathfrak d_n}$ is the
mean squared magnetization of the atomic measure $\mu_{p_n}(\lambda_n,m_n)$ for the
admissible covariance $\Psi_n$, \zcref{rem:origin-condition}
applies.  Condition~(iii) holds by
\zcref{lem:stationary-candidates}.  By \zcref{lem:contact-persistence}, for
large $n$, we have $f_n\leq0$ on
$[0,\bar q]$ with zero set $\supp\nu_{z_n}$, and $f_n''<0$ at the positive
support points.  Further,
$f_n(\bar q)\to f^\circ_{m_\infty}(\bar q)\leq-4\gamma$ by
\zcref{lem:observable-continuity} and \eqref{eq:cutoff-choice}.  Then all
assertions hold at $(p_n,\lambda_n,m_n)$ for large $n$, a contradiction.
\end{proof}

\section{Large-variance Cole--Hopf estimates}\label{app:large-variance}

We collect three estimates for $T_{m,V}\log\cosh$ used in
\zcref{sec:high-degree-perturbation}.  They follow from the identity
\[
 T_{m,V}\log\cosh=\frac1m\log\mathsf H_V(\cosh^m)
\]
and comparisons of $\cosh^m$ with the exponentials $e^{\pm m\,\cdot}$,
for which $\mathsf H_V(e^{\pm m\,\cdot})(x)=e^{m^2V/2\pm mx}$.

\begin{lemma}\label{lem:large-V}
Let $I\subset(0,1]$ be nonempty and compact, and let $c=\frac12(\min I)^2$.  There
exists $\overline V\geq1$, depending only on $I$, with the following
property.  For every integer $\ell\geq0$, there exists $C_\ell<\infty$,
depending only on
$I$ and $\ell$, such that, for all $m\in I$, $V\geq\overline V$, and
$x\in\R$,
\begin{equation}\label{eq:large-V}
 (T_{m,V}\log\cosh)(x)
 =\frac{mV}{2}+\left(\frac1m-1\right)\log2
 +\frac1m\log\cosh(mx)+\mathcal E_{m,V}(x),
\end{equation}
where
\begin{equation}\label{eq:large-V-error}
 \max_{i+j+k\leq\ell}\sup_{x\in\R}
 \left|
 \partial_x^i\partial_m^j\partial_V^k
 \mathcal E_{m,V}(x)
 \right|
 \leq C_\ell V^\ell e^{-cV}.
\end{equation}
\end{lemma}

The main term in \eqref{eq:large-V} is obtained by replacing $\cosh^m$ with
$2^{-m}(e^{m\,\cdot}+e^{-m\,\cdot})$, for which the heat semigroup can be
computed exactly.  The difference between $\cosh^m$ and the exponential approximation
is uniformly bounded.  The heat semigroup keeps this difference
bounded while multiplying the two exponentials by $e^{m^2V/2}$,
so the relative error is suppressed by $e^{-m^2V/2}$.

\begin{proof}
We write $\cosh^my=2^{-m}(e^{my}+e^{-my})+\varrho_m(y)$.  We first show that
every fixed mixed $(y,m)$-derivative of $\varrho_m$ is bounded on
$\R\times I$.  Since $\cosh>0$, the function $\varrho_m(y)$ is smooth in
$(y,m)$, and it is even in $y$.  It suffices to consider $y\geq0$.  There,
\[
 \varrho_m(y)
 =2^{-m}e^{-(2-m)y}\phi_m(e^{-2y})-2^{-m}e^{-my},
 \qquad
 \phi_m(u)=\frac{(1+u)^m-1}{u}.
\]
With $\phi_m(0)=m$, we have
$\phi_m(u)=m\int_0^1(1+su)^{m-1}\dd s$, which is smooth in
$(u,m)\in[0,1]\times I$.  Every derivative in
$y$ or $m$ of $e^{-(2-m)y}$ or $e^{-my}$ is a polynomial in $y$ times the
same exponential, and $y^ke^{-(\min I)y}$ is bounded on $[0,\infty)$ for
each $k$; since $2-m\geq1\geq m\geq\min I$, the claim follows.

Since $\mathsf H_V(e^{\pm m\,\cdot})(x)=e^{m^2V/2\pm mx}$, we have
\[
 \mathsf H_V(\cosh^m)(x)
 =2^{1-m}e^{m^2V/2}\cosh(mx)\bigl(1+\zeta_{m,V}(x)\bigr),
\]
where
$\zeta_{m,V}(x)=2^{m-1}e^{-m^2V/2}\sech(mx)\,\mathsf H_V\varrho_m(x)$.
Since $\partial_V\mathsf H_V=\frac12\mathsf H_V\partial_y^2$ and
$\partial_x\mathsf H_V=\mathsf H_V\partial_y$, we have
\[
 \partial_x^i\partial_m^j\partial_V^k\mathsf H_V\varrho_m
 =2^{-k}\mathsf H_V\bigl(\partial_y^{i+2k}\partial_m^j\varrho_m\bigr),
\]
which is bounded uniformly, because $\mathsf H_V$ contracts the supremum
norm.  Next, every mixed $(x,m,V)$-derivative of order at most $\ell$ of
$2^{m-1}e^{-m^2V/2}\sech(mx)$ equals
$e^{-m^2V/2}\sech(mx)$ times a polynomial in $V$ and $x$ of degree at most
$\ell$ in each variable, with coefficients that are bounded functions of
$(x,m)$: each $m$-derivative of $e^{-m^2V/2}$ produces at most one factor
$V$, and each derivative of $\sech(mx)$ produces bounded multiples of
$\tanh(mx)$, $\sech^2(mx)$, and at most one factor $x$.  Since
$|x|^n\sech(mx)\leq C_n(\min I)^{-n}$, the powers of $x$ are absorbed, and,
by the Leibniz rule, every derivative of $\zeta_{m,V}$ of total order at
most $\ell$ is bounded by $C_\ell V^\ell e^{-m^2V/2}$ for $V\geq1$.  Fix
$\overline V\geq1$ such that $|\zeta_{m,V}|\leq\frac12$ for
$V\geq\overline V$.  Since
$T_{m,V}\log\cosh=m^{-1}\log\mathsf H_V(\cosh^m)$, we take logarithms and
obtain
\eqref{eq:large-V} with $\mathcal E_{m,V}=m^{-1}\log(1+\zeta_{m,V})$.  By
the chain rule, every derivative of $\mathcal E_{m,V}$ of total order at
most $\ell$ is a sum of products of derivatives of $\zeta_{m,V}$ whose orders
add up to at most $\ell$, multiplied by bounded factors, which proves
\eqref{eq:large-V-error}.
\end{proof}

The next lemma compares the derivative of $T_{m,V}\log\cosh$ with that of
$\psi_m$.  The derivative of $v=T_{m,V}\log\cosh$ equals $\tanh$ at $V=0$
and converges to $\tanh(m\,\cdot)$ as $V\to\infty$, by
\zcref{lem:large-V}.  The lemma says that $|v'(y)|$
never falls below its limiting value $\tanh(m|y|)$.

\begin{lemma}\label{lem:tanh-comparison}
Let $0<m\leq1$ and $V\geq0$, and let $v=T_{m,V}\log\cosh$.  Then
$|v'(y)|\geq\tanh(m|y|)$ for all $y\in\R$; in particular,
$1-v'(y)^2\leq\sech^2(my)$.
\end{lemma}

\begin{proof}
For $V=0$ this is $\tanh|y|\geq\tanh(m|y|)$.  Let $V>0$.  Since $v$ is
even, it suffices to take $y\geq0$.  We have
$mv=\log\mathsf H_V(\cosh^m)$ and
$\log\mathsf H_V(\cosh(m\,\cdot))(y)=\frac{m^2V}2+\log\cosh(my)$.  Then
$m(v'(y)-\tanh(my))$ is the derivative of $\log\Phi$, where
\[
 \Phi(y)=\frac{\mathsf H_V(\cosh^m)(y)}{\mathsf H_V(\cosh(m\,\cdot))(y)}
 =\int\omega\dd\pi_y,
 \qquad
 \omega(w)=\frac{\cosh^mw}{\cosh(mw)},
\]
and $\pi_y$ is the probability measure with
$\pi_y(\dd w)\propto\mathsf p_V(y,w)\cosh(mw)\dd w$.  Since
$\partial_y\log\mathsf p_V(y,w)=(w-y)/V$, we differentiate in $y$ and find
that $\Phi'(y)=V^{-1}\operatorname{Cov}_{\pi_y}(\omega(w),w)$.  Under $\pi_y$,
conditionally on $|w|=s$, the sign of $w$ is positive with probability
$e^{sy/V}/(2\cosh(sy/V))$, because
$\mathsf p_V(y,s)/\mathsf p_V(y,-s)=e^{2sy/V}$.  Then
$\E_{\pi_y}[w\mid|w|=s]=s\tanh(sy/V)$, and, since $\omega$ is even, we have
\[
 \Phi'(y)
 =V^{-1}\operatorname{Cov}_{\pi_y}\bigl(\omega(|w|),\,|w|\tanh(|w|y/V)\bigr).
\]
Both $s\mapsto\omega(s)$ and $s\mapsto s\tanh(sy/V)$ are nondecreasing on
$[0,\infty)$, the former because $(\log\omega)'(s)=m(\tanh s-\tanh(ms))\geq0$
for $m\leq1$.  Chebyshev's association inequality states that, for every
real random variable $S$ and nondecreasing functions $\varphi_1$ and
$\varphi_2$ for which the covariance is defined,
$\operatorname{Cov}(\varphi_1(S),\varphi_2(S))\geq0$.  Applied with $S=|w|$
under $\pi_y$, it gives $\Phi'\geq0$ on $[0,\infty)$.  Then
$v'(y)\geq\tanh(my)$ for $y\geq0$.  Squaring
$|v'(y)|\geq\tanh(m|y|)\geq0$ gives $1-v'(y)^2\leq\sech^2(my)$.
\end{proof}

Our final lemma describes the optimal diffusion across a Cole--Hopf step
with level $m$ and terminal function $\log\cosh$.  By
\zcref{lem:one-step}(iii), its law after variance $t$ is the kernel
$K_{m,t,v_t}$ of \eqref{eq:finite-transition-kernel}, the Doob
$h$-transform of Brownian motion by the space-time harmonic function
$e^{mv_t(y)}=\mathsf H_{W-t}(\cosh^m)(y)$.  Replacing $\cosh^m$ by
$\cosh(m\,\cdot)$, which is comparable to $\cosh^m$ up to constant
factors, gives the comparison kernel $K^\infty_t$, the Doob transform by
$e^{-m^2t/2}\cosh(my)$.  By Girsanov's theorem, as in
\zcref{lem:one-step}(iii), this kernel is the law at time $t$ of the
diffusion with drift $m\tanh(mX)$, and by \eqref{eq:Kinfty-mixture} it is
an explicit mixture of two Gaussians that move away from the starting
point in opposite directions with speed $m$.

\begin{lemma}\label{lem:doob-kernel}
Let $0<m\leq1$ and $W>0$, and set $v_t=T_{m,W-t}\log\cosh$ for
$0\leq t\leq W$.  For $t\in(0,W]$ and $x\in\R$,
\begin{equation}\label{eq:exact-Doob}
 K_{m,t,v_t}(x,\dd y)
 =\mathsf p_t(x,y)\,\frac{e^{mv_t(y)}}{e^{mv_0(x)}}\dd y.
\end{equation}
Let
\begin{equation}\label{eq:Kinfty-mixture}
 K^\infty_t(x,\dd y)
 =\mathsf p_t(x,y)\,e^{-m^2t/2}\,\frac{\cosh(my)}{\cosh(mx)}\dd y
 =\frac{e^{mx}\mathsf p_t(x+mt,y)+e^{-mx}\mathsf p_t(x-mt,y)}{2\cosh(mx)}\dd y.
\end{equation}
Then, for $t\in(0,W]$ and $x\in\R$,
\begin{equation}\label{eq:Doob-comparison}
 \frac14K^\infty_t(x,\cdot)\leq K_{m,t,v_t}(x,\cdot)\leq4K^\infty_t(x,\cdot).
\end{equation}
\end{lemma}

\begin{proof}
By \zcref{lem:one-step}(i), we have
$T_{m,t}v_t=T_{m,t}T_{m,W-t}\log\cosh=v_0$.  Then \eqref{eq:exact-Doob} is \eqref{eq:finite-transition-kernel} with $a=m$,
$V=t$, and $g=v_t$.  The second expression in \eqref{eq:Kinfty-mixture}
follows from $\cosh(my)=\frac12(e^{my}+e^{-my})$ and the identity
$\mathsf p_t(x,y)e^{\pm my-m^2t/2}=e^{\pm mx}\mathsf p_t(x\pm mt,y)$, which
is obtained by completing the square.

Since $\frac12e^{|y|}\leq\cosh y\leq e^{|y|}$ and
$e^{m|y|}\leq e^{my}+e^{-my}\leq2e^{m|y|}$, we have
$2^{-m-1}(e^{my}+e^{-my})\leq\cosh^my\leq e^{my}+e^{-my}$.  Applying
$\mathsf H_s$ gives
\[
 2^{-m}e^{m^2s/2}\cosh(mx)\leq\mathsf H_s(\cosh^m)(x)\leq2e^{m^2s/2}\cosh(mx)
 \qquad(s\geq0).
\]
Since $e^{mv_t}=\mathsf H_{W-t}(\cosh^m)$ for $0\leq t\leq W$, inserting
these bounds with
$s=W-t$ in the numerator and $s=W$ in the denominator of
\eqref{eq:exact-Doob} gives \eqref{eq:Doob-comparison}, because
$2^{m+1}\leq4$.
\end{proof}

\bibliographystyle{amsplain}
\bibliography{references}

\end{document}